\documentclass[oneside, a4paper,11pt,reqno]{amsart}
\usepackage{yfonts}

\RequirePackage[colorlinks,citecolor=blue,urlcolor=blue]{hyperref}

\makeatletter
\def\timenow{\@tempcnta\time
  \@tempcntb\@tempcnta
  \divide\@tempcntb60
  \ifnum10>\@tempcntb0\fi\number\@tempcntb
  \multiply\@tempcntb60
  \advance\@tempcnta-\@tempcntb
  :\ifnum10>\@tempcnta0\fi\number\@tempcnta}
\makeatother

\usepackage[ddmmyyyy,hhmmss]{datetime}
\usepackage{color}

\usepackage{euscript}

\usepackage[applemac]{inputenc}
\usepackage{amssymb,amsmath,amsthm,dsfont}
\usepackage{graphics,graphicx}
\usepackage[T1]{fontenc}
\usepackage{color}
\usepackage{epsfig,psfrag}

\newtheorem{theo}{Theorem}[section]
\newtheorem{prop}[theo]{Proposition}
\newtheorem{lemme}[theo]{Lemma}

\newtheorem{cor}[theo]{Corollary}

\newtheorem{remarque}[theo]{Remark}

\def\tA2{\tilde{A}({\tilde L_2})}
\def \tts2{\tilde \tau^+_2(  h_t/2 ))}

\title{Exponential functionals of spectrally one-sided L\'evy processes conditioned to stay positive}

\author{Gr\'egoire V\'echambre}
\address{NYU Shanghai, Office 1133, 1555 Century avenue, Pudong New Area, Shanghai 200122, China}

\email{gregoire.vechambre@ens-rennes.fr}


\date{\today\ \`a \currenttime}

\begin{document} 

\maketitle

\begin{abstract}
We study the properties of the exponential functional $\int_0^{+ \infty} e^{- X^{\uparrow} (t)}dt$ where $X^{\uparrow}$ is a spectrally one-sided L\'evy process conditioned to stay positive. In particular, we study finiteness, self-decomposability, existence of finite exponential moments, asymptotic tail at $0$ and smoothness of the density. 
\end{abstract}


\pagestyle{myheadings}
\markboth{Right}{Exponential functionals of spectrally one-sided L\'evy processes conditioned to stay positive}

\section{Introduction}

We consider a spectrally negative L\'evy process $V$ that is not the opposite of a subordinator. We denote its Laplace exponent by $\Psi_V$: 
\[ \forall t, \lambda \geq 0, \ \mathbb{E} \left [ e^{\lambda V(t)} \right ] = e^{t \Psi_V(\lambda)}. \]
In the case where $V$ drifts to $-\infty$, it is well known that its Laplace exponent admits a non trivial zero that we denote here by $\kappa$, $\kappa := \inf \{ \lambda > 0, \ \Psi_V(\lambda) = 0 \}$. If $V$ does not drift to $-\infty$, then $0$ is the only zero of $\Psi_V$ on $[0, +\infty[$ so we put $\kappa := 0$ in this case. We denote by $(Q, \gamma, \nu)$ the generating triplet of $V$. According to the L\'evy-Kintchine formula, $\Psi_V$ can be expressed as
\begin{eqnarray}
\Psi_V(\lambda) = \frac{Q}{2} \lambda^2 - \gamma \lambda + \int_{-\infty}^0 (e^{\lambda x} - 1 - \lambda x \mathds{1}_{|x| < 1}) \nu(dx). \label{levkhin}
\end{eqnarray}
We also consider $Z$, a spectrally positive L\'evy process drifting to $+\infty$ with unbounded variation. 

We are interested in the exponential functionals of $V$ and $Z$ conditioned to stay positive, 
\[ I(V^{\uparrow}) := \int_0^{+ \infty} e^{- V^{\uparrow} (t)}dt \ \ \ \text{and} \ \ \ I(Z^{\uparrow}) := \int_0^{+ \infty} e^{- Z^{\uparrow} (t)}dt. \]
For both we study finiteness, exponential moments, and the asymptotic tail at $0$. For $I(V^{\uparrow})$, we use its property of self-decomposability to get precise estimates on the asymptotic tail at $0$ and properties of the density. The exponential functional $I(V^{\uparrow})$ has already appeared in the study of self-similar Markov processes with no positive jumps by Bertoin, Caballero \cite{bertoincaballero} and Pardo \cite{Pardo2009}. In both those articles is proved the self-decomposability of the functional $I(V^{\uparrow})$. Moreover, in \cite{bertoincaballero}, the authors also compute the moments of small positive order of this functional. In \cite{Pardo2009}, the asymptotic almost sure behavior at $0$ and $+\infty$ of a positive self-similar Markov process is linked to the left tail of $I(V^{\uparrow})$. As a consequence, it seems that the study of the left tail of $I(V^{\uparrow})$ in the present paper can be put in relation with the results of \cite{Pardo2009} to understand the almost sure behavior of positive self-similar Markov processes with no positive jumps. 

Our first motivation here is to extend to spectrally one-sided L\'evy processes conditioned to stay positive the general study of exponential functionals of L\'evy processes. Those functionals have been widely studied because of their importance in probability theory. For example they are fundamental in the study of diffusions in random environments and appear in many applications such as the study of self-similar Markov processes and mathematical finance, see Bertoin, Yor \cite{Bertoinyor} and Pardo, Rivero \cite{pardosurvey} for surveys on those functionals and their applications. For a general L\'evy process, equivalent conditions for the finiteness of the exponential functional are given in \cite{Bertoinyor}, the asymptotic tail at $+ \infty$ of the functional is studied in \cite{rivero2005}, \cite{Maulik2006156}, \cite{10.230725464921}, \cite{rivero2012} (see also \cite{Bertoinyor}, \cite{pardosurvey}), the absolute continuity is proved in \cite{blr}, \cite{pardo2013}, and properties of the density (such as regularity) are studied in \cite{MR1648657}, \cite{pardo2013}, \cite{Patie2013393}. Recently, factorization identities for exponential functionals of L\'evy processes have been proved in \cite{ref6patie}, \cite{patieref8}, \cite{Patie2013393} (see also \cite{pardosurvey}). They are helpful for the study of some properties of exponential functionals such as unimodality, smoothness of the density, complete monotonicity of the density, and they allow to derive {explicit} representations of the density in some special cases. Moreover, these factorization identities give a central place to the study of exponential functionals of subordinators and of spectrally negative L\'evy processes, since the latter two are the two typical factors in which a general exponential functional is decomposed. A fruitful approach in the study of exponential functionals of L\'evy processes is to consider their Mellin transform that satisfies a functional equation, see \cite{MR1648657}, \cite{Maulik2006156}, \cite{Kuznetsov2013}. The solution of this equation in term of a generalized Weierstrass product is given in \cite{Patie2013393}. The equation on the Mellin transform allows to compute the entire moments of exponential functionals in some cases, see \cite{Bertoinyor}. The relation between self-similar Markov processes and exponential functionals of L\'evy processes can also be used to study the latter, see \cite{pierre2009}, \cite{patie2012}. {Let us also mention that recently, advancements have been made by Arista, Rivero \cite{aristarivero} in the study of exponential functionals: they establish a functional relation satisfied by the density of exponential functionals which allows to retrieve many properties of the functionals (such as the above mentioned factorization identities and the behavior of the asymptotic tails) and to establish new ones.} In this paper we highlight and exploit some links between exponential functionals of L\'evy processes and exponential functionals of L\'evy processes conditioned to stay positive. 
Also, as a by-product of our approach, we retrieve some known results about the exponential functionals of spectrally negative L\'evy processes. 

Our second motivation is the possibility to apply our results to the study of diffusions in spectrally negative L\'evy environments. Such processes, introduced by Brox \cite{Brox} when the environment is given by a Brownian motion, have been specifically studied for the spectrally negative L\'evy case by Singh \cite{Singh}. In \cite{advech}, Andreoletti, Devulder and the author prove that the supremum of the local time, $\mathcal{L}_X^*$, of a diffusion in a drifted Brownian environment converges in distribution. They express the limit law in term of a bidimensional stable subordinator that depends on an exponential functional of the environment conditioned to stay positive. In order to generalize their results to a diffusion in a spectrally negative L\'evy environment, knowledge on the exponential functionals involved is needed. These are precisely exponential functionals of the environment (which is spectrally negative) and its dual (which is spectrally positive) conditioned to stay positive. 

Finally, we have hints that the almost sure asymptotic behavior of $\mathcal{L}_X^*$, for a diffusion in the spectrally negative L\'evy environment $V$, is crucially linked to the right and left tails of the distribution of $I(V^{\uparrow})$. This is why we study these tails here and give for the left tail a precise asymptotic estimate when it is possible, in particular when $\Psi_V$ is regularly varying. For the right tail, we are mainly interested in the existence of some finite exponential moments. The application of the present work to diffusions in random environments is a work in preparation by the author \cite{caslevyvech}, \cite{psvech}. 


For a process $A$ and a Borel set $S$ we denote
\[ \tau(A, S) := \inf \left \{ t \geq 0, \ A(t) \in S \right \}, \ \ \ \mathcal{R}(A, S) := \sup \left \{ t \geq 0, \ A(t) \in S \right \}, \]
where by convention $\inf \emptyset = \infty$. We shall only write $\tau(A, x)$ (respectively $\mathcal{R}(A, x)$) instead of $\tau(A, \{x\})$ (respectively $\mathcal{R}(A, \{x\})$) and $\tau(A, x+)$ instead of $\tau(A, [x, +\infty [)$. For example, since $V^{\uparrow}$ has no positive jumps, we see that it reaches each positive level continuously: $\forall x > 0, \ \tau(V^{\uparrow}, x+) = \tau(V^{\uparrow}, x)$ and since moreover $V^{\uparrow}$ converges to $+\infty$ we have $\forall x > 0, \ \mathcal{R}(V^{\uparrow}, [0, x]) = \mathcal{R}(V^{\uparrow}, x)$. 

If $A$ is Markovian and $x \in \mathbb{R}$ we denote $A_x$ for the process $A$ starting from $x$. For $A_0$ we shall only write $A$. 
For any (possibly random) time $T > 0$, we write $A^T$ for the process $A$ shifted and centered at time $T$: $\forall s \geq 0, \ A^T(s) := A(T+s)-A(T)$. 

Let $W$ be the scale function of $V$, defined as in Section VII.2 of Bertoin \cite{Bertoin}. It satisfies 
\[ \forall 0 < x < y, \ \mathbb{P} \left ( \tau(V_x,y) < \tau(V_x,]-\infty, 0]) \right ) = W (x) / W (y). \]
According to Theorem VII.8 in \cite{Bertoin}, this function is continuous, increasing, and for any $\lambda > \kappa$, 
\[ \int_{0}^{+\infty} e^{-\lambda x} W(x) dx = \frac{1}{\Psi_V(\lambda)} < +\infty. \]

We now explain how the scale function $W$ allows to define $V^{\uparrow}$, that is, $V$ conditioned to stay positive (see \cite{Bertoin}, Section VII.3 for proofs and more details). The relation 
\[ \forall t \geq 0, \ x, y > 0, \ p^{\uparrow}_t(x, dy) := W(y) \mathbb{P} \left ( V_x(t) \in dy, \ \inf_{[0, t]} V_x > 0 \right ) / W(x), \]
defines a Markovian semigroup. 
For any $x > 0$, we denote by $V^{\uparrow}_x$ the Markovian process starting from $x$, taking values in $]0, +\infty[$ and whose transition semigroup is $p^{\uparrow}_t$. $V^{\uparrow}_x$ is commonly called $V$ \textit{conditioned to stay positive starting from $x$}. This is justified by the fact that for any $ y> x > 0$ the process $( V_x^{\uparrow}(t), \ 0 \leq t \leq \tau(V_x^{\uparrow}, y))$ is equal in distribution to the process $( V_x(t), \ 0 \leq t \leq \tau(V_x, y))$ conditionally on $\{ \tau(V_x, y) < \tau(V_x, ]- \infty, 0]) \}$. 
In the case where $V$ drifts to $+\infty$ this relation remains true after replacing the hitting time of $y$ by $+\infty$. In the cases where $V$ oscillates or drifts to $-\infty$, it is not possible to condition in the usual {sense} $V_x$ to remain positive for ever, because this occurs with a probability equal to $0$. However, even in these cases, $V^{\uparrow}_x$ has infinite life-time. It is moreover possible to show that in all cases $V^{\uparrow}_x$ converges almost surely to $+\infty$. 
\medbreak
It is possible to show that there exists a process $V^{\uparrow}_0$, that we denote by $V^{\uparrow}$, starting from $0$ and which is the limit in distribution of the processes $V^{\uparrow}_x$ as $x$ goes to $0$. $V^{\uparrow}$ is called $V$ \textit{conditioned to stay positive}, it is a Feller process whose restriction to $]0, +\infty[$ of the transition laws is given by $p^{\uparrow}_t$. Moreover, note that for any positive $x$, we have from the Markov property and the absence of positive jumps that the process $V^{\uparrow}$, shifted at $\tau(V^{\uparrow}, x)$, is equal in law to $V^{\uparrow}_x$.

In the case where $V$ drifts to $-\infty$, we recall the definition of $V$ conditioned to drift to $+\infty$ (see \cite{Bertoin}, Section VII.1 for proofs and more details). Note that, since $\kappa$ is a zero of $\Psi_V$, the process $e^{\kappa V(t)}$ is a martingale with expectation $1$. This allows to define the convolution semigroup of probability measures $(p^{\sharp}_t)_{t \geq 0}$ via the relation 
\[ p^{\sharp}_t (dy) := e^{\kappa y} \mathbb{P} \left ( V(t) \in dy \right ). \]
The convolution semigroup $(p^{\sharp}_t)_{t \geq 0}$ is associated to a spectrally negative L\'evy process that we denote by $V^{\sharp}$ and which is commonly called \textit{$V$ conditioned to drift to $+\infty$}. This is justified by the fact that 
\begin{itemize}
\item $V^{\sharp}$ drifts to $+\infty$, 
\item For any $y > 0$ the process $( V^{\sharp}(t), \ 0 \leq t \leq \tau(V^{\sharp}, y))$ is equal in distribution to the process $( V(t), \ 0 \leq t \leq \tau(V, y))$ conditionally on $\{ \tau(V, y) < +\infty \}$.  
\end{itemize}
It is not difficult to see that the Laplace exponent $\Psi_{V^{\sharp}}$ of $V^{\sharp}$ is obtained by translation of $\Psi_V$: $\Psi_{V^{\sharp}} = \Psi_V(\kappa + .)$. 
As a consequence $\Psi'_{V^{\sharp}}(0) > 0$ which, according to Corollary VII.2 in \cite{Bertoin}, proves that $V^{\sharp}$ drifts to $+\infty$ as stated above. It is also proven that $V^{\uparrow} = (V^{\sharp})^{\uparrow}$ and this is why, for convenience, we often work with $V^{\sharp}$ instead of $V$. 

In order to do our proofs in a systematic way, we define $V^{\sharp}$ to be \textit{$V$ conditioned to drift to $+\infty$} in the case where $V$ drifts to $-\infty$ and \textit{only $V$} in the other cases (when $V$ oscillates or drifts to $+\infty$). As a consequence, $V^{\sharp}$ always denotes a spectrally negative L\'evy process that does not drifts to $-\infty$ (it oscillates if $V$ does and it drifts to $+\infty$ if $V$ drifts to $+\infty$ or $-\infty$) and that satisfies $\Psi_{V^{\sharp}} = \Psi_V(\kappa + .)$. In any case we have that for all $0 < x < y$, $(V^{\uparrow}_x(t), \ 0 \leq t \leq \tau(V^{\uparrow}_x, y))$ is equal in law to $(V^{\sharp}_x(t), \ 0 \leq t \leq \tau(V^{\sharp}_x, y))$ conditionally on $\{ \tau(V^{\sharp}_x, y) < \tau(V^{\sharp}_x, ]-\infty, 0]) \}$. 
Note that the same identity is true with $V$ instead of $V^{\sharp}$, but the advantage of dealing with $V^{\sharp}$ is that $\tau(V^{\sharp}_x, y)$ is always finite (while $\tau(V_x, y)$ can possibly be infinite) which simplifies the argumentation. 

Let us now recall some facts about $Z^{\uparrow}$, that is, $Z$ conditioned to stay positive. By assuming unbounded variation for $Z$ we excluded the case where $Z$ is a subordinator. In this case, $Z$ would stay positive and $I(Z^{\uparrow})$ would be only $I(Z)$ which has already been widely studied. It is already known to be finite and have some finite exponential moments (see for example Theorem 2 in \cite{Bertoinyor}), so Theorem \ref{vneglapl} below is already known in the subordinator case. 
Since, in our case, $-Z$ is spectrally negative and not the opposite of a subordinator (we denote by $\kappa_Z$ the non-trivial zero of $\Psi_{-Z}$, the Laplace exponent of $-Z$), it is regular for $]0, +\infty[$ according to Theorem VII.1 in \cite{Bertoin}, so $Z$ is regular for $]-\infty, 0[$. Moreover, $Z$ drifts to $+\infty$. We can thus define the Markov family $( Z^{\uparrow}_x, \ x > 0 )$ as in Doney \cite{Doney}, Chapter 8. It can be seen from there that for any $x > 0$ the process $Z^{\uparrow}_x$ is Markovian and has infinite life-time (here is used the hypothesis that $Z$ drifts to $+\infty$). Then, note from Theorem 25 in \cite{Doney} that $Z^{\uparrow}_0$, that we denote by $Z^{\uparrow}$, is well defined if and only if $Z$ is regular for $]0, +\infty[$. According to Corollary VII.5 in \cite{Bertoin} applied to the dual of $Z$, the latter is equivalent to the fact that $Z$ has unbounded variation. This is why we have made this assumption. 

\subsection{Results}

In the special case of the exponential functional of a drifted Brownian motion conditioned to stay positive, all the properties that are studied here are already known. We discuss this case in the next subsection. 

Our first result is the finiteness of $I(V^{\uparrow})$ and the fact that it admits exponential moments. 

\begin{theo} \label{finiteandexpomoments}

The random variable $I(V^{\uparrow})$ is almost surely finite 
and 
\begin{eqnarray}
\forall \lambda < \Psi_V(\kappa + 1), \ \mathbb{E} \left [ e^{\lambda I(V^{\uparrow})} \right ] < + \infty. \label{momentsexpo}
\end{eqnarray}
Moreover, there are two positive constants $K_1 > K_2 > 0$ such that for all $x$ large enough 
\begin{eqnarray}
\exp \left ( -K_1 x \right ) \leq \mathbb{P} \left ( I(V^{\uparrow}) \geq x \right ) \leq \exp \left ( -K_2 x \right ). \label{queueexpo}
\end{eqnarray}
Note that from \eqref{momentsexpo}, $K_2$ in \eqref{queueexpo} can be chosen as close as we want to $\Psi_V(\kappa + 1)$. 
\end{theo}

If $V$ drifts to $+\infty$, is non-arithmetic, is such that the Laplace transform of $V(1)$ extends on the negative half-line, and if there is $\alpha > 0$ such that $\Psi_V(-\alpha) = 0$ and $\Psi_V'(-\alpha) \in ]-\infty, 0[$ then, according to the results of \cite{rivero2005} and \cite{Maulik2006156} (see also \cite{Bertoinyor}, \cite{pardosurvey}, \cite{aristarivero}), we have that the right tail of $I(V)$ is of order $x^{-\alpha}$. If the L\'evy measure of $V$ is subexponential or convolution equivalent, then the right tail of $I(V)$ is given by the results of \cite{Maulik2006156} and \cite{rivero2012} (see also \cite{pardosurvey}, \cite{aristarivero}). {In any case the right tail of $I(V)$ is at least inverse polynomial, this is guaranteed, not only for $V$ spectrally negative but for a general $V$ that is not a subordinator, by Theorem 2.14(1) of \cite{bersteingamma} (under some assumptions, that theorem also provides the asymptotic of the density of the functional and of its derivatives).} Therefore the existence of finite exponential moments for $I(V^{\uparrow})$ given by Theorem \ref{finiteandexpomoments} contrasts with the right tail known for $I(V)$. This difference between the two tails comes from the fact that $V^{\uparrow}$ stays positive so the right tail of $I(V^{\uparrow})$ is determined by the right tail of the local time of $V^{\uparrow}$ at positive levels, which is exponential by the Markov property. To the contrary, the absolute minimum of $V$ is negative and has a strong influence on the right tail of $I(V)$ (see for example \cite{rivero2012} {and Theorem 4 in \cite{aristarivero}} for {comparisons} between the right tail of the exponential functional and the right tail of the absolute minimum). 

%
%
%
%

Then, a fundamental point of our study is Proposition \ref{decomposition} which says that for any positive $y$, $I(V^{\uparrow})$ satisfies the random affine equation 
\begin{eqnarray}
I(V^{\uparrow}) \overset{\mathcal{L}}{=} A^y + e^{-y} I(V^{\uparrow}), \label{rae}
\end{eqnarray}
where $A^y$ is independent of the second term and will be specified later. From \eqref{rae} we can notice that $I(V^{\uparrow})$ is a positive self-decomposable random variable and is therefore infinitely divisible, absolutely continuous and unimodal. It is well known that the exponential functional $I(V)$ of a spectrally negative L\'evy process $V$ is also self-decomposable (as long as it is finite). It can be seen by splitting the trajectory at $\tau(V, y)$, the first passage time at $y$ (see also 
Remark 1 in Rivero \cite{rivero2012} for an explicit representation of $I(V)$ as the stationary distribution of an Ornstein-Uhlenbeck type process, and Section 5 in Patie, Savov \cite{Patieref10} where $I(V)$ is characterized via its L\'evy triplet). 
As we said, the right tail of $I(V)$ (when the latter is finite) never has exponential moments while, by Theorem \ref{finiteandexpomoments}, the right tail of $I(V^{\uparrow})$ always has. It follows that the class of self-decomposable random variables that can be represented by $I(V^{\uparrow})$, for some spectrally negative L\'evy process $V$, is disjoint from the class of self-decomposable random variables that can be represented by $I(V)$, for some spectrally negative L\'evy process $V$. However, the following results (especially Proposition \ref{vposqueue0decond}) show that the asymptotic tails at $0$ of these two classes of self-decomposable random variables are comparable. 
Another consequence of \eqref{rae} is that for any positive $y$, $I(V^{\uparrow})$ can be written as the random series
\[ I(V^{\uparrow}) \overset{\mathcal{L}}{=} \sum_{k \geq 0} e^{-ky} A^y_k, \]
where the random variables $A_k^y$ are \textit{iid} and have the same law as $A^y$. The novelty of our approach is mainly to study the exponential functional $I(V^{\uparrow})$ via the random coefficient $A^y$ of which we decompose the law into two convolution factors that are accessible via their Laplace transforms. This decomposition is thus a useful tool for the study of $I(V^{\uparrow})$ and is the base of the proofs of many results we present below. 

Our next results link the asymptotic behavior of $\Psi_{V}$ with the left tail of $I(V^{\uparrow})$. 

\begin{theo} \label{queuefonctexpo}

Assume that there is $\alpha > 1$ and a positive constant $C$ such that for all $\lambda$ large enough we have $\Psi_V(\lambda) \leq C \lambda^{\alpha}$. Then for all $\delta \in ]0, 1[$ and $x$ small enough we have 
\begin{eqnarray}
\mathbb{P} \left ( I(V^{\uparrow}) \leq x \right ) \leq \exp \left ( - \delta (\alpha - 1) / (C x)^{1/(\alpha - 1)} \right ). \label{queuefonctexpo1}
\end{eqnarray}

Assume that there is $\alpha > 1$ and a positive constant $c$ such that for all $\lambda$ large enough we have $\Psi_V(\lambda) \geq c \lambda^{\alpha}$. Then for all $\delta > 1$ and $x$ small enough we have 
\begin{eqnarray}
\mathbb{P} \left ( I(V^{\uparrow}) \leq x \right ) \geq \exp \left ( - \delta \alpha^{\alpha/(\alpha-1)} / (c x)^{1/(\alpha - 1)} \right ). \label{queuefonctexpo2}
\end{eqnarray}

\end{theo}

Let us now recall how is usually quantified the asymptotic behavior of $\Psi_V$. We define, as in \cite{Bertoin} page 94, 
\begin{align*} \sigma & := \sup \left \{ \alpha \geq 0, \ \lim_{\lambda \rightarrow + \infty} \lambda^{- \alpha} \Psi_{V}(\lambda) = \infty \right \}, \\
\beta & := \inf \left \{ \alpha \geq 0, \ \lim_{\lambda \rightarrow + \infty} \lambda^{- \alpha} \Psi_{V}(\lambda) = 0 \right \}. \end{align*} 
Recall that $\Psi_{V^{\sharp}} (.) = \Psi_{V}(\kappa + .)$, so $\sigma$ and $\beta$ are identical whether they are defined from $\Psi_{V^{\sharp}}$ or $\Psi_{V}$. 

If $\Psi_{V}$ has $\alpha$-regular variation for $\alpha \in [1, 2]$, for example if $V$ is a (drifted or not) $\alpha$-stable L\'evy process with no positive jumps, we have $\sigma = \beta = \alpha$. 
Recall that $Q$ is the Brownian component of $V$, it is well known that $\Psi_{V} (\lambda) / \lambda^2$ converges to $Q/2$ when $\lambda$ goes to infinity. As a consequence: when $Q > 0$, $\Psi_{V}$ has $2$-regular variation, and when $Q=0$, $1 \leq \sigma \leq  \beta \leq 2$, where $1 \leq \sigma$ comes from the convexity of $\Psi_{V}$. 

\begin{remarque} \label{casvb1}
When $V$ has bounded variation, we know (see for example \cite{Bertoin} Section I.1) that the Brownian component of $V$ is null, the L\'evy measure $\nu$ of $V$ satisfies $\int_{-1}^0 |x| \nu(dx) < +\infty$ and $\gamma^* := - \gamma - \int_{-1}^0 x \nu(dx)$, the factor of $\lambda$ in the expression of $\Psi_V(\lambda)$, is positive (otherwise $V$ would be the opposite of a subordinator). It is thus easy to see that in this case $\Psi_{V} (\lambda) / \lambda$ converges to $\gamma^*$ when $\lambda$ goes to infinity, so $\sigma = \beta = 1$. In the next theorem, we sometimes assume that $\sigma > 1$, the reader should be aware that it excludes the case where $V$ has bounded variation. However, this case is quite easy and shall be treated in the remarks. 
\end{remarque}

We are now ready to state a general result on the asymptotic tails at $0$ of $I(V^{\uparrow})$: 

\begin{theo} \label{vposcompbrownnullequeue0}
{Recall that $1 \leq \sigma \leq  \beta \leq 2$.} We have 
\begin{eqnarray}
\forall \beta' > \beta, \ \underset{x \rightarrow 0}{\lim} \ x^{1/(\beta' - 1)} \log \left ( \mathbb{P} \left ( I(V^{\uparrow}) \leq x \right ) \right ) = - \infty, \label{majogeneralenew}
\end{eqnarray}
\begin{eqnarray}
\text{if} \ \sigma > 1, \  \forall \sigma' \in ]1, \sigma[, \ \underset{x \rightarrow 0}{\lim} \ x^{1/(\sigma' - 1)} \log \left ( \mathbb{P} \left ( I(V^{\uparrow}) \leq x \right ) \right ) = 0. \label{minogeneralenew}
\end{eqnarray}

\end{theo}

Theorem \ref{vposcompbrownnullequeue0} gives for $\mathbb{P} ( I(V^{\uparrow}) \leq x )$ a lower bound involving $\sigma$ and an upper bound involving $\beta$. In the case where $\Psi_V$ is regularly varying we can naturally expect a more precise result. Let us introduce the notation $\varphi_V(x) := \inf \{ \lambda \geq 0, \ \Psi_V(\lambda + \kappa) / \lambda > x \}$, where we recall that $\kappa = 0$ in the case where $V$ drifts to $+\infty$ or oscillates. 

\begin{theo} \label{varreg}

Assume that $\Psi_V$ has $\alpha$-regular variation at $+ \infty$ for some $\alpha \in ]1, 2]$. Then, 
\[ \log \left ( \mathbb{P} \left ( I(V^{\uparrow}) \leq x \right ) \right ) \underset{x \rightarrow 0}{\sim} - (\alpha - 1) \ \varphi_V \left ( \frac1{x} \right ). \]
\end{theo} This result has the following corollary: 

\begin{cor} \label{vposstablequeue0}
We assume that there is a positive constant $C$ and $\alpha \in ]1, 2]$ such that $\Psi_V(\lambda) \sim_{\lambda \rightarrow +\infty} C \lambda^{\alpha}$, then 
\[ \log \left ( \mathbb{P} \left ( I(V^{\uparrow}) \leq x \right ) \right ) \underset{x \rightarrow 0}{\sim} -\frac{\alpha - 1}{(Cx)^{\frac1{\alpha-1}}}. \]

\end{cor}

The above result is true in particular when, for some $\alpha \in ]1, 2]$, $V$ is an $\alpha$-stable spectrally negative L\'evy process (with adjonction or not of a drift). In particular, it agrees exactly with the tail \eqref{casmb} given in the next subsection for the particular case of a drifted Brownian motion. Also, we note that Corollary \ref{vposstablequeue0} is particularly interesting for applications to diffusions in random environments. Indeed, it allows to determine the exact value of the $\limsup$ of the renormalized supremum of the local time for a diffusion in a L\'evy environment $V$ that satisfies the assumptions of the Corollary, see \cite{psvech}. 

\begin{remarque} \label{majouniverselle}
Since $\Psi_{V} (\lambda) / \lambda^2$ has always a finite limit at $+ \infty$ we get, from Theorem \ref{queuefonctexpo}, that there is always a positive constant $K$ (depending on $V$) such that for $x$ small enough
\[ \mathbb{P} \left ( I(V^{\uparrow}) \leq x \right ) \leq e^{-K/x}. \]
\end{remarque}

\begin{remarque} \label{morepreceisevb}
Note that Theorem \ref{vposcompbrownnullequeue0} holds when $\beta = 1$ and $1/(\beta' - 1)$ can then equal any number in $]0, +\infty[$. When $V$ has bounded variation, we even have a stronger result: the support of the distribution of $I(V^{\uparrow})$ is $[1/\gamma^*, +\infty[$, so $\mathbb{P} ( I(V^{\uparrow}) \leq x )$ is null for $x$ small enough. This fact is the counterpart of Proposition 5.10(2) from \cite{Patieref10}. 
\end{remarque}

\begin{remarque} \label{estimdensite}
Recall that $I(V^{\uparrow})$ is unimodal. If $0$ was a mode, then we would have $\mathbb{P} ( I(V^{\uparrow}) \leq x ) \geq cx$ for some positive constant $c$ and $x$ small enough, which is incompatible with, for example, Remark \ref{majouniverselle}. As a consequence the density of $I(V^{\uparrow})$ is non-decreasing on a neighborhood of $0$. This allows to remark that Theorems \ref{queuefonctexpo}, \ref{vposcompbrownnullequeue0}, \ref{varreg}, Corollary \ref{vposstablequeue0} and Remarks \ref{majouniverselle}, \ref{morepreceisevb} are true for the density of $I(V^{\uparrow})$ in place of the distribution function $\mathbb{P} ( I(V^{\uparrow}) \leq . )$. 
\end{remarque}

When $V$ drifts to $+\infty$ we compare the left tails of $I(V^{\uparrow})$ and $I(V)$ in the next proposition. 

\begin{prop} \label{vposqueue0decond}

If $V$ drifts to $+\infty$, there is a positive constant $c$ such that for all positive $\epsilon$ and $x$ small enough, 
\[ \mathbb{P} \left ( I(V) \leq x \right ) \leq \mathbb{P} \left ( I(V^{\uparrow}) \leq x \right ) \leq \mathbb{P} \left ( I(V) \leq (1+\epsilon)x \right ) / c \epsilon x. \]

\end{prop}

This implies that all the results we prove for the left tail of $I(V^{\uparrow})$ (Theorems \ref{queuefonctexpo}, \ref{vposcompbrownnullequeue0}, \ref{varreg}, Corollary \ref{vposstablequeue0} and Remarks \ref{majouniverselle}, \ref{morepreceisevb}) are true for $I(V)$ in place of $I(V^{\uparrow})$. In particular, the combination of Theorem \ref{varreg} with Proposition \ref{vposqueue0decond} allows to retrieve Proposition 1 of Pardo \cite{pardoonthefuture}. Reciprocally, we can combine Proposition \ref{vposqueue0decond} with the very precise result recently obtained by Patie and Savov \cite{Patieref10}, Theorem 5.24, to prove another result on the left tail of $I(V^{\uparrow})$ when $V$ does not oscillate and has unbounded variation.

\begin{prop} \label{encadrementpatiesavov}

Assume that $V$ has unbounded variation and does not oscillate. There are positive constants $K_1$ and $K_2$ such that for any positive $\delta > 1$ and $x$ small enough, 
\begin{align*}
K_1 (\delta - 1) \sqrt{\varphi_V' \left ( \frac{\delta}{x} \right )} \exp \left ( - \int_{\Psi_V'(\kappa)}^{\delta/ x} \frac{\varphi_V(r)}{r} dr \right ) & \leq \mathbb{P} \left ( I(V^{\uparrow}) \leq x \right ) \\
& \leq \frac{K_2}{(\delta-1) x} \sqrt{\varphi_V' \left ( \frac1{\delta x} \right )} \exp \left ( - \int_{\Psi_V'(\kappa)}^{1/\delta x} \frac{\varphi_V(r)}{r} dr \right ). 
\end{align*}


\end{prop}

It seems that our previous results on the left tail of $I(V^{\uparrow})$ could be proved from Proposition \ref{encadrementpatiesavov} when $V$ has unbounded variation and does not oscillate. Moreover, it seems to give a more precise estimate compared to Theorem \ref{vposcompbrownnullequeue0} in the case where 
we do not have assumptions on the asymptotic behavior of $\Psi_V$. However, since Proposition \ref{encadrementpatiesavov} follows from the results that are known for the left tail of $I(V)$, and not from a direct study of $I(V^{\uparrow})$, it dramatically excludes the case where $V$ oscillates. This is why we emphasis on the method developed for the previous theorems that treats indistinctly the three cases of the three possible behaviors for $V$. 

Considering $I(V^{\uparrow})$ from the point of view of self-decomposable random variables, we can combine our results on the left tail of $I(V^{\uparrow})$ with the fine results of Sato, Yamazato \cite{satoyamazato} to get the following Theorem. It proves in particular the smoothness of the density of $I(V^{\uparrow})$ when $V$ has unbounded variation. 
\begin{theo} \label{asin13}
Assume that $V$ has unbounded variation, then 
\begin{itemize}
\item $I(V^{\uparrow})$ is a self-decomposable random variable of type $I_6$, in the terminology of \cite{satoyamazato}, 
\item The density of $I(V^{\uparrow})$ belongs to the Schwartz space. All its derivatives converge to $0$ at $+\infty$ and $0$, 
\item The distribution function $x \mapsto \mathbb{P} (I(V^{\uparrow}) \leq x)$ is $\log$-concave on $]0, +\infty[$. 
\end{itemize}
\end{theo}

We note that, in the class $I$ of self-decomposable random variables (still in the terminology of \cite{satoyamazato}), the subclasses $I_6$ and $I_7$ are the ones for which the less is known about the left tails. 

In \cite{Patieref10}, the exponential functional $I(V)$ is also studied (when $V$ is assumed to drift to $+\infty$) as a self-decomposable random variable using the fine tools developed by \cite{satoyamazato}. More precisely, they give a representation of the L\'evy triplet of $I(V)$ which allows to determine its type as a self-decomposable random variable and the exact regularity of its density. 
In Lemma 5.14(2) of \cite{Patieref10}, they prove in particular that $I(V)$ is of type $I_6$ if and only if $V$ has non-zero Gaussian component or infinite L\'evy measure. This is also the necessary and sufficient condition under which the density of $I(V)$ is of class $C^{\infty}$, according to Proposition 5.10(3) of \cite{Patieref10}. 
In our case, we believe that the same condition is also equivalent with the fact that $I(V^{\uparrow})$ is of type $I_6$ and with the infinite differentiability its density. 
However, the proof requires some further developments in the study of $A^y$ and could be made in a subsequent study. Also, note that even though we expect to have the same regularity for $I(V^{\uparrow})$ as for $I(V)$ (when the latter is finite), the density of $I(V)$ never belongs to the Schwartz space for it is too heavy-tailed. 

In the spectrally positive case we first prove the finiteness and the existence of some finite exponential moments. We can state the result as follows: 

\begin{theo} \label{vneglapl}

The random variable $I(Z^{\uparrow})$ is almost surely finite and admits some finite exponential moments. 

\end{theo}

\begin{remarque} \label{compnoncondpecpos}
Many results are already known for the exponential functional $I(Z)$. In Patie \cite{patie2012} (see also Corollary 2.1 in \cite{ref6patie} and Corollary 1.3 in \cite{patieref8}) a decomposition of its density into a power series is given, from which follows smoothness and precise asymptotics at $+\infty$ for the density of $I(Z)$ and its derivatives. In particular, the right tail of $I(Z)$ is of order $x^{-\alpha}$ where $\alpha$ is the non-trivial zero of the Laplace exponent of $-Z$ (this fact is also a consequence of the results of \cite{rivero2005}, \cite{Maulik2006156}, see also \cite{Bertoinyor}, \cite{pardosurvey}, \cite{aristarivero}). The existence of finite exponential moments for $I(Z^{\uparrow})$ given by Theorem \ref{vneglapl} thus contrasts with the right tail known for $I(Z)$. Here again, this difference between the two tails comes from the fact that $Z^{\uparrow}$ stays positive whereas the absolute minimum of $Z$ is negative (it follows an exponential distribution with parameter $\alpha$) and has a strong influence on the right tail of $I(Z)$. 
\end{remarque}


We now provide estimates for the asymptotic tail at $0$ of $I(Z^{\uparrow})$. When $Z$ has positive jumps, they show that the tail of $I(Z^{\uparrow})$ is heavier than the one given for $I(V^{\uparrow})$ and this comes precisely from the effect of positive jumps. {Let us recall that, since $-Z$ is spectrally negative, $\tau(-Z,.)$ is a subordinator and let us denote by $\Phi_{-Z}$ its Laplace exponent: 
\[ \forall \lambda \geq 0, \ \Phi_{-Z}(\lambda) := - \log \left ( \mathbb{E} \left [ e^{-\lambda \tau \left ( -Z, 1 \right )} \right ] \right ). \]
According to Section VII in \cite{Bertoin}, $\Phi_{-Z}$ is the inverse of the restriction of $\Psi_{-Z}$ to $[\kappa_Z, +\infty[$: $\Phi_{-Z} = \Psi_{-Z |[\kappa_Z, +\infty[}^{-1}$. We can now state a first result that compares the conditioned case with the non-conditioned case. 

\begin{theo} \label{vnegqueue0}

$I(Z)$ is stochastically greater than $I(Z^{\uparrow})$. As a consequence, 
\[ \forall x > 0, \ \mathbb{P} \big ( I(Z) \leq x \big ) \leq \mathbb{P} \big ( I(Z^{\uparrow}) \leq x \big ). \]
More precisely, for any $\epsilon > 0$ there is a positive constant $C$ (that depends on the choice of $Z$ and on $\epsilon$) such that for all $x$ small enough we have 
\[ \frac{1-\epsilon}{\kappa_Z e} \Phi_Z \left (\frac{1}{x} \right ) \times \mathbb{P} \big ( I(Z) \leq x \big ) \leq \mathbb{P} \big ( I(Z^{\uparrow}) \leq x \big ) \leq C \Phi_Z \left (\frac{1}{x} \right ) \times \mathbb{P} \big ( I(Z) \leq (1+\epsilon)x \big ). \]

\end{theo}

Theorem \ref{vnegqueue0} has to be put in relation with results about left tails of exponential functionals of L\'evy processes. We now provide such results for L\'evy processes that possess positive jumps, which is the case we are, here, mostly interested in. Indeed, if the L\'evy measure of $Z$ is the zero measure then it is known, from the L\'evy-Khintchine formula, that $Z$ is a drifted Brownian motion. The exact asymptotic tail at $0$ of $I(Z^{\uparrow})$ is then given by Corollary \ref{vposstablequeue0}. Let $Y$ be a L\'evy process drifting to $+\infty$, that has positive jumps, and let $\pi$ denote its L\'evy measure (unless stated otherwise $Y$ is not assumed to be spectrally positive). We also define $\overline{\pi}(x) := \pi([x, +\infty[)$ for all $x > 0$. In the following theorem we link the left tail of the exponential functional $I(Y)$ with the L\'evy measure $\pi$.} 

\begin{theo} \label{queue0general}

There is a positive constant $C$ such that for all $x$ small enough, 
\begin{eqnarray}
\mathbb{P} \big ( I(Y) \leq x \big ) \geq C x \overline{\pi}(\log(1/x)). \label{suppnb}
\end{eqnarray}
Another bound is required, for example in the case where the support of $\pi$ is bounded from above or in the case where $\overline{\pi}$ decreases too rapidly. Denote by $S$ the supremum of the support of $\pi( . \cap ]0, +\infty[)$ (recall that this measure is assumed to be non-zero) and note that $S = +\infty$ if this support is not bounded from above. We fix any $c > 1/2S$ (where by convention, $1/2 \infty = 0$). Then, for all $x$ small enough we have 
\begin{eqnarray}
\mathbb{P} \big ( I(Y) \leq x \big ) \geq e^{-c (\log(x))^2}.  \label{suppborne}
\end{eqnarray}
In particular, if $Y=\alpha N$ where $N$ is a standard Poisson process and $\alpha > 0$, then we have 
\begin{eqnarray}
-\log \left ( \mathbb{P} \left ( I(\alpha N) \leq x \right ) \right ) \underset{x \rightarrow 0}{\sim} \frac1{2\alpha} \left ( \log (x) \right )^2. \label{caspoissonth}
\end{eqnarray}
Note that the tail in \eqref{caspoissonth} does not depend on the parameter of the standard Poisson process. 
\end{theo}

{The estimates given in the above theorem have to be put in relation with Theorem 2.19 in \cite{bersteingamma} which, in our case (the case of an unkilled L\'evy process $Y$), ensures that $\mathbb{P} ( I(Y) \leq x ) = o(x)$.} It is interesting to put in relation Remark \ref{majouniverselle} and the combination Theorem \ref{queue0general}-Theorem \ref{vnegqueue0}. We see that in the presence of positive jumps, the asymptotic tail at $0$ of $I(Z^{\uparrow})$ (or $I(Z)$) is always at least of order $e^{-c (\log(x))^2}$ for some constant $c$. To the contrary, in the absence of positive jumps, the asymptotic tail at $0$ of the exponential functional is always at most of order $e^{-c/x}$ for some constant $c$. The existence of jumps thus plays a determinant role for the asymptotic tail at $0$ of an exponential functional. 

Also, note that \eqref{suppborne} gives exactly the minimal possible order of $\mathbb{P} ( I(Y) \leq x )$, for $Y$ lying among L\'evy processes drifting to $+\infty$ with positive jumps. Indeed, \eqref{suppborne} is always satisfied for such $Y$ and \eqref{caspoissonth} shows that this bound is optimal for the exponential functional of a multiple of a standard Poisson process. 


When $\overline{\pi}$ does not decrease too rapidly then \eqref{suppnb} is preferable to \eqref{suppborne}. Actually, we can improve \eqref{suppnb} in the spectrally positive case under a regularity hypothesis on $\pi$. 
Let us introduce the following assumption on $\pi$ that is made in \cite{pardo2013} (to study the asymptotic tails of exponential functionals of subordinators) {and in \cite{aristarivero}}: 
\[ (A) : \ \ \ \exists \alpha > 0 \ \text{s.t.} \ \forall y \in \mathbb{R}, \ \frac{\overline{\pi}(x+y)}{\overline{\pi}(x)} \underset{x \rightarrow + \infty}{\longrightarrow} e^{-\alpha y}. \]

Our last result refines \eqref{suppnb} when $Y$ is spectrally positive and $\pi$ satisfies $(A)$. It complements the study \cite{patie2012} of densities of exponential functionals of spectrally positive L\'evy processes and we prove it as a consequence of the factorization identity in \cite{ref6patie} together with a result of \cite{pardo2013} about subordinators. 

\begin{prop} \label{vnegprecisequeue0}

Assume that $Y$ is spectrally positive and $(A)$ is satisfied for some $\alpha > 0$. We denote by $m_Y$ the density of $I(Y)$. {Then $\mathbb{E}[I(Y)^{-\alpha}] < +\infty$ and 
\[ m_Y(x)  \underset{x \rightarrow 0}{\sim} \mathbb{E}[I(Y)^{-\alpha}] \ \overline{\pi}(\log (1/x)). \]}
\end{prop}

{\begin{remarque}
In the case of a spectrally positive $Y$, the result of Proposition \ref{vnegprecisequeue0} together with some easy calculations on regularly varying functions allow to retrieve Theorem 7 of \cite{aristarivero}. That result has a proof, in \cite{aristarivero}, similar to the one that we give for Proposition \ref{vnegprecisequeue0}, and it also refines \eqref{suppnb}: it gives an equivalent at $0$ of the form $C x \overline{\pi}(\log(1/x))$ (with $C = \mathbb{E}[I(Y)^{-\alpha}]/(1+\alpha)$) for the distribution function of an exponential functional of a general L\'evy process that has positive jumps, satisfy assumption $(A)$, and a non-Cram\'er condition. Note that in \cite{aristarivero} they talk about negative jumps because, contrarily to us, there is no minus in their definition of the exponential functional. 
\end{remarque}
}



The study of the spectrally positive case does not go as far as the study of the spectrally negative case. The reason for this is twofold. First, we do not have, in the spectrally positive case, a decomposition of the law of $I(Z^{\uparrow})$ as in \eqref{rae}, which deprives us of an important tool for the study. Secondly we do not need, in the applications, the results on the exponential functional to be as precise, in the spectrally positive case, as in the spectrally negative case. Indeed, in the study of a diffusion in a spectrally negative L\'evy environment $V$ drifting to $-\infty$, a random variable $\mathcal{R}$ appears. Its law is the convolution of the laws of $I(V^{\uparrow})$ and $I(\hat V^{\uparrow})$, where $\hat V := -V$ is the dual process of $V$ and is thus spectrally positive. The combination of the above theorems shows that for some things the behavior of $I(V^{\uparrow})$ is dominant in the study of $\mathcal{R}$ when $V$ has jumps. In particular, the asymptotic tail at $0$ of $\mathcal{R}$ is the same as the one of $I(V^{\uparrow})$. 

The rest of the paper is organized as follows. In Section \ref{prelim} we prove some preliminary results on $V^{\uparrow}$. In Section \ref{decomp} we prove Theorem \ref{finiteandexpomoments} and establish Proposition \ref{decomposition}, the decomposition of the law of $I(V^{\uparrow})$. In Section \ref{tail} we prove Theorems \ref{queuefonctexpo}, \ref{vposcompbrownnullequeue0} and \ref{varreg} by studying the asymptotic behavior of the Laplace transform of $I(V^{\uparrow})$, and we establish Propositions \ref{vposqueue0decond} and \ref{encadrementpatiesavov}. In Section \ref{densit} we prove Theorem \ref{asin13} looking at $I(V^{\uparrow})$ as a self-decomposable random variable. Section \ref{brief} is devoted to the spectrally positive case, it contains the proofs of Theorem \ref{vneglapl} and Theorem \ref{vnegqueue0}. In Section \ref{queue0sautspos} we study the asymptotic tails at $0$ of exponential functionals in the presence of positive jumps, this is where we prove Theorem \ref{queue0general} and Proposition \ref{vnegprecisequeue0}. 

\subsection{The example of drifted Brownian motion conditioned to stay positive} \label{exemplembd}

The simplest case is the intersection of the spectrally positive and the spectrally negative case, that is, when $V$ is a drifted Brownian motion. Most of the results of this paper are already known in this case. We define the $\kappa$-drifted Brownian motion by $W_{\kappa}(t) := W(t) - \frac{\kappa}{2} t$, where $W$ is a standard Brownian motion. It is known that the two processes $W_{\kappa}^{\uparrow}$ and $W_{-\kappa}^{\uparrow}$ are equal in law. This follows, for example, from the expression of the generator of $W_{\kappa}^{\uparrow}$, or from the fact that for positive $\kappa$, the Laplace exponent of $W_{\kappa}^{\sharp}$ is equal to the Laplace exponent of $W_{-\kappa}$, so the processes conditioned to stay positive have the same law. We thus only consider positive $\kappa$. 

It is known (see (4.6) in Andreoletti, Devulder \cite{AndDev}, see also Lemma 6.6 in \cite{advech}) that $I(W_k^{\uparrow})$ is almost surely finite and has Laplace transform
\[ \mathbb{E} [e^{-\lambda I(W_k^{\uparrow})}] = \frac{\frac1{2^{\kappa} \Gamma(1 + \kappa)} (2 \sqrt{2 \lambda})^{\kappa}}{I_{\kappa} (2 \sqrt{2 \lambda})}, \]
where $I_{\kappa}$ is a modified Bessel function. This expression can also be written
\begin{eqnarray}
\mathbb{E} [e^{-\lambda I(W_k^{\uparrow})}] = \frac1{\Gamma(1 + \kappa)} \frac1{\sum_{j=0}^{+\infty} \frac{(2 \lambda)^{j}}{j! \Gamma(1 + j + \kappa)}}. \label{laplcasmb}
\end{eqnarray}
It is easy to see that it can be analytically extended in a neighborhood of $0$, so the random variable $I(W_k^{\uparrow})$ admits some finite exponential moments. {More precisely, we can see that for  any $\lambda \in [-(\kappa+1)/2, 0], j \geq 0$, $\frac{(2 \lambda)^{2j}}{(2j)! \Gamma(1 + 2j + \kappa)} + \frac{(2 \lambda)^{2j+1}}{(2j+1)! \Gamma(1 + 2j + 1 + \kappa)} \geq 0$ (and even $>0$ for any $j\geq 1$). Therefore, $\sum_{j=0}^{+\infty} \frac{(2 \lambda)^{j}}{j! \Gamma(1 + j + \kappa)} > 0$ for $\lambda \in [-(\kappa+1)/2, 0]$ so that \eqref{laplcasmb} can be analytically extended on this interval, and even on a neighborhood of $-(\kappa+1)/2$. This shows that there exists $\epsilon_{\kappa} > 0$ such that 
\[ \forall \lambda < \epsilon_{\kappa} + (\kappa+1)/2, \ \mathbb{E} \left [ e^{\lambda I(W_k^{\uparrow})} \right ] < + \infty. \]
Comparatively, the condition \eqref{momentsexpo} from Theorem \ref{finiteandexpomoments} shows that 
\[ \forall \lambda < \Psi_{W_k}(\kappa + 1) = (\kappa+1)/2, \ \mathbb{E} \left [ e^{\lambda I(W_k^{\uparrow})} \right ] < + \infty, \]
which is a little weaker. This shows in particular that, at least in the brownian case, the condition \eqref{momentsexpo} is not optimal.} 

An easy calculation on the asymptotic of \eqref{laplcasmb} when $\lambda$ goes to infinity yields
\begin{eqnarray}
- \log \left ( \mathbb{E} [e^{-\lambda I(W_k^{\uparrow})}] \right ) \underset{\lambda \rightarrow + \infty}{\sim} 2 \sqrt{2 \lambda}. \label{loglaplcasmb}
\end{eqnarray}
Combining \eqref{loglaplcasmb} and De Bruijn's Theorem (see Theorem 4.12.9 in \cite{regvar}) we get
\begin{eqnarray}
- \log \left ( \mathbb{P} \left ( I(W_k^{\uparrow}) \leq x \right ) \right ) \underset{x \rightarrow 0}{\sim} \frac{2}{x}. \label{casmb}
\end{eqnarray}
This estimate can be seen as a particular case of Corollary \ref{vposstablequeue0} (when the latter is applied in the case of a drifted Brownian motion). 

As the expression \eqref{laplcasmb} extends to the imaginary line, we get the expression of the characteristic function of $I(W_k^{\uparrow})$ which can be proved, using estimates on modified Bessel functions, to belong to the Schwartz space. Therefore, the density of $I(W_k^{\uparrow})$, which is the Fourier transform of its characteristic function, belongs to the Schwartz space, which agrees with Theorem \ref{asin13}. 

\section{Preliminary results on $V^{\uparrow}$} \label{prelim}

\subsection{Exponential functionals and excursion theory}

We fix $y>0$. In this subsection, we use excursions to prove that the integrals of exponential $V^{\uparrow}(\tau(V^{\uparrow}, y) + .)$ or $(V^{\sharp}(t), \ 0 \leq t \leq \tau(V^{\sharp}, y))$ stopped at their last passage time at $y$ and $0$ respectively are equal in law to some subordinators stopped at independent exponential random variables. 

It is easy to see that the regularity of $\left \{ y \right \}$ for the Markovian processes $V^{\sharp}_y$ and $V^{\uparrow}_y$ is equivalent to the regularity of $\left \{ 0 \right \}$ for $V$ (or $V^{\sharp}$) which in turn, according to Corollary VII.5 in \cite{Bertoin}, is equivalent to the fact that $V$ has unbounded variation. The property of $\left \{ y \right \}$ being instantaneous for $V^{\sharp}_y$ and $V^{\uparrow}_y$ is equivalent to the same property of $\left \{ 0 \right \}$ for $V$, but this is a well known property of spectrally negative L\'evy processes that are not opposite of subordinators. $\left \{ y \right \}$ is thus always instantaneous for $V^{\sharp}_y$ and $V^{\uparrow}_y$ and the only alternative is whether it is regular or not, which corresponds to the fact that $V$ has or not unbounded variation.

We apply excursion theory away from $y$ (see Section IV of \cite{Bertoin}). Let us denote by $L_y^{\uparrow}$ (respectively $L_y^{\sharp}$) a local time at $y$ of the process $V^{\uparrow}_y$ (respectively $V^{\sharp}_y$) and $\eta_y^{\uparrow}$ (respectively $\eta_y^{\sharp}$) the associated excursion measure. We denote $\eta^{\sharp}$ for $\eta_0^{\sharp}$.  The inverse of the local time $L_y^{\uparrow, -1}$ (respectively $L_y^{\sharp, -1}$) is a subordinator and $V^{\uparrow}_y$ (respectively $V^{\sharp}_y$) can be represented as a Poisson point process on the set of excursions, with intensity measure $\eta_y^{\uparrow}$ (respectively $\eta_y^{\sharp}$). Note that this is also true in the irregular case (when $V$ has bounded variation) if the local time $L_y^{\uparrow}$ (respectively $L_y^{\sharp}$) is defined artificially as in \cite{Bertoin}, Section IV.5. In this case, the excursion measure is proportional to the law of the first excursion and in particular the total mass of the excursion measure is finite. 


Given $\xi : [0, \zeta] \rightarrow \mathbb{R}$ an excursion away from $y$, we define $\zeta(\xi)$ to be its life-time, $H_y(\xi) := \sup_{[0, \zeta(\xi)]} \xi - y$ its height and $G(\xi) := \int_0^{\zeta(\xi)} e^{-\xi(t)} dt$. 

For any $h > 0$, we consider $IP_h$, $FP_h$ and $N_h$ three subsets that make a partition of the excursions of $V^{\sharp}$ away from $y$. These three subsets are respectively: the set of excursions higher than $h$ that stay positive before reaching $y + h$, the set of excursions of height smaller than $h$ that stay positive, the set of excursions that reach $]-\infty, 0]$ before ever reaching $y+h$: 
\[ IP_h := \left \{ \xi, \ H_y(\xi) \geq h, \ \forall t \in [0, \tau(\xi, y+h)], \xi(t) > 0 \right \}, \] 
\[ FP_h := \left \{ \xi, \ H_y(\xi) < h, \ \forall t \in [0, \zeta(\xi)], \xi(t) > 0 \right \}, \ \ \ N_h := \left \{ \xi, \ \tau(\xi, ]-\infty, 0]) < \tau(\xi, y+h) \right \}. \]
$IP_{\infty}$, $FP_{\infty}$ and $N_{\infty}$ are defined as the monotone limits of the sets $IP_h$, $FP_h$ and $N_h$: $IP_{\infty}$ is the set of infinite excursions that stay positive, $FP_{\infty}$ is the set of finite excursions that stay positive and $N_{\infty}$ is the set of excursions that reach $]-\infty, 0]$. 
$\eta^{\sharp}_y(IP_{\infty})$ and $\eta^{\sharp}_y(N_{\infty})$ are always finite whereas $\eta^{\sharp}_y(FP_{\infty})$ is infinite in the regular case (when $V$ has unbounded variation). Also, note that $\eta^{\sharp}_y(IP_{\infty}) = 0$ if $V$ oscillates. 

\begin{lemme} \label{=subexp}

Let $y$ be positive and let $S$ be a pure jump subordinator with L\'evy measure $G \eta^{\sharp}_y(.\cap FP_{\infty} )$, the image measure of $\eta^{\sharp}_y(.\cap FP_{\infty} )$ by $G$. Let $T$ be an exponential random variable with parameter $\eta_y^{\sharp}(IP_{\infty}) + \eta_y^{\sharp}(N_{\infty})$, independent of $S$. We have 
\begin{eqnarray} \int_{\tau(V^{\uparrow}, y)}^{\mathcal{R}(V^{\uparrow}, y)} e^{- V^{\uparrow}(t)} dt \overset{\mathcal{L}}{=} S_T, \label{=subexp1} \end{eqnarray}
where $\tau(., .)$ and $\mathcal{R}(., .)$ are defined in the introduction. 
\end{lemme}

\begin{proof}

$V^{\uparrow}(\tau(V^{\uparrow}, y) + .)$ has the same law as $V^{\uparrow}_y$, from the Markov property applied to $V^{\uparrow}$ at time $\tau(V^{\uparrow}, y)$ and the absence of positive jumps. As a consequence, $\int_{\tau(V^{\uparrow}, y)}^{\mathcal{R}(V^{\uparrow}, y)} e^{- V^{\uparrow}(t)} dt$ is equal in law to $\int_0^{\mathcal{R}(V_y^{\uparrow}, y)} e^{- V_y^{\uparrow}(t)} dt$ and we are left to prove the result for the latter. 

Then, let us fix $h >0$. As it is mentioned in the introduction, $(V^{\uparrow}_y(t), \ 0 \leq t \leq \tau(V^{\uparrow}_y, y +h))$ is equal in law to $(V^{\sharp}_y(t), \ 0 \leq t \leq \tau(V^{\sharp}_y, y +h))$ conditionally on $\{ \tau(V^{\sharp}_y, y+h) < \tau(V^{\sharp}_y, ]-\infty, 0]) \}$. 

$(V^{\sharp}_y(t), \ 0 \leq t \leq \tau(V^{\sharp}_y, y +h))$ can be built from a Poisson point process on the set of excursions with intensity measure $\eta_y^{\sharp}$. $(V^{\uparrow}_y(t), \ 0 \leq t \leq \tau(V^{\uparrow}_y, y +h))$ can be built from this same process, conditioned not to have jumps in $N_h$ before its first jump in $IP_h$. In other words, we build $(V^{\uparrow}_y(t), \ 0 \leq t \leq \tau(V^{\uparrow}_y, y +h))$ from the process of jumps in $FP_h$ stopped at the exponential time (that has parameter $\eta_y^{\sharp}(IP_h) + \eta_y^{\sharp}(N_h)$) at which occurs the first jump in $IP_h \cup N$ and conditionally to the fact that this jump belongs to $IP_h$. Then, {note that the processes of jumps in $FP_h$ and in $IP_h \cup N$ are independent from one another,} and by a property of Poisson point processes the fact that the first jump in $IP_h \cup N$ belongs to $IP_h$ is independent of the time when this jump occurs. As a consequence, $(V^{\uparrow}_y(t), \ 0 \leq t \leq \tau(V^{\uparrow}_y, y +h))$ is built from a Poisson point process with intensity measure $\eta_y^{\sharp}(. \cap FP_h)$, until an independent exponential time, $T_h$, of parameter $\eta_y^{\sharp}(IP_h) + \eta_y^{\sharp}(N_h)$ where we pick, independently, a jump following the law $\eta_y^{\sharp}(. \cap IP_h)/\eta_y^{\sharp}(IP_h)$ (and we only keep the part of this excursion that is before its hitting time of $y+h$). 


Let $\mathcal{R}^h(V_y^{\uparrow}, y)$ be the last passage time of $V_y^{\uparrow}$ at $y$ before $\tau(V_y^{\uparrow}, y+h)$: 
\[ \mathcal{R}^h(V_y^{\uparrow}, y) := \sup \{ t \in [0, \tau(V_y^{\uparrow}, y+h)], \ V_y^{\uparrow}(t) = y \}. \]
From above, if $(p_s)_{s \geq 0}$ is a Poisson point process in $FP_{\infty}$ with intensity measure $\eta^{\sharp}_y(.\cap FP_{\infty} )$ and if $T_h$ is an independent exponential random variable with parameter $\eta_y^{\sharp}(IP_{h}) + \eta_y^{\sharp}(N_h)$, then $(V^{\uparrow}_y(t), \ 0 \leq t \leq \mathcal{R}^h(V_y^{\uparrow}, y))$ is built by putting aside the excursions of the process $(p_s \mathds{1}_{p_s \in FP_h}, \ 0 \leq s \leq T_h)$. Since $V_y^{\uparrow}$ converges almost surely to $+\infty$, $\mathcal{R}^h(V_y^{\uparrow}, y)$ converges almost surely to $\mathcal{R}(V_y^{\uparrow}, y)$, the last passage time at $y$, when $h$ goes to infinity. On the other hand, $\eta_y^{\sharp}(IP_h) + \eta_y^{\sharp}(N_h)$ converges to $\eta_y^{\sharp}(IP_{\infty}) + \eta_y^{\sharp}(N_{\infty})$ when $h$ goes to infinity. 
As a consequence $T_h$ converges to an exponential random variable $T$ with parameter $\eta_y^{\sharp}(IP_{\infty}) + \eta_y^{\sharp}(N_{\infty}) > 0$. Also, $FP_h$ increases to $FP_{\infty}$ when $h$ goes to infinity. 
Then, identifying the limits when $h$ goes to infinity of both $(V^{\uparrow}_y(t), \ 0 \leq t \leq \mathcal{R}^h(V_y^{\uparrow}, y))$ and $(p_s \mathds{1}_{p_s \in FP_h}, \ 0 \leq s \leq T_h)$, we get that $(V^{\uparrow}_y(t), \ 0 \leq t \leq \mathcal{R}(V_y^{\uparrow}, y))$ is built by putting aside the excursions of the process $(p_s, \ 0 \leq s \leq T)$, where $T$ is an exponential random variable with parameter $\eta_y^{\sharp}(IP_{\infty}) + \eta_y^{\sharp}(N_{\infty})$ and is independent from $(p_s, \ s \geq 0)$. 

Now, remark that $\int_0^{\mathcal{R}(V_y^{\uparrow}, y)} e^{- V_y^{\uparrow}(t)} dt$ is the sum of the images by $G$ of the excursions of $(V^{\uparrow}_y(t), \ 0 \leq t \leq \mathcal{R}(V_y^{\uparrow}, y))$ away from $y$. We thus have 
\begin{eqnarray} \int_0^{\mathcal{R}(V_y^{\uparrow}, y)} e^{- V_y^{\uparrow}(t)} dt \overset{\mathcal{L}}{=} \sum_{0 < s < T} G(p_s). \label{=subexp3} \end{eqnarray}


By properties of Poisson point processes, the process in the right hand side, $\sum_{0 < s < .} G(p_s)$, is the sum of the jumps of a Poisson point process on $\mathbb{R}_+$, with intensity measure $G \eta_y^{\sharp}(.\cap FP_{\infty} )$. Thus, from the L\'evy-Ito decomposition, it has the same law as the subordinator $S$, which yields the result. 

\end{proof}

\begin{remarque} \label{casvb0}
In the case where $V$ has bounded variation, the total mass of $\eta_y^{\sharp}$ is finite so $S$ is only a compound Poisson process. In particular $S_T$  is then null with a positive probability. 
\end{remarque}

{In the rest of this subsection} $y>0$ is still fixed (and arbitrary). Let $\mathcal{R}^y(V^{\sharp}, 0)$ be the last passage time of $V^{\sharp}$ at $0$ before $\tau(V^{\sharp}, y)$: 
\[ \mathcal{R}^y(V^{\sharp}, 0) := \sup \{ t \in [0, \tau(V^{\sharp}, y)], \ V^{\sharp}(t) = 0 \}. \]
In order to study the trajectory of $V^{\sharp}$ before $\mathcal{R}^y(V^{\sharp}, 0)$, we now consider excursions away from $0$. 
Let $I_y$ and $F_y$ denote respectively the subset of excursions higher than $y$ and lower than $y$: 
\[ I_y := \left \{ \xi, \ H_0(\xi) \geq y \right \}, \ \ \ F_y := \left \{ \xi, \ H_0(\xi) < y \right \}. \]
A simpler proof as for Lemma \ref{=subexp} gives the following lemma. 

\begin{lemme} \label{=subexpbis}

Let $S$ be a pure jump subordinator with L\'evy measure $G \eta^{\sharp}(.\cap F_y )$, the image measure of $\eta^{\sharp}(.\cap F_y )$ by $G$. Let $T$ be an exponential random variable with parameter $\eta^{\sharp}(I_y)$ which is independent of $S$. We have 
\begin{eqnarray} \int_0^{\mathcal{R}^y(V^{\sharp}, 0)} e^{- V^{\sharp}(t)} dt \overset{\mathcal{L}}{=} S_T. \label{=subexp4} \end{eqnarray}

\end{lemme}

In the case where $V$ does not oscillate, $V^{\sharp}$ drifts to $+\infty$ and we need to study the trajectory before $\mathcal{R}(V^{\sharp}, 0)$, the last passage time of $V^{\sharp}$ at $0$. We still consider excursions away from $0$. Let $I$ and $F$ denote respectively the subsets of infinite and finite excursions: 
\[ I := \left \{ \xi, \ \zeta(\xi) = +\infty \right \}, \ \ \ F := \left \{ \xi, \ \zeta(\xi) < +\infty \right \}. \]
A similar proof as for Lemma \ref{=subexp} gives the following lemma. 

\begin{lemme} \label{=subexpbisbis}

We assume that $V$ does not oscillate. Let $S$ be a pure jump subordinator with L\'evy measure $G \eta^{\sharp}(.\cap F )$, the image measure of $\eta^{\sharp}(.\cap F )$ by $G$. Let $T$ be an exponential random variable with parameter $\eta^{\sharp}(I)$ which is independent of $S$. We have 
\begin{eqnarray} \int_0^{\mathcal{R}(V^{\sharp}, 0)} e^{- V^{\sharp}(t)} dt \overset{\mathcal{L}}{=} S_T. \label{=subexp4bis} \end{eqnarray}

\end{lemme}

\subsection{$V^{\uparrow}$ and $V^{\sharp}$ shifted at a last passage time}

To obtain the decomposition \eqref{rae} of the law of $I(V^{\uparrow})$, we split $V^{\uparrow}$ at its last passage time at a point $y$ and obtain two independent trajectories that we can identify. 

\begin{lemme} (Corollary VII.19 of \cite{Bertoin}) \label{dernierpassage1}

For any positive $y$, the two trajectories 
\[ \left ( V^{\uparrow}(t), \ 0 \leq t \leq \mathcal{R}(V^{\uparrow}, y) \right ) \ \text{and} \ \left ( V^{\uparrow}(t + \mathcal{R}(V^{\uparrow}, y)) - y, \ t \geq 0 \right ) \]
are independent and the second is equal in law to $V^{\uparrow}$. 
\end{lemme}

\begin{lemme} \label{dernierpassage2}

\begin{itemize}
\item The two trajectories $( V^{\sharp}(t), \ 0 \leq t \leq \mathcal{R}^y(V^{\sharp}, 0) )$
\[ \text{and} \ \left ( V^{\sharp}(t + \mathcal{R}^y(V^{\sharp}, 0)), \ 0 \leq t \leq \tau(V^{\sharp}, y) - \mathcal{R}^y(V^{\sharp}, 0) \right ) \]
are independent and the second is equal in law to $( V^{\uparrow}(t), \ 0 \leq t \leq \tau(V^{\uparrow}, y) )$. As a consequence we have $\tau(V^{\uparrow}, y) \overset{\mathcal{L}}{=} \tau(V^{\sharp}, y) - \mathcal{R}^y(V^{\sharp}, 0) \leq \tau(V^{\sharp}, y)$. 
\medbreak
\item We assume that $V$ does not oscillate. The two trajectories 
\[ \left ( V^{\sharp}(t), \ 0 \leq t \leq \mathcal{R}(V^{\sharp}, 0) \right ) \ \text{and} \ \left ( V^{\sharp}(t + \mathcal{R}(V^{\sharp}, 0)), \ t \geq 0 \right ) \]
are independent and the second is equal in law to $V^{\uparrow}$. 
\end{itemize}

\end{lemme}

\begin{proof}

We fix $y > 0$ and $a \in ]0, y[$. Let us denote by $(e(s), \ s \geq 0)$ the process of excursions of $V^{\sharp}$ away from $0$.  
Recall the notations $I_y$ and $F_y$. 
$T_{y} := \inf \{ s \geq 0, \ e(s) \in I_y \}$ is the time when occurs the first excursion higher than $y$ and $\xi_{y}$ denotes this excursion. 

Decomposing $V^{\sharp}$ as its excursions away from $0$, we see that $\mathcal{R}^y(V^{\sharp}, 0)$ is the instant when begins the first excursion higher than $y$, so
\begin{eqnarray}
\left ( V^{\sharp}(t + \mathcal{R}^y(V^{\sharp}, 0)), \ 0 \leq t \leq \tau(V^{\sharp}, y) - \mathcal{R}^y(V^{\sharp}, 0) \right ) = \left ( \xi_{y}(t), \ 0 \leq t \leq \tau(\xi_{y}, y) \right ). \label{dernierpassage1.1}
\end{eqnarray}

$( V^{\sharp}(t), \ 0 \leq t \leq \mathcal{R}^y(V^{\sharp}, 0) )$ is thus a function of $(e(s) \mathds{1}_{e(s) \in F_y}, \ 0 \leq s \leq T_{y})$ while $( V^{\sharp}(t + \mathcal{R}^y(V^{\sharp}, 0)), \ 0 \leq t \leq \tau(V^{\sharp}, y) - \mathcal{R}^y(V^{\sharp}, 0) )$ is a function of $\xi_{y}$. By properties of Poisson point processes, $T_{y}$ is an exponential random variable independent of $\xi_{y}$ and the process of finite excursions $(e(s) \mathds{1}_{e(s) \in F_y}, \ s \geq 0)$ is also independent of $\xi_{y}$. Therefore the objects $(e(s) \mathds{1}_{e(s) \in F_y}, \ 0 \leq s \leq T_{y})$ and $\xi_{y}$ are independent. From this independence we deduce that
\[ \left ( V^{\sharp}(t), \ 0 \leq t \leq \mathcal{R}^y(V^{\sharp}, 0) \right ) \perp\!\!\!\perp \left ( V^{\sharp}(t + \mathcal{R}^y(V^{\sharp}, 0)), \ 0 \leq t \leq \tau(V^{\sharp}, y) - \mathcal{R}^y(V^{\sharp}, 0) \right ), \]
which is the required independence. It only remains to prove that the right hand side in \eqref{dernierpassage1.1} has the same law as $( V^{\uparrow}(t), \ 0 \leq t \leq \tau(V^{\uparrow}, y) )$. 

Using the Markov property at time $\tau(\xi_a, a)$, for an excursion $\xi_a \in I_a$, we have that $\xi_a(. + \tau(\xi_a, a))$ equals in law $V^{\sharp}_{a}$ killed when it ever reaches $0$. 

Since $I_y \subset I_a$ we can apply this to an excursion $\xi_y \in I_y$ and get that $( \xi_y(t + \tau(\xi_y, a)), \ 0 \leq t \leq \tau(\xi_y, y) - \tau(\xi_y, a))$ is equal in law to $( V^{\sharp}_{a}(t), \ 0 \leq t \leq \tau(V^{\sharp}_{a}, y))$ conditioned to reach $y$ before $0$. Since $V^{\sharp}_{a}$ has no positive jumps, reaching $y$ before $0$ is the same as reaching $y$ before $]-\infty, 0]$. 

As we mentioned in the introduction, $( V^{\sharp}_{a}(t), \ 0 \leq t \leq \tau(V^{\sharp}_{a}, y))$ conditioned to reach $y$ before $]-\infty, 0]$ is equal in law to $( V^{\uparrow}_{a}(t), \ 0 \leq t \leq \tau(V^{\uparrow}_{a}, y))$. Putting all this together we get 
\[ \left ( \xi_y(t + \tau(\xi_y, a)), \ 0 \leq t \leq \tau(\xi_y, y) - \tau(\xi_y, a) \right ) \overset{\mathcal{L}}{=} \left ( V^{\uparrow}_{a}(t), \ 0 \leq t \leq \tau(V^{\uparrow}_a, y) \right ). \]
Since $\tau(\xi_y, a)$ converges almost surely to $0$ when $a$ goes to $0$ and $V^{\uparrow}_a$ converges in distribution to $V^{\uparrow}$ according to Proposition VII.14 in \cite{Bertoin}, we can let $a$ go to $0$ in both members and get
\begin{eqnarray}
\left ( \xi_y(t), \ 0 \leq t \leq \tau(\xi_y, y) \right ) \overset{\mathcal{L}}{=} \left ( V^{\uparrow}(t), \ 0 \leq t \leq \tau(V^{\uparrow}, y) \right ). \label{dernierpassage1.2}
\end{eqnarray}
As a consequence the right hand side in \eqref{dernierpassage1.1} has the same law as $( V^{\uparrow}(t), \ 0 \leq t \leq \tau(V^{\uparrow}, y) )$, which concludes the proof of the first point of the lemma. 

We now assume that $V$ does not oscillate (so that $V^{\sharp}$ drifts to $+\infty$) and we prove the second point. For the independence, the arguments of the proof of the first point can be repeated, just replacing $y$ by $+\infty$ (we consider $\xi_{\infty}$, the infinite excursion away from $0$, instead of the first excursion higher than $y$). To prove that $( V^{\sharp}(t + \mathcal{R}(V^{\sharp}, 0)), \ t \geq 0)$ is equal in law to $V^{\uparrow}$, it suffices to prove that $\xi_{\infty}$ is equal in law to $V^{\uparrow}$. Let $y$ be finite, we know from the proof of the first point that \eqref{dernierpassage1.2} is true for any excursion in $I_y$. Since $\xi_{\infty} \in I_y$ we have 
\[ \left ( \xi_{\infty}(t), \ 0 \leq t \leq \tau(\xi_{\infty}, y) \right ) \overset{\mathcal{L}}{=} \left ( V^{\uparrow}(t), \ 0 \leq t \leq \tau(V^{\uparrow}, y) \right ). \]
Since $y$ is arbitrary and $\tau(V^{\uparrow}, y)$ converges almost surely to $+\infty$ when $y$ goes to $+\infty$, we get 
\[ \left ( \xi_{\infty}(t), \ t \geq 0 \right ) \overset{\mathcal{L}}{=} \left ( V^{\uparrow}(t), \ t \geq 0 \right ), \]
which gives the result. 

\end{proof}

%
%
%

\section{Finiteness, exponential moments, and self-decomposability} \label{decomp}

\subsection{Finiteness and exponential moments: Proof of Theorem \ref{finiteandexpomoments}} \label{existelapl}

We are grateful to an anonymous referee for the following proof that is considerably simpler than the one given by the author in the previous versions of this paper. 

\begin{proof} of Theorem \ref{finiteandexpomoments} (the finiteness of $I(V^{\uparrow})$ and \eqref{momentsexpo})

The idea of the proof is to provide finite upper bounds for the moments of $I(V^{\uparrow})$. The first step is to prove that $\mathbb{E} [I(V^{\uparrow})] < +\infty$. Using Fubini's Theorem and Corollary VII.16 of \cite{Bertoin} we have 
\begin{align*}
\mathbb{E} [I(V^{\uparrow})] = \int_0^{+\infty} \mathbb{E} [e^{-V^{\uparrow}(t)}] dt & = \int_0^{+\infty} \int_0^{+\infty} e^{-y} \mathbb{P} (V^{\uparrow}(t) \in dy ) dt \\
& = \int_0^{+\infty} \int_0^{+\infty} e^{-y} \frac{y W(y)}{t} \mathbb{P} (V(t) \in dy ) dt. 
\end{align*}
Now, using Corollary VII.3 of \cite{Bertoin} we get 
\[ \mathbb{E} [I(V^{\uparrow})] = \int_0^{+\infty} \int_0^{+\infty} e^{-y} W(y) \mathbb{P} (\tau(V, y) \in dt ) dy = \int_0^{+\infty} e^{-y} W(y) \mathbb{P} (\tau(V, y) < +\infty ) dy. \]
Since $\mathbb{P} (\tau(V, y) < +\infty ) = \mathbb{P} (\sup_{[0, +\infty[} V \geq y) = e^{-\kappa y}$ we obtain 
\[ \mathbb{E} [I(V^{\uparrow})] = \int_0^{+\infty} e^{-(1+\kappa)y} W(y) dy = \frac{1}{\Psi_V(\kappa + 1)} < +\infty. \]
The finiteness in the above expression comes from the fact that $\int_0^{+\infty} e^{-\lambda y} W(y) dy < +\infty$ for $\lambda > \kappa$. 
As a consequence, the exponential functional $I(V^{\uparrow})$ is almost surely finite and has finite expectation. 

We now turn to the proof of the finiteness of the Laplace transform. We proceed by bounding the moments of the exponential functional. For any $x \geq 0$ let us define $h(x) := \mathbb{E} [I(V^{\uparrow}_x)]$. For any $x > 0$ we have 
\[ I(V^{\uparrow}) = \int_0^{+\infty} e^{-V^{\uparrow}(t)} dt \geq \int_{\tau(V, x)}^{+\infty} e^{-V^{\uparrow}(t)} dt \overset{\mathcal{L}}{=} \int_0^{+\infty} e^{-V_x^{\uparrow}(t)} dt = I(V_x^{\uparrow}). \]
For the equality in law in the above expression we have used the Markov property for $V^{\uparrow}$ at time $\tau(V, x)$. As a consequence we have 
\begin{eqnarray}
\forall x > 0, \ h(x) = \mathbb{E} [I(V^{\uparrow}_x)] \leq \mathbb{E} [I(V^{\uparrow})] = 1/\Psi_V(\kappa + 1) < +\infty. \label{majo1stmoment}
\end{eqnarray}

{We could now conclude the proof of \eqref{momentsexpo} by making explicit Khas'minskii's condition in our particular case (see Section 1 of \cite{FITZSIMMONS1999117}), but let us give more details for the reader's convenience. We proceed similarly as in the proof of Kac's moment formula in Section 2 of \cite{FITZSIMMONS1999117} where they calculate by induction the moments of an additive functional of a Markov process in term of its potential kernel.} Note that for any $k \geq 1$, 
\begin{align*}
\mathbb{E} \left [ \left ( I(V^{\uparrow}) \right )^k \right ] & = \mathbb{E} \left [ \int_0^{+\infty} ... \int_0^{+\infty} e^{-V^{\uparrow}(t_1)} \times ... \times e^{-V^{\uparrow}(t_k)} dt_1 ... dt_k \right ] \\
& = k ! \ \mathbb{E} \left [ \int_{0 \leq t_1 < ... < t_k} e^{-V^{\uparrow}(t_1)} \times ... \times e^{-V^{\uparrow}(t_k)} dt_1 ... dt_k \right ], 
\end{align*}
so that 
\begin{eqnarray}
\forall k \geq 1, \ \mathbb{E} \left [ \left ( I(V^{\uparrow}) \right )^k \right ] / k ! \ = \mathbb{E} \left [ \int_{0 \leq t_1 < ... < t_k} e^{-V^{\uparrow}(t_1)} \times ... \times e^{-V^{\uparrow}(t_k)} dt_1 ... dt_k \right ]. \label{majokthmoment1}
\end{eqnarray}

Let us prove by induction that for any $k \geq 1$, 
\begin{eqnarray}
\mathbb{E} \left [ \left ( I(V^{\uparrow}) \right )^k \right ] \leq k ! \left ( \mathbb{E} [I(V^{\uparrow})] \right )^k < +\infty. \label{majokthmoment}
\end{eqnarray}
\eqref{majokthmoment} is clearly true for $k = 1$. Let us assume that it is true for some arbitrary rank $k$. According to \eqref{majokthmoment1}, $\mathbb{E} [ ( I(V^{\uparrow}) )^{k+1} ] / (k+1) ! $ equals 
\begin{align*}
& \mathbb{E} \left [ \int_{0 \leq t_1 < ... < t_{k+1}} e^{-V^{\uparrow}(t_1)} \times ... \times e^{-V^{\uparrow}(t_k)} \times e^{-V^{\uparrow}(t_{k+1})} dt_1 ... dt_k dt_{k+1} \right ] \\
= & \mathbb{E} \left [ \int_{0 \leq t_1 < ... < t_{k}} e^{-V^{\uparrow}(t_1)} \times ... \times e^{-V^{\uparrow}(t_k)} \left ( \int_{t_k}^{+\infty} e^{-V^{\uparrow}(s)} ds \right ) dt_1 ... dt_k \right ] \\
= & \mathbb{E} \left [ \int_{0 \leq t_1 < ... < t_{k}} e^{-V^{\uparrow}(t_1)} \times ... \times e^{-V^{\uparrow}(t_k)} \mathbb{E} \left [ \int_{t_k}^{+\infty} e^{-V^{\uparrow}(s)} ds \big | \sigma (V^{\uparrow}(u), \ 0 \leq u \leq t_k) \right ] dt_1 ... dt_k \right ]. 
\end{align*}
From the Markov property at time $t_k$, the conditional expectation in the above expression equals $h(V^{\uparrow}(t_k))$ which, according to \eqref{majo1stmoment}, is almost surely less than $\mathbb{E} [I(V^{\uparrow})]$. We thus get 
\begin{align*}
\mathbb{E} [ ( I(V^{\uparrow}) )^{k+1} ] / (k+1) ! & \leq \mathbb{E} [I(V^{\uparrow})] \times \mathbb{E} \left [ \int_{{0 \leq t_1 < ... < t_k}} e^{-V^{\uparrow}(t_1)} \times ... \times e^{-V^{\uparrow}(t_k)} dt_1 ... dt_k \right ] \\
& = \mathbb{E} [I(V^{\uparrow})] \times \mathbb{E} \left [ \left ( I(V^{\uparrow}) \right )^k \right ] / k ! \leq \left ( \mathbb{E} [I(V^{\uparrow})] \right )^{k+1}, 
\end{align*}
where we have used \eqref{majokthmoment1} and the induction hypothesis. Thus the induction is proved. As a consequence, for all $\lambda \in ]0, 1/ \mathbb{E} [I(V^{\uparrow})] [ =  ]0, \Psi_V(\kappa + 1)[$ we have 
\[ \mathbb{E} \left [ e^{\lambda I(V^{\uparrow})} \right ] = \sum_{k \geq 0} \lambda^k \mathbb{E} \left [ \left ( I(V^{\uparrow}) \right )^k \right ] / k ! \leq \sum_{k \geq 0} \left ( \lambda \mathbb{E} [I(V^{\uparrow})] \right )^k < +\infty. \]
The finiteness of the Laplace transform is obvious for $\lambda \leq 0$, so \eqref{momentsexpo} is proved. 
\end{proof}

{Note that \eqref{majokthmoment1} is a particular case of $(16)-(17)$ from \cite{FITZSIMMONS1999117}. Then, the way we took the conditional expectation in the multiple integral in order to use the Markov property is similar to the procedure applied in Section 2 of \cite{FITZSIMMONS1999117}. Finally, note that, as we said in the proof, the condition \eqref{momentsexpo} for a real $\lambda$ to be in the domain of the Laplace transform of $I(V^{\uparrow})$ coincides, in our particular case, with Khas'minskii's condition stated in Section 1 of \cite{FITZSIMMONS1999117}. Indeed, the function denoted by $G_v \mathds{1}$ in \cite{FITZSIMMONS1999117} corresponds, for our particular case, with the function denoted by $h$ in the above proof, and we see from \eqref{majo1stmoment} that $||h||_{\infty}$, the maximum value of $h$, is attained at $0$ and equals $\mathbb{E} [I(V^{\uparrow})] = 1/\Psi_V(\kappa + 1)$.} Note moreover that the argument of the above proof shows that, if $I = \int_0^{+\infty} f(A_0(t)) dt$ where $f$ is a positive function and $A$ is a positive c\`ad-l\`ag Markovian process with no positive jumps, then $I$ admits some finite exponential moments if and only if $I$ has finite expectation. 

\begin{remarque} \label{preuvealternative}

If $V$ does not oscillate, the finiteness of $I(V^{\uparrow})$ can be derived as a consequence of Lemma \ref{dernierpassage2}. Indeed, from the second statement of Lemma \ref{dernierpassage2} we have
\[ I(V^{\uparrow}) \overset{\mathcal{L}}{=} \int_0^{+\infty} e^{-V^{\sharp}(t + \mathcal{R}(V^{\sharp}, 0))} dt = \int_{\mathcal{R}(V^{\sharp}, 0)}^{+\infty} e^{-V^{\sharp}(t)} dt \leq \int_0^{+\infty} e^{-V^{\sharp}(t)} dt = I(V^{\sharp}). \]
Then, $V^{\sharp}$ drifts to $+\infty$ so Theorem 1 in \cite{Bertoinyor} ensures that $I(V^{\sharp})$ is almost surely finite which yields the result. Note that, if one would wish to study $I(V^{\uparrow})$ using the above inequality, one would {lose} all the relevant information about the right tail of the functional. Indeed, the right tail of $I(V^{\uparrow})$ is at most exponential while the right tail of $I(V^{\sharp})$ is at least inverse polynomial, as we said in the Introduction. However, the above inequality leads to a relevant {comparison} of the left tails of $I(V^{\uparrow})$ and $I(V^{\sharp})$ in Proposition \ref{vposqueue0decond}. 

\end{remarque}

\subsection{Decomposition of the law of $I(V^{\uparrow})$}

In this subsection, we prove that the law of $I(V^{\uparrow})$ is solution of the random affine equation \eqref{rae} and we give a decomposition of its non-trivial coefficient $A^y$. This is a key point of our analysis of the law of $I(V^{\uparrow})$. The self-decomposability of $I(V^{\uparrow})$, which is a direct consequence of the next proposition, has already been proved in \cite{bertoincaballero} and \cite{Pardo2009} using the same argument. However we give a proof for the sake of completeness, and because we go further, by decomposing the variable $A^y$ appearing in the relation of self-decomposability. 

\begin{prop} \label{decomposition}
For any $y>0$, the law of $I(V^{\uparrow})$ satisfies the random affine equation 
\begin{eqnarray}
I(V^{\uparrow}) \overset{\mathcal{L}}{=} \int_0^{\tau(V^{\uparrow}, y)} e^{- V^{\uparrow}(t)} dt + S_T + e^{-y} I(\tilde V^{\uparrow}). \label{kesten1.0}
\end{eqnarray}
The three terms of the right hand side are independent, $S_T$ is as in Lemma \ref{=subexp}, and $\tilde V^{\uparrow}$ is an independent copy of $V^{\uparrow}$. We define 
\[ A^y := \int_0^{\tau(V^{\uparrow}, y)} e^{- V^{\uparrow}(t)} dt + S_T \]
to lighten notations. As a consequence, $I(V^{\uparrow})$ has the same law as the sum of the random series 
\begin{eqnarray}
I(V^{\uparrow}) \overset{\mathcal{L}}{=} \sum_{k \geq 0} e^{-ky} A^y_k, \label{kesten2.0}
\end{eqnarray}
where the random variables $A^y_k$ are \textit{iid} and have the same law as $A^y$. 

\end{prop}

Let us remark that, according to Theorem VII.18 in \cite{Bertoin}, $A^y$ is equal in law to $e^{-y} \int_0^{\tau(V, y)} e^{V(t)} dt$ conditionally on $\{ \tau(V, y) < +\infty \}$. 

\begin{remarque} \label{randompowerseries}
The almost sure convergence of the random series in \eqref{kesten2.0} is a consequence of the almost sure finiteness, given by Theorem \ref{finiteandexpomoments}, of the positive random variable $I(V^{\uparrow})$. Also, it is a well known fact on random power series with $\textit{iid}$ coefficients that their radius of convergence is almost surely equal to a constant belonging to $\{ 0, 1 \}$. Since this constant, in the case of the power series in \eqref{kesten2.0}, is greater than $e^{-y}$, we deduce that it equals $1$. 
\end{remarque}

\begin{proof} of Proposition \ref{decomposition}

We fix $y>0$. As $V^{\uparrow}$ has no positive jumps and goes to infinity we have $\tau(V^{\uparrow}, y) \leq \mathcal{R}(V^{\uparrow}, y) < +\infty$ and $V^{\uparrow} ( \tau(V^{\uparrow}, y) ) = V^{\uparrow} ( \mathcal{R}(V^{\uparrow}, y) ) = y$. We write 
\begin{align*} I(V^{\uparrow}) & = \int_0^{\mathcal{R}(V^{\uparrow}, y)} e^{- V^{\uparrow}(t)} dt + \int_{\mathcal{R}(V^{\uparrow}, y)}^{+ \infty} e^{- V^{\uparrow}(t)} dt \\
& = \int_0^{\mathcal{R}(V^{\uparrow}, y)} e^{- V^{\uparrow}(t)} dt + e^{-y} \int_0^{+ \infty} e^{- \left ( V^{\uparrow}(t + \mathcal{R}(V^{\uparrow}, y)) - y \right )} dt \\
& \overset{\mathcal{L}}{=} \int_0^{\mathcal{R}(V^{\uparrow}, y)} e^{- V^{\uparrow}(t)} dt + e^{-y} I(\tilde V^{\uparrow}). 
\end{align*}
We have used Lemma \ref{dernierpassage1} for the last equality in which $\tilde V^{\uparrow}$ is an independent copy of $V^{\uparrow}$. 
We now decompose 
\begin{eqnarray}
\int_0^{\mathcal{R}(V^{\uparrow}, y)} e^{- V^{\uparrow}(t)} dt = \int_0^{\tau(V^{\uparrow}, y)} e^{- V^{\uparrow}(t)} dt + \int_{\tau(V^{\uparrow}, y)}^{\mathcal{R}(V^{\uparrow}, y)} e^{- V^{\uparrow}(t)} dt. \label{coupetau1}
\end{eqnarray}

Since $V^{\uparrow}(\tau(V^{\uparrow}, y))=y$, combining with the Markov property at time $\tau(V^{\uparrow}, y)$, the two terms in the right hand side of $(\ref{coupetau1})$ are independent: 
\[ \int_0^{\tau(V^{\uparrow}, y)} e^{- V^{\uparrow}(t)} dt \perp\!\!\!\perp \int_{\tau(V^{\uparrow}, y)}^{\mathcal{R}(V^{\uparrow}, y)} e^{- V^{\uparrow}(t)} dt. \]
Now, thanks to Lemma \ref{=subexp}, the second term has the same law as $S_T$ with $S_T$ as in the lemma. This achieves the proof. 

\end{proof}

As a direct consequence of \eqref{kesten2.0} we can now prove the remaining part of Theorem \ref{finiteandexpomoments}. 

\begin{proof} of Theorem \ref{finiteandexpomoments} \eqref{queueexpo}

Let $S$ and $T$ be as in Proposition \ref{decomposition} and $\epsilon \in ]0, \mathbb{E}[S_1][$ (we can see that $\mathbb{E}[S_1] < +\infty$ but this is not important here). Then we have 
\[ \mathbb{P} \left ( S_T \geq x \right ) \geq \mathbb{P} \left ( S_{x/\epsilon} \geq x \right ) \times \mathbb{P} \left ( T \geq x/\epsilon \right ). \]
The first factor in the right hand side converges to $1$ thanks to the law of large numbers for L\'evy processes (see for example Theorem 36.5 in Sato \cite{Sato}). The second is equal to $e^{-px/\epsilon}$, where $p$ is the parameter of the exponential random variable $T$. Therefore, we have an exponential lower bound for the right tail of $S_T$:  $\mathbb{P} ( S_T \geq t ) \geq e^{-C x}$ for some positive constant $C$, when $x$ is large enough. We then notice from \eqref{kesten2.0} that $I(V^{\uparrow})$ is stochastically greater than $S_T$ so we get an exponential lower bound for the right tail of $I(V^{\uparrow})$, which proves the lower bound in \eqref{queueexpo}. The upper bound in \eqref{queueexpo} follows from \eqref{momentsexpo} and Markov's inequality which finishes the proof of Theorem \ref{finiteandexpomoments}. 
\end{proof}

\begin{remarque}
It is possible to prove the first part of Theorem \ref{finiteandexpomoments} (the fact that $I(V^{\uparrow})$ is finite and admits some finite exponential moments) by invoking \eqref{kesten2.0} and proving that each of the two terms composing $A^y$ admit some finite exponential moments. 
\end{remarque}

\section{Asymptotic tail at $0$} \label{tail}

First, let us prove that Theorem \ref{queuefonctexpo} easily implies Theorem \ref{vposcompbrownnullequeue0} and prove Remark \ref{morepreceisevb}. 

\begin{proof} of Theorem \ref{vposcompbrownnullequeue0}

Assume that Theorem \ref{queuefonctexpo} is proved. 

We first prove \eqref{majogeneralenew}. {Recall that $1 \leq \sigma \leq  \beta \leq 2$.} Let us fix $\beta' > \beta$ and $\epsilon > 0$. From the definition of $\beta$ we have that $\Psi_V(\lambda) \leq \epsilon \lambda^{\beta'}$ for all $\lambda$ large enough. {Since $\beta' > 1$ we can use} the first point of Theorem \ref{queuefonctexpo} with $C = \epsilon, \alpha = \beta'$ and $\delta = 1/2$, we deduce that 
\[ \underset{x \rightarrow 0}{\limsup} \ x^{1/(\beta' - 1)} \log \left ( \mathbb{P} \left ( I(V^{\uparrow}) \leq x \right ) \right ) \leq - (\beta' - 1) / 2 \epsilon^{1/(\beta' - 1)}. \]
Since $\epsilon$ can be chosen as small as we want we obtain \eqref{majogeneralenew}. 

We now assume that $\sigma > 1$ and prove \eqref{minogeneralenew}. Let us fix $\sigma' \in ]1, \sigma[$ and $M > 0$. From the definition of $\sigma$ we have that $\Psi_V(\lambda) \geq M \lambda^{\sigma'}$ for all $\lambda$ large enough. Using the second point of Theorem \ref{queuefonctexpo} with $c = M, \alpha = \sigma'$ and $\delta = 2$ we deduce that 
\[ 0 \geq \underset{x \rightarrow 0}{\liminf} \ x^{1/(\sigma' - 1)} \log \left ( \mathbb{P} \left ( I(V^{\uparrow}) \leq x \right ) \right ) \geq - 2 \sigma'^{\sigma' / (\sigma' - 1)} / M^{1/(\sigma' - 1)}. \]
Since $M$ can be chosen as large as we want we obtain \eqref{minogeneralenew}. 

\end{proof}

\begin{proof} of Remark \ref{morepreceisevb}

Let us assume that $V$ has bounded variation. As it can be seen from Remark \ref{casvb1}, it is the difference of a positive drift $\gamma^* t$ and a pure jump subordinator $S_t$: $\forall t > 0, \ V(t) = \gamma^* t - S_t \leq \gamma^* t$. Let us fix $y>0$, we have almost surely 
\begin{eqnarray}
\int_0^{\tau(V, y)} e^{-V(t)} dt \geq \int_0^{\tau(V, y)} e^{-\gamma^* t} dt = \frac1{\gamma^*} \left (1- e^{-\tau(V, y)} \right ). \label{queuecasvb}
\end{eqnarray}
Since $V$ has bounded variation, we have $\mathbb{P}(V(t) > 0, \ 0 \leq t \leq \tau(V, y)) > 0$ (see for example $(47.1)$ in \cite{Sato}) and we can see that $(V^{\uparrow}(t), \ 0 \leq t \leq \tau(V^{\uparrow}, y))$ is equal in law to $(V, \ 0 \leq t \leq \tau(V, y))$ conditioned in the usual sense to remain positive. 
Combining with \eqref{queuecasvb}, we see that almost surely 
\begin{eqnarray}
I(V^{\uparrow}) \geq \int_0^{\tau(V^{\uparrow}, y)} e^{-V^{\uparrow}(t)} dt \geq \frac1{\gamma^*} \left (1- e^{-\tau(V^{\uparrow}, y)} \right ). \label{queuecasvb1}
\end{eqnarray}
Then, since $\tau(V^{\uparrow}, y)$ converges almost surely to $+\infty$ when $y$ goes to $+\infty$, we deduce that $I(V^{\uparrow})$ is more than the positive constant $1/\gamma^*$ almost surely. 

Reciprocally, the support of $S_t$ contains $0$ because it is pure jump. Therefore, for any $\epsilon > 0$ and $y > 0$, with a positive probability, the left hand side in \eqref{queuecasvb} can be $\epsilon$-close to the right hand side while $V$ stays positive on $[0, \tau(V^{\uparrow}, y)]$. Since $(V^{\uparrow}(t), \ 0 \leq t \leq \tau(V^{\uparrow}, y))$ is equal in law to $(V, \ 0 \leq t \leq \tau(V, y))$ conditioned in the usual sense to remain positive, we deduce that the second term in \eqref{queuecasvb1} can be $\epsilon$-close to the third term with a positive probability. Then, note that almost surely $\tau(V^{\uparrow}, y) \geq y / \gamma^*$ (because it holds for $\tau(V, y)$) and that, from the Markov property, the event $\{ \int_{\tau(V^{\uparrow}, y)}^{+\infty} e^{-V^{\uparrow}(t)} dt < \epsilon \}$ is independent from the event considered in \eqref{queuecasvb1} and has a positive probability provided $y$ is large enough. As a consequence $I(V^{\uparrow})$ can be $3 \epsilon$-close to $1 / \gamma^*$ with a positive probability so the support of the distribution of $I(V^{\uparrow})$ contains $1/\gamma^*$. Then, from the properties of non-degenerate self-decomposable distributions, we deduce that the support of $I(V^{\uparrow})$ is $[1/\gamma^*, +\infty[$. 

\end{proof}

In the next Subsection we exploit Proposition \ref{decomposition} to prepare the proofs of Theorems \ref{queuefonctexpo} and \ref{varreg}. 

\subsection{Laplace transform of $I(V^{\uparrow})$} \label{estimelaplace}

In order to prove asymptotic estimates on $\mathbb{P} ( I(V^{\uparrow}) \leq x )$, we first study the Laplace transform of $I(V^{\uparrow})$ via the series decomposition \eqref{kesten2.0}. It is thus natural that we need first to study the Laplace transform of $A^y$. 

First, let us define a notation. $V^{\sharp}$ is a spectrally negative L\'evy process, so, according to Theorem VII.1 in \cite{Bertoin}, the process $\tau ( V^{\sharp}, . )$ is a subordinator which Laplace exponent $\Phi_{V^{\sharp}}$ is defined for $\lambda \geq 0$ by 
\[ \Phi_{V^{\sharp}}(\lambda) := - \log \left ( \mathbb{E} \left [ e^{-\lambda \tau \left ( V^{\sharp}, 1 \right )} \right ] \right ). \]
Moreover, we have $\Phi_{V^{\sharp}} = \Psi_{V^{\sharp}}^{-1}$. 

\begin{prop} \label{loglaplaceA}

We fix $y > 0$. Let $A^y$ be as in Proposition \ref{decomposition}, then, for all $\epsilon > 0$ and $\lambda$ large enough we have 
\begin{eqnarray}
(1 - \epsilon) y \Phi_{V^{\sharp}}(e^{-y} \lambda) \leq - \log \left ( \mathbb{E} \left [ e^{-\lambda A^y} \right ] \right ) \leq (1 + \epsilon) y \Phi_{V^{\sharp}}(\lambda). \label{asyptlaplay}
\end{eqnarray}

\end{prop}

\begin{proof} 

According to the definition of $A^y$ in Proposition \ref{decomposition}, $A^y$ can be decomposed as the sum of two independent random variables, one having the same law as $\int_0^{\tau(V^{\uparrow}, y)} e^{- V^{\uparrow}(t)} dt$ and another having the same law as $S_T$, defined as in Lemma \ref{=subexp}. Let $\Phi_S$ be the Laplace exponent of the subordinator $S$: 
\begin{eqnarray}
\forall \lambda \geq 0, \ \Phi_S(\lambda) := - \log \left ( \mathbb{E} \left [ e^{-\lambda S_1} \right ] \right ). \label{deflaplexposub}
\end{eqnarray}We can see that the Laplace transform of the random variable $S_T$ is given by 
\[ \forall \lambda \geq 0, \ \mathbb{E} \left [ e^{-\lambda S_T} \right ] = \frac{\eta_y^{\sharp}(IP_{\infty}) + \eta_y^{\sharp}(N_{\infty})}{\eta_y^{\sharp}(IP_{\infty}) + \eta_y^{\sharp}(N_{\infty}) + \Phi_S(\lambda)}. \]

We thus have
\begin{align}
- \log \left ( \mathbb{E} \left [ e^{-\lambda A^y} \right ] \right ) & = - \log \left ( \mathbb{E} \left [ \exp \left (-\lambda \int_0^{\tau(V^{\uparrow}, y)} e^{- V^{\uparrow}(t)} dt \right ) \right ] \right ) - \log \left ( \mathbb{E} \left [ e^{-\lambda S_T} \right ] \right ) \nonumber \\
& = - \log \left ( \mathbb{E} \left [ \exp \left (-\lambda \int_0^{\tau(V^{\uparrow}, y)} e^{- V^{\uparrow}(t)} dt \right ) \right ] \right ) - \log \left ( \frac{p}{p + \Phi_S(\lambda)} \right ), \label{decomplaplay}
\end{align}
where we denoted $p$ for the constant $\eta_y^{\sharp}(IP_{\infty}) + \eta_y^{\sharp}(N_{\infty})$. Using the fact that $V^{\uparrow}$ is non-negative and the first point of Lemma \ref{dernierpassage2} we have 
\begin{eqnarray}
\int_0^{\tau(V^{\uparrow}, y)} e^{- V^{\uparrow}(t)} dt \leq \tau(V^{\uparrow}, y) \overset{sto}{\leq} \tau(V^{\sharp}, y), \label{nouvellepreuve1}
\end{eqnarray}
where $\overset{sto}{\leq}$ denotes a stochastic inequality. As a consequence of \eqref{nouvellepreuve1} and of the definition of $\Phi_{V^{\sharp}}$ we have 
\[ - \log \left ( \mathbb{E} \left [ \exp \left (-\lambda \int_0^{\tau(V^{\uparrow}, y)} e^{- V^{\uparrow}(t)} dt \right ) \right ] \right ) \leq - \log \left ( \mathbb{E} \left [ e^{-\lambda \tau(V^{\sharp}, y)} \right ] \right ) = y \Phi_{V^{\sharp}}(\lambda). \]
Combining this inequality with \eqref{decomplaplay} we obtain 
\begin{eqnarray}
- \log \left ( \mathbb{E} \left [ e^{-\lambda A^y} \right ] \right ) \leq y \Phi_{V^{\sharp}}(\lambda) - \log \left ( \frac{p}{p + \Phi_S(\lambda)} \right ). \label{majolaplay}
\end{eqnarray}
Using the first point of Lemma \ref{dernierpassage2} and Lemma \ref{=subexpbis} we have
\[ - \log \left ( \mathbb{E} \left [ \exp \left (-\lambda \int_0^{\tau(V^{\uparrow}, y)} e^{- V^{\uparrow}(t)} dt \right ) \right ] \right ) = - \log \left ( \mathbb{E} \left [ e^{-\lambda \int_0^{\tau(V^{\sharp}, y)} e^{- V^{\sharp}(t)} dt} \right ] / \mathbb{E} \left [ e^{-\lambda \tilde S_{\tilde T}} \right ] \right ), \]
where $\tilde S_{\tilde T}$ is as $S_T$ from Lemma \ref{=subexpbis}. Here again, if $\Phi_{\tilde S}$ denotes the Laplace exponent of the subordinator $\tilde S$ as in \eqref{deflaplexposub} we have 
\[ \forall \lambda \geq 0, \ \mathbb{E} \left [ e^{-\lambda \tilde S_{\tilde T}} \right ] = \frac{\eta^{\sharp}(I_y)}{\eta^{\sharp}(I_y) + \Phi_{\tilde S}(\lambda)}. \]
Moreover we have 
\[ \int_0^{\tau(V^{\sharp}, y)} e^{- V^{\sharp}(t)} dt \geq e^{-y} \tau(V^{\sharp}, y). \]
Putting together the above three expressions and the definition of $\Phi_{V^{\sharp}}$ we obtain
\[ - \log \left ( \mathbb{E} \left [ \exp \left (-\lambda \int_0^{\tau(V^{\uparrow}, y)} e^{- V^{\uparrow}(t)} dt \right ) \right ] \right ) \geq y \Phi_{V^{\sharp}}(e^{-y} \lambda) + \log \left ( \frac{\eta^{\sharp}(I_y)}{\eta^{\sharp}(I_y) + \Phi_{\tilde S}(\lambda)} \right ). \]
Combining the above inequality with \eqref{decomplaplay} and the fact that the term $- \log ( p/(p + \Phi_{S}(\lambda)) )$ is non-negative we get 
\begin{eqnarray}
- \log \left ( \mathbb{E} \left [ e^{-\lambda A^y} \right ] \right ) \geq y \Phi_{V^{\sharp}}(e^{-y} \lambda) + \log \left ( \frac{\eta^{\sharp}(I_y)}{\eta^{\sharp}(I_y) + \Phi_{\tilde S}(\lambda)} \right ). \label{minolaplay}
\end{eqnarray}
According to the L\'evy-Khintchine formula applied to the pure jump subordinator $S$, the Laplace exponent $\Phi_S$ can be written 
\[ \Phi_S(\lambda) = \int_0^{+\infty} (1 - e^{-\lambda x}) \nu_S(dx). \]
By dominated convergence $\Phi_S(\lambda) / \lambda$ converges to $0$ when $\lambda$ goes to $+\infty$. As a consequence we have $\Phi_S(\lambda) \leq \lambda$ for large $\lambda$. Similarly we have $\Phi_{\tilde S}(\lambda) \leq \lambda$ for large $\lambda$. On the other hand, since $\Psi_{V^{\sharp}} (\lambda) / \lambda^2$ is bounded when $\lambda$ goes to infinity, there is a positive constant $c$ such that $\Phi_{V^{\sharp}}(\lambda) \geq c \lambda^{1/2}$ for large $\lambda$. 
Combining all this with \eqref{majolaplay} and \eqref{minolaplay} we get \eqref{asyptlaplay} for any fixed $\epsilon > 0$ and $\lambda$ large enough. 

\end{proof}

\begin{prop} \label{loglaplacefonctexpo}

Assume that there is $\alpha \geq 1$ and a positive constant $C$ such that for all $\lambda$ large enough we have $\Psi_V(\lambda) \leq C \lambda^{\alpha}$, then we have 
\begin{eqnarray}
\liminf_{\lambda \rightarrow +\infty} - \log \left ( \mathbb{E} \left [ e^{-\lambda I(V^{\uparrow})} \right ] \right ) / \lambda^{1/\alpha} \geq \alpha / C^{1/\alpha}. \label{loglaplacefonctexpo1}
\end{eqnarray}

Assume that there is $\alpha \geq 1$ and a positive constant $c$ such that for all $\lambda$ large enough we have $\Psi_V(\lambda) \geq c \lambda^{\alpha}$, then we have 
\begin{eqnarray}
\limsup_{\lambda \rightarrow +\infty} - \log \left ( \mathbb{E} \left [ e^{-\lambda I(V^{\uparrow})} \right ] \right ) / \lambda^{1/\alpha} \leq \alpha / c^{1/\alpha}. \label{loglaplacefonctexpo2}
\end{eqnarray}

\end{prop}


\begin{proof}


Let us fix $y > 0$ for which we apply the decomposition \eqref{kesten2.0}. Let us denote 
\[ \mathcal{M} (\lambda) := - \log \left ( \mathbb{E} \left [ e^{-\lambda I(V^{\uparrow})} \right ] \right ) = \sum_{k \geq 0} - \log \left ( \mathbb{E} \left [ e^{-\lambda e^{-ky} A^y} \right ] \right ), \]
where the second equality comes from \eqref{kesten2.0} and from the fact that the sequence $(A^y_k)_{k \geq 0}$ is \textit{iid}. 

The asymptotic behavior of $\mathcal{M} (\lambda)$ is thus related to the asymptotic behavior of $- \log ( \mathbb{E} [ e^{-\lambda A^y} ] )$. Unfortunately, Proposition \ref{loglaplaceA} can not be applied simultaneously to all the terms of the sum defining $\mathcal{M} (\lambda)$. We separate this sum into three parts: a sum over a finite number of small indices for which we can apply \eqref{asyptlaplay} to each term, a sum over an infinite number of large indices that can be neglected, and a sum over the remaining indices (in finite number) that can be neglected too. 

We now prove the second point of the proposition. We assume that there is $\alpha \geq 1$ and a positive constant $c$ such that for all $\lambda$ large enough we have $\Psi_V(\lambda) \geq c \lambda^{\alpha}$, and we prove \eqref{loglaplacefonctexpo2}. Let us fix $\delta > 1$. Since $\Psi_{V^{\sharp}}(.) = \Psi_V(\kappa + .)$ and $\Phi_{V^{\sharp}} = \Psi_{V^{\sharp}}^{-1}$, we have for all $\lambda$ large enough that $\Phi_{V^{\sharp}}(\lambda) \leq \delta \lambda^{1/\alpha} / c^{1/\alpha}$. 

According to Proposition \ref{loglaplaceA} there exists $\lambda_{\delta} > 1$ such that \eqref{asyptlaplay} is satisfied with $\epsilon = \delta - 1$ for all $\lambda \geq \lambda_{\delta}$. By increasing $\lambda_{\delta}$ if necessary, we can also assume that $\Phi_{V^{\sharp}}(\lambda) \leq \delta \lambda^{1/\alpha} / c^{1/\alpha}$ for all $\lambda \geq \lambda_{\delta}$. 
Putting all this together we get 
\begin{eqnarray}
\lambda \geq \lambda_{\delta} \Rightarrow - \log \left ( \mathbb{E} \left [ e^{-\lambda A^y} \right ] \right ) \leq \delta^2 y \lambda^{1/\alpha} / c^{1/\alpha}. \label{loglaplacefonctexpo3}
\end{eqnarray}
%
%
Also, let us choose $M \in ]0, 1[$ small enough so that 
\begin{eqnarray}
\forall \lambda \in [0, M], \ 0 \leq - \log \left ( \mathbb{E} \left [ e^{-\lambda A^y} \right ] \right ) \leq 2 \lambda \mathbb{E} \left [ A^y \right ]. \label{loglaplacefonctexpo4}
\end{eqnarray}

For any $\lambda > \lambda_{\delta}$ we define $n_1(\lambda) := \lfloor \log(\lambda/\lambda_{\delta})/y \rfloor - 1$ and $n_2(\lambda) := \lfloor \log(\lambda/M)/y \rfloor$. {$n_2(\lambda) > n_1(\lambda)$ is guaranteed because we have $\lambda_{\delta} > 1 > M$. From the definition of $\mathcal{M} (\lambda)$ we can write, for $\lambda > \lambda_{\delta}$,} 
\begin{eqnarray}
\mathcal{M} (\lambda) = T_1(\lambda) + T_2(\lambda) + T_3(\lambda), \label{vsharpqueue08bonus}
\end{eqnarray}
with
\begin{align}
T_1(\lambda) & := \sum_{k=0}^{n_1(\lambda)} - \log \left ( \mathbb{E} \left [ e^{-\lambda e^{-ky} A^y} \right ] \right ), \ T_2(\lambda) := \sum_{k=n_1(\lambda) +1}^{n_2(\lambda)} - \log \left ( \mathbb{E} \left [ e^{-\lambda e^{-ky} A^y} \right ] \right ), \nonumber \\
T_3(\lambda) & := \sum_{k=n_2(\lambda) +1}^{+\infty} - \log \left ( \mathbb{E} \left [ e^{-\lambda e^{-ky} A^y} \right ] \right ). \label{loglaplacefonctexpo5}
\end{align}

From the definition of $n_1(\lambda)$, \eqref{loglaplacefonctexpo3} can be applied to each term of the sum defining $T_1(\lambda)$, we thus have
\[ T_1(\lambda) \leq \delta^2 y ( \lambda / c)^{1/\alpha} \sum_{k=0}^{n_1(\lambda)} e^{-k y /\alpha} = \delta^2 y ( \lambda / c)^{1/\alpha} \frac{1 - e^{-y (n_1(\lambda) + 1)/\alpha}}{1-e^{-y /\alpha}}. \]
We get that for $\lambda$ large enough 
\begin{eqnarray}
T_1(\lambda) \leq \delta^2 y ( \lambda / c)^{1/\alpha} \frac{1}{1-e^{-y /\alpha}}. \label{mino1erterme}
\end{eqnarray}
{Note that from the definitions of $n_1(\lambda)$ and $n_2(\lambda)$ we have that 
\begin{eqnarray}
n_2(\lambda) -n_1(\lambda) \leq \log(\lambda/M)/y - \log(\lambda/\lambda_{\delta})/y + 2 = \frac{2y + \log(\lambda_{\delta}/M)}{y} =: C_y \in ]0, +\infty[, \label{intermn_1n_21}
\end{eqnarray}
where the positivity is ensured by $\lambda_{\delta} > M$. Also, from the definitions of $n_1(\lambda)$ and $n_2(\lambda)$ we have 
\begin{eqnarray}
(n_1(\lambda) + 1)y \geq \log(\lambda/\lambda_{\delta}) - y, \ \ \ (n_2(\lambda) + 1)y \geq \log(\lambda/M). \label{intermn_1n_22}
\end{eqnarray}
Using the monotonicity of the Laplace transform, \eqref{intermn_1n_21}, and the first part of \eqref{intermn_1n_22} we get 
\begin{align}
0 \leq T_2(\lambda) & \leq -(n_2(\lambda) -n_1(\lambda)) \log \left ( \mathbb{E} \left [ e^{-\lambda e^{-(n_1(\lambda) + 1)y} A^y} \right ] \right ) \nonumber \\
& \leq - C_y \log \left ( \mathbb{E} \left [ e^{-\lambda_{\delta} e^{y} A^y} \right ] \right ) < +\infty. \label{cutsum2bonus}
\end{align}
From the definition of $n_2(\lambda)$, \eqref{loglaplacefonctexpo4} can be applied to each term of the sum defining $T_3(\lambda)$, we thus have} 
\begin{align}
0 \leq T_3(\lambda) \leq 2 \lambda \mathbb{E} \left [ A^y \right ]  \sum_{k=n_2(\lambda) +1}^{+\infty} e^{-ky} & = 2 \lambda e^{-y(n_2(\lambda) +1)} \mathbb{E} \left [ A^y \right ] /(1-e^{-y}) \nonumber \\
& \leq 2 M \mathbb{E} \left [ A^y \right ] /(1-e^{-y}) < +\infty, \label{cutsum3bonus}
\end{align}
{where we have used the second part of \eqref{intermn_1n_22} to pass from the first line to the second.} 

Putting \eqref{mino1erterme}, \eqref{cutsum2bonus}, and \eqref{cutsum3bonus} into \eqref{vsharpqueue08bonus} we obtain that for $\lambda$ large enough 
\[ \limsup_{\lambda \rightarrow +\infty} \mathcal{M} (\lambda) / \lambda^{1/\alpha} \leq \frac{\delta^2 y}{c^{1/\alpha} (1-e^{-y /\alpha})}. \]
Since, from the definition $\mathcal{M} (\lambda) = - \log ( \mathbb{E} [ e^{-\lambda I(V^{\uparrow})} ] )$ which does not depend on $\delta$ nor on $y$, we can let $\delta$ go to $1$ and then $y$ go to $0$ in the above expression. We obtain 
\[ \limsup_{\lambda \rightarrow +\infty} - \log \left ( \mathbb{E} \left [ e^{-\lambda I(V^{\uparrow})} \right ] \right ) / \lambda^{1/\alpha} \leq \alpha / c^{1/\alpha}. \]

For the first point of the proposition, we proceed exactly as we just did for the second point, using the lower bound of \eqref{asyptlaplay} instead of the upper bound. In particular $\sum_{k=0}^{n_1(\lambda)} e^{-k y /\alpha}$ from the above argument is {replaced} by $e^{-y /\alpha} \sum_{k=0}^{n_1(\lambda)} e^{-k y /\alpha}$. We finally get, for $y > 0$, $\delta \in ]0, 1[$ and $\lambda$ large enough, 
\[ T_1(\lambda) \geq \delta^2 y (\lambda/C)^{1/\alpha} e^{-y/\alpha} \frac{1 - e^{-(n_1(\lambda) + 1)y/\alpha}}{1 - e^{-y/\alpha}}. \]
Then, since $n_1(\lambda)$ converges toward $+\infty$ as $\lambda$ goes to infinity, $T_2(\lambda) \geq 0$ and $T_3(\lambda) \geq 0$, we get $\liminf_{\lambda \rightarrow +\infty} \mathcal{M} (\lambda) / \lambda^{1/\alpha} \geq \delta^2 y (\lambda/C)^{1/\alpha} e^{-y/\alpha} /(1 - e^{-y/\alpha})$. Letting $\delta$ go to $1$ and then $y$ go to $0$ we obtain 
\[ \liminf_{\lambda \rightarrow +\infty} - \log \left ( \mathbb{E} \left [ e^{-\lambda I(V^{\uparrow})} \right ] \right ) / \lambda^{1/\alpha} \geq \alpha / C^{\alpha}, \]
which terminates the proof. 


\end{proof}

We also study what happens when $\Psi_V$ has regular variation in the next proposition: 

\begin{prop} \label{loglaplacefonctexpocasregvar}

Assume that $\Psi_V$ has $\alpha$-regular variation at $+ \infty$ for some $\alpha \in ]1, 2]$. Then, 
\begin{eqnarray}
- \log \left ( \mathbb{E} \left [ e^{-\lambda I(V^{\uparrow})} \right ] \right ) \underset{\lambda \rightarrow +\infty}{\sim} \alpha \ \Phi_{V^{\sharp}}(\lambda). \label{loglaplacefonctexporegvar}
\end{eqnarray}

\end{prop}

\begin{proof}

The idea is to modify a little the proof of Proposition \ref{loglaplacefonctexpo}. Let us fix $y > 0$ and $\epsilon \in ]0, 1/\alpha[$. We keep the notation 
\[ \mathcal{M} (\lambda) := - \log \left ( \mathbb{E} \left [ e^{-\lambda I(V^{\uparrow})} \right ] \right ) = \sum_{k \geq 0} - \log \left ( \mathbb{E} \left [ e^{-\lambda e^{-ky} A^y} \right ] \right ). \] 
Since $\Psi_V$ has $\alpha$-regular variation, $\Phi_{V^{\sharp}}$ has $1/\alpha$-regular variation so, thanks to Karamata representation theorem we can write 
\begin{eqnarray}
\Phi_{V^{\sharp}}(\lambda) = \lambda^{\frac{1}{\alpha}} \ \exp \left ( c (\lambda) + \int_{1}^{\lambda} \frac{r(t)}{t} dt \right ). \label{karamata}
\end{eqnarray}
In this representation $c$ is a bounded measurable function that converges to a constant $c$ at $+\infty$, and $r$ is a bounded measurable function that converges to $0$ at $+\infty$. Let us chose $\lambda_{\epsilon} > 1$ large enough such that for all $\lambda \geq \lambda_{\epsilon}$: \eqref{asyptlaplay} is satisfied, $c - \epsilon \leq c (\lambda) \leq c + \epsilon$, and $|r(\lambda)| \leq \epsilon$. We also fix $M$ as in \eqref{loglaplacefonctexpo4} and for any $\lambda > \lambda_{\epsilon}$ we define $n_1(\lambda) := \lfloor \log(\lambda/{\lambda_{\epsilon}})/y \rfloor - 1$ and $n_2(\lambda) := \lfloor \log(\lambda/M)/y \rfloor$. For $\lambda > \lambda_{\epsilon}$, we then still have $\mathcal{M} (\lambda) = T_1(\lambda) + T_2(\lambda) + T_3(\lambda)$ with $T_1(\lambda), T_2(\lambda)$ and $T_3(\lambda)$ as in \eqref{loglaplacefonctexpo5}. 
Here again, from the definition of $n_1(\lambda)$, \eqref{asyptlaplay} can be applied to each term of the sum defining $T_1(\lambda)$, and then the representation \eqref{karamata}. We get 
\begin{align*}
\frac{T_1(\lambda)}{\Phi_{V^{\sharp}}(\lambda)} & \leq (1 + \epsilon) y \sum_{k=0}^{n_1(\lambda)} \frac{\Phi_{V^{\sharp}}(e^{-ky} \lambda)}{\Phi_{V^{\sharp}}(\lambda)} \\
& = (1 + \epsilon) y \sum_{k=0}^{n_1(\lambda)} \exp \left ( -\frac{ky}{\alpha} + c (e^{-ky} \lambda) - c (\lambda) - \int_{e^{-ky} \lambda}^{\lambda} \frac{r(t)}{t} dt \right ). 
\end{align*}
For $k \leq n_1(\lambda)$, we have $e^{-ky} \lambda \geq \lambda_{\epsilon}$ so, according to the choice of $\lambda_{\epsilon}$, we have $- 2 \epsilon \leq c (e^{-ky} \lambda) - c (\lambda) \leq 2 \epsilon$ and $\epsilon \log(\lambda / e^{-ky} \lambda) \geq \int_{e^{-ky} \lambda}^{\lambda} \frac{r(t)}{t} dt \geq - \epsilon \log(\lambda / e^{-ky} \lambda) = - \epsilon k y$. As a consequence, 
\begin{align*}
\frac{T_1(\lambda)}{\Phi_{V^{\sharp}}(\lambda)} & \leq (1 + \epsilon) y e^{2 \epsilon} \sum_{k=0}^{n_1(\lambda)} \exp \left ( - k \left ( \frac{y}{\alpha} - \epsilon y \right ) \right ) \\
& = (1 + \epsilon) y e^{2 \epsilon} \frac{1 - e^{-(n_1(\lambda) + 1)(y/\alpha - \epsilon y)}}{1 - e^{-y/\alpha + \epsilon y}} \leq \frac{(1 + \epsilon) y e^{2 \epsilon}}{1 - e^{-y/\alpha + \epsilon y}}. 
\end{align*}
Then, $T_2(\lambda)$ and $T_3(\lambda)$ are bounded as in the proof of Proposition \ref{loglaplacefonctexpo} so we get 
\[ \limsup_{\lambda \rightarrow +\infty} \mathcal{M} (\lambda) / \Phi_{V^{\sharp}}(\lambda) \leq \frac{(1 + \epsilon) y e^{2 \epsilon}}{1 - e^{-y/\alpha + \epsilon y}}. \]
Here again, the left hand side does not depend on $\epsilon$ nor on $y$ so we can let $\epsilon$ go to $0$ and then $y$ go to $0$ in the above expression. We obtain 
\[ \limsup_{\lambda \rightarrow +\infty} - \log \left ( \mathbb{E} \left [ e^{-\lambda I(V^{\uparrow})} \right ] \right ) / \Phi_{V^{\sharp}}(\lambda) \leq \alpha. \]
To prove a lower bound, we proceed exactly as we just did, using the lower bound of \eqref{asyptlaplay}. In particular $\sum_{k=0}^{n_1(\lambda)} \Phi_{V^{\sharp}}(e^{-ky} \lambda)/\Phi_{V^{\sharp}}(\lambda)$ from the above argument is {replaced} by 

\noindent $\sum_{k=1}^{n_1(\lambda) + 1} \Phi_{V^{\sharp}}(e^{-ky} \lambda)/\Phi_{V^{\sharp}}(\lambda)$. We finally get 
\[ \frac{T_1(\lambda)}{\Phi_{V^{\sharp}}(\lambda)} \geq (1 - \epsilon) y e^{-2 \epsilon - y/\alpha - \epsilon y} \frac{1 - e^{-(n_1(\lambda) + 1)(y/\alpha + \epsilon y)}}{1 - e^{-y/\alpha - \epsilon y}}. \]
Then, since $n_1(\lambda)$ converges toward $+\infty$ as $\lambda$ goes to infinity, $T_2(\lambda) \geq 0$ and $T_3(\lambda) \geq 0$, we get $\liminf_{\lambda \rightarrow +\infty} \mathcal{M} (\lambda) / \Phi_{V^{\sharp}}(\lambda) \geq (1 - \epsilon) y e^{-2 \epsilon - y/\alpha - \epsilon y} / (1 - e^{-y/\alpha - \epsilon y})$. Letting $\epsilon$ go to $0$ and then $y$ go to $0$ we obtain 
\[ \liminf_{\lambda \rightarrow +\infty} - \log \left ( \mathbb{E} \left [ e^{-\lambda I(V^{\uparrow})} \right ] \right ) / \Phi_{V^{\sharp}}(\lambda) \geq \alpha, \]
which terminates the proof. 

\end{proof}

\subsection{Tail at $0$ of $I(V^{\uparrow})$: proof of Theorems \ref{queuefonctexpo} and \ref{varreg}}


We now prove Theorems \ref{queuefonctexpo} and \ref{varreg} by using Propositions \ref{loglaplacefonctexpo} and \ref{loglaplacefonctexpocasregvar} together with the deep link that exists between the left tail of a random variable and the asymptotic behavior of its Laplace transform. 

\begin{proof} of Theorem \ref{queuefonctexpo}

We assume that there is $\alpha > 1$ and a positive constant $C$ such that for all $\lambda$ large enough we have $\Psi_V(\lambda) \leq C \lambda^{\alpha}$. We fix $\delta \in ]0, 1[$. According to the first point of Proposition \ref{loglaplacefonctexpo}, there exists $\lambda_{\delta} > 0$ such that for all $\lambda > \lambda_{\delta}$ we have 
\begin{eqnarray}
\mathbb{E} \left [ e^{-\lambda I(V^{\uparrow})} \right ] \leq \exp \left ( - \delta^{1-1/\alpha} \alpha \lambda^{1/\alpha} / C^{1/\alpha} \right ). \label{queuefonctexpo1.1}
\end{eqnarray}
Let us fix $x \in ]0, \delta^{(\alpha-1)/\alpha} C^{-1/\alpha} \lambda_{\delta}^{(1-\alpha)/\alpha}[$. Using Markov's inequality and \eqref{queuefonctexpo1.1} we get that for any $\lambda > \lambda_{\delta}$, 
\[ \mathbb{P} \left ( I(V^{\uparrow}) \leq x \right ) \leq e^{\lambda x} \mathbb{E} \left [ e^{-\lambda I(V^{\uparrow})} \right ] \leq \exp \left ( \lambda x - \delta^{1-1/\alpha} \alpha \lambda^{1/\alpha} / C^{1/\alpha} \right ). \]
{Let $h : x \mapsto \delta C^{-1/(\alpha - 1)} x^{-\alpha/(\alpha - 1)}$. $h$ is decreasing on $\mathbb{R}_+$ since $\alpha > 1$, moreover 

\noindent $h(\delta^{(\alpha-1)/\alpha} C^{-1/\alpha} \lambda_{\delta}^{(1-\alpha)/\alpha})=\lambda_{\delta}$. Therefore, since $x \in ]0, \delta^{(\alpha-1)/\alpha} C^{-1/\alpha} \lambda_{\delta}^{(1-\alpha)/\alpha}[$, we have $h(x) \geq \lambda_{\delta}$ so, in the above inequality, we can replace $\lambda$ by $h(x) = \delta C^{-1/(\alpha - 1)} x^{-\alpha/(\alpha - 1)}$.} We obtain 
\[ \mathbb{P} \left ( I(V^{\uparrow}) \leq x \right ) \leq \exp \left ( - \delta (\alpha - 1) / (C x)^{1/(\alpha - 1)} \right ), \]
which is \eqref{queuefonctexpo1}. 

We now prove the second point. Assume that there is $\alpha > 1$ and a positive constant $c$ such that for all $\lambda$ large enough we have $\Psi_V(\lambda) \geq c \lambda^{\alpha}$. We fix $\delta > 1$ and $r \in ]0, 1[$. According to the second point of Proposition \ref{loglaplacefonctexpo}, there exists $\lambda_{\delta} > 0$ such that for all $\lambda > \lambda_{\delta}$ we have 
\begin{eqnarray}
\mathbb{E} \left [ e^{-\lambda I(V^{\uparrow})} \right ] \geq \exp \left ( - \delta^{r(\alpha - 1)/\alpha} \alpha \lambda^{1/\alpha} / c^{1/\alpha} \right ). \label{queuefonctexpo1.2}
\end{eqnarray}
Let us fix $x \in ]0, \delta^{(\alpha - 1)/\alpha} \alpha c^{-1/\alpha} \lambda_{\delta}^{(1-\alpha)/\alpha}[$. Using \eqref{queuefonctexpo1.2} we get that for any $\lambda > \lambda_{\delta}$, 
\begin{align*}
\exp \left ( - \delta^{r(\alpha - 1)/\alpha} \alpha \lambda^{1/\alpha} / c^{1/\alpha} \right ) & \leq \mathbb{E} \left [ e^{- \lambda I(V^{\uparrow})} \right ] = \mathbb{E} \left [ e^{- \lambda I(V^{\uparrow})} \mathds{1}_{ \left \{ I(V^{\uparrow}) \leq x \right \} } \right ] + \mathbb{E} \left [ e^{- \lambda I(V^{\uparrow})} \mathds{1}_{ \left \{ I(V^{\uparrow}) > x \right \} } \right ] \\
& \leq \mathbb{P} \left ( I(V^{\uparrow}) \leq x \right ) + \exp \left (-\lambda x \right ), 
\end{align*}
so we get
\[ \mathbb{P} \left ( I(V^{\uparrow}) \leq x \right ) \geq \exp \left ( - \delta^{r(\alpha - 1)/\alpha} \alpha \lambda^{1/\alpha} / c^{1/\alpha} \right ) - \exp \left (-\lambda x \right ). \]
{Let $\tilde h : x \mapsto \delta \alpha^{\alpha /(\alpha - 1)} c^{-1/(\alpha - 1)} x^{-\alpha/(\alpha - 1)}$. $\tilde h$ is decreasing on $\mathbb{R}_+$ since $\alpha > 1$, moreover $\tilde h(\delta^{(\alpha - 1)/\alpha} \alpha c^{-1/\alpha} \lambda_{\delta}^{(1-\alpha)/\alpha}) = \lambda_{\delta}$. Therefore, since $x \in ]0, \delta^{(\alpha - 1)/\alpha} \alpha c^{-1/\alpha} \lambda_{\delta}^{(1-\alpha)/\alpha}[$, we have $\tilde h(x) \geq \lambda_{\delta}$ so, in the above inequality, we can replace $\lambda$ by $\tilde h(x) = \delta \alpha^{\alpha /(\alpha - 1)} c^{-1/(\alpha - 1)} x^{-\alpha/(\alpha - 1)}$.} We obtain 
\[ \mathbb{P} \left ( I(V^{\uparrow}) \leq x \right ) \geq \exp \left ( - \delta^{(1+r(\alpha - 1))/\alpha} \alpha^{\alpha/(\alpha-1)} / (c x)^{1/(\alpha - 1)} \right ) - \exp \left ( - \delta \alpha^{\alpha/(\alpha-1)} / (c x)^{1/(\alpha - 1)} \right ). \]
Since $\delta > \delta^{(1+r(\alpha - 1))/\alpha}$, the second term converges to $0$ faster than the first one when $x$ goes to $0$. We thus get, for $x$ small enough, 
\[ \mathbb{P} \left ( I(V^{\uparrow}) \leq x \right ) \geq \frac1{2} \exp \left ( - \delta^{(1+r(\alpha - 1))/\alpha} \alpha^{\alpha/(\alpha-1)} / (c x)^{1/(\alpha - 1)} \right ) \geq \exp \left ( - \delta \alpha^{\alpha/(\alpha-1)} / (c x)^{1/(\alpha - 1)} \right ). \]
The last inequality is true for $x$ small enough and comes from the fact that $\delta > \delta^{(1+r(\alpha - 1))/\alpha}$. The above is precisely \eqref{queuefonctexpo2}. 

\end{proof}

\begin{proof} of Theorem \ref{varreg}

We assume that $\Psi_V$ has $\alpha$-regular variation at $+ \infty$ for some $\alpha \in ]1, 2]$. From Proposition \ref{loglaplacefonctexpocasregvar} we deduce that 
\[ - \log \left ( \mathbb{E} \left [ e^{-\lambda I(V^{\uparrow})} \right ] \right ) \underset{\lambda \rightarrow +\infty}{\sim} \alpha \ \Phi_{V^{\sharp}}(\lambda). \]
Then, since $\Phi_{V^{\sharp}}$ is the inverse function of $\Psi_V(\kappa + .)$, it has $1/\alpha$-regular variation at $+\infty$ and the application of De Bruijn's Theorem (see Theorem 4.12.9 in \cite{regvar}) yields the result. 

\end{proof}

%
%
%


\subsection{Connection between $I(V^{\uparrow})$ and $I(V^{\sharp})$: Proof of Propositions \ref{vposqueue0decond} and \ref{encadrementpatiesavov}} \label{connection}

In this subsection, we assume that $V$ does not oscillate (so that $I(V^{\sharp}) < +\infty$) and we relate 
$I(V^{\uparrow})$ and $I(V^{\sharp})$. This allows to prove Proposition \ref{vposqueue0decond} and to combine it with Theorem 5.24 of \cite{Patieref10} to prove Proposition \ref{encadrementpatiesavov}. 
Note that in particular, when $V$ drifts to $+\infty$, this is only a {comparison} between 
$I(V^{\uparrow})$ and $I(V)$. 

\begin{prop} \label{vpospartoffonctclassic}

If $V$ does not oscillates we have 
\begin{eqnarray}
I(V^{\sharp}) \overset{\mathcal{L}}{=} S_T + I(V^{\uparrow}), \label{factconvo}
\end{eqnarray}
where the two terms of the right hand side are independent, and $S_T$ is as in Lemma \ref{=subexpbisbis}. 
\end{prop}

Thanks to Proposition \ref{vpospartoffonctclassic}, we can explain heuristically why the functionals $I(V^{\uparrow})$ and $I(V^{\sharp})$ have similar left tails but different right tails: 

\begin{remarque}
In \eqref{factconvo}, the term $S_T$ contains the contributions of the excursions of $V^{\sharp}$ that meet the half-line of negative real numbers, it is thus heavy tailed. On the other hand, the right tail of $I(V^{\uparrow})$ is given by Theorem \ref{finiteandexpomoments} and is at most exponential. This is why only the contribution of the term $S_T$ is relevant regarding the right tail of $I(V^{\sharp})$ and the two functionals have different right tails. For the left tails, the one of $S_T$ is at least the left tail of an exponential random variable (see Lemma \ref{subexp} below) while the left tail of $I(V^{\uparrow})$ is lighter according to our preceding results. This is why the left tail of $I(V^{\sharp})$ is mainly determined by the term $I(V^{\uparrow})$ in \eqref{factconvo}. 
\end{remarque}


\begin{proof} of Proposition \ref{vpospartoffonctclassic}

We write 
\begin{align} I(V^{\sharp}) & = \int_0^{\mathcal{R}(V^{\sharp}, 0)} e^{- V^{\sharp}(t)} dt + \int_{\mathcal{R}(V^{\sharp}, 0)}^{+ \infty} e^{- V^{\sharp}(t)} dt \nonumber \\
& = \int_0^{\mathcal{R}(V^{\sharp}, 0)} e^{- V^{\sharp}(t)} dt + \int_0^{+ \infty} e^{- V^{\sharp}(t + \mathcal{R}(V^{\sharp}, 0))} dt \nonumber \\
& \overset{\mathcal{L}}{=} \int_0^{\mathcal{R}(V^{\sharp}, 0)} e^{- V^{\sharp}(t)} dt + I(V^{\uparrow}). \label{vposqueue0decond1}
\end{align}
We have used Lemma \ref{dernierpassage2} for the last equality in which the two terms $\int_0^{\mathcal{R}(V^{\sharp}, 0)} e^{- V^{\sharp}(t)} dt$ and $I(V^{\uparrow})$ are independent. 

Now, thanks to Lemma \ref{=subexpbisbis}, the term $\int_0^{\mathcal{R}(V^{\sharp}, 0)} e^{- V^{\sharp}(t)} dt $ has the same law as $S_T$ with $S_T$ as in the lemma. This achieves the proof. 


\end{proof}

Similarly to the proof of Proposition \ref{decomposition}, we decomposed $I(V^{\sharp})$ as the sum of two independent random variables, one having the same law as a subordinator stopped at an independent exponential time and the other having the same law as $I(V^{\uparrow})$. We now need a lemma about the asymptotic tail at $0$ of a subordinator stopped at an independent exponential time. 

\begin{lemme} \label{subexp}

Let $S$ be a subordinator and $T$ an independent exponential random variable, there exists a positive constant $c$ such that for all $x$ small enough
\[ \mathbb{P} \left ( S_T < x \right ) \geq cx. \]

\end{lemme}

\begin{proof}

We prove in fact a stronger result: the function $x \mapsto \mathbb{P} \left ( S_T < x \right )$ is sub-additive, that is, 
\begin{eqnarray}
\forall x, y \geq 0, \ \mathbb{P} \left ( S_T < x + y \right ) \leq \mathbb{P} \left ( S_T < x \right ) + \mathbb{P} \left ( S_T < y \right ). \label{subexp0.1}
\end{eqnarray}
The lemma follows easily once \eqref{subexp0.1} is proved. Recall from the introduction the notation $\tau(S, h+)$ for $\tau(S, [h, +\infty[)$. Let $x, y > 0$ (the case where $x=0$ or $y=0$ is obvious), we have 
\begin{align*} \mathbb{P} \left ( S_T < x + y \right ) & = \mathbb{P} \left ( S_T < x \right ) + \mathbb{P} \left ( T \geq \tau(S, x+), \ T < \tau(S, (x+y)+) \right ) \\
& \leq \mathbb{P} \left ( S_T < x \right ) + \mathbb{P} \left ( T \geq \tau(S, x+), \ T < \tau(S, (S_{\tau(S, x+)}+y) + )\right ). \end{align*}
For the inequality we have used that $S_{\tau(S, x+)} \geq x$ almost surely. From the characteristic property of the exponential distribution and the Markov property applied to $S$ at time $\tau(S, x+)$, the latter is equal to 
\[ = \mathbb{P} \left ( S_T < x \right ) + \mathbb{P} \left ( T \geq \tau(S, x+) \right ) \times \mathbb{P} \left ( T < \tau(S, y+) \right ). \]
Since $\mathbb{P} \left ( T \geq \tau(S, x+) \right ) \leq 1$ and $\mathbb{P} \left ( T < \tau(S, y+) \right ) =  \mathbb{P} \left ( S_T < y \right )$ we obtain \eqref{subexp0.1}. 


\end{proof}

We can now prove Proposition \ref{vposqueue0decond}. 

\begin{proof} of Proposition \ref{vposqueue0decond}. 

We are in the case where $V$ drifts to $+\infty$ so $V^{\sharp} = V$. Applying Proposition \ref{vpospartoffonctclassic} we get 
\begin{eqnarray}
\mathbb{P} \left ( I(V) \leq x \right ) \leq \mathbb{P} \left ( I(V^{\uparrow}) \leq x \right ). \label{vposqueue0decond2}
\end{eqnarray}

For $\epsilon > 0$, using the equality in law of Proposition \ref{vpospartoffonctclassic} and Lemma \ref{subexp} we obtain
\[ \mathbb{P} \left ( I(V) \leq (1+\epsilon)x \right ) \geq \mathbb{P} \left ( I(V^{\uparrow}) \leq x \right ) \times \mathbb{P} \left ( S_T \leq \epsilon x \right ) \geq c \epsilon x \mathbb{P} \left ( I(V^{\uparrow}) \leq x \right ), \]
for an appropriate constant $c > 0$ and $x$ small enough. Combining with $(\ref{vposqueue0decond2})$ we get the result. 

\end{proof}

\begin{proof} of Proposition \ref{encadrementpatiesavov} 

If $V$ does not oscillate and has unbounded variation, $V^{\sharp}$ drifts to $+\infty$, has Laplace exponent $\Psi_{V^{\sharp}} = \Psi_V(\kappa + .)$, and $(V^{\sharp})^{\uparrow} = V^{\uparrow}$. Note that, since $I(V^{\sharp})$ is self-decomposable, it is {unimodal}. Then, the mode is not $0$ (because of the continuity of the density or because of the argument of Remark \ref{estimdensite} which is valid since Remark \ref{majouniverselle} is true for $I(V^{\sharp})$). As a consequence, the density of $I(V^{\sharp})$ is non-decreasing on a neighborhood of $0$. Let us denote by $m^{\sharp}$ this density so that for any $x \geq 0$, $\mathbb{P} ( I(V^{\sharp}) \leq x ) = \int_{0}^{x} m^{\sharp}(z) dz$. Let us fix $\delta > 1$. The combination of Proposition \ref{vposqueue0decond} (applied to $V^{\sharp}$ and $\epsilon = \delta - 1$) with the non-decrease of $m^{\sharp}$ gives for all $x$ small enough 
\[ \mathbb{P} \left ( I(V^{\uparrow}) \leq x \right ) \geq \mathbb{P} \left ( I(V^{\sharp}) \leq x \right ) \geq \int_{x/\delta}^{x} m^{\sharp}(z) dz \geq \left ( 1 - \frac1{\delta} \right ) x \ m^{\sharp} \left ( \frac{x}{\delta} \right ), \]
and  
\[ \mathbb{P} \left ( I(V^{\uparrow}) \leq x \right ) \leq \frac1{c (\delta - 1) x} \mathbb{P} \left ( I(V^{\sharp}) \leq \delta x \right ) = \frac1{c (\delta - 1) x} \int_{0}^{\delta x} m^{\sharp}(z) dz \leq \frac{\delta}{c (\delta - 1)} m^{\sharp} (\delta x). \]
The constant $c$ above is as in Proposition \ref{vposqueue0decond}. Applying Theorem 5.24 of \cite{Patieref10} to $m^{\sharp}(x/\delta)$ and $m^{\sharp} (\delta x)$ and combining with the above estimates we get the result. 

\end{proof}

\begin{remarque}

In Theorem 5.24 of \cite{Patieref10} is given an equivalent for the left tail of $I(V)$ (not an inequality). 
However, it would not be possible to state this estimate for the left tail of $I(V^{\uparrow})$, since the convolution factor $S_T$ in Proposition \ref{vpospartoffonctclassic} changes the {behavior} of the left tail. This is why Propositions \ref{vpospartoffonctclassic} and \ref{vposqueue0decond} only allow to transfer inequalities. 
Then, if one wants to deduce from Theorem 5.24 of \cite{Patieref10} an equivalent for the $\log$ of the left tail of $I(V^{\uparrow})$, one has to be aware that in Proposition \ref{vpospartoffonctclassic}, the upper bound for $\mathbb{P} ( I(V^{\uparrow}) \leq x )$ is given in term of $\mathbb{P} ( I(V) \leq (1+\epsilon)x )$ and not $\mathbb{P} ( I(V) \leq x )$. Due to this, the asymptotic of the $\log$ of the left tail of $I(V)$ can be transferred to $I(V^{\uparrow})$ only in the case of regular variation, but this case is already treated by Theorem \ref{varreg} in full generality. 

\end{remarque}


\section{Smoothness of the density, proof of Theorem \ref{asin13}} \label{densit}

$I(V^{\uparrow})$ is positive and infinitely divisible so it has no Brownian component, its L\'evy measure does not charge the negative half-line and has a structure of bounded variation. As a consequence, in the terminology of \cite{satoyamazato}, $I(V^{\uparrow})$ is of type $I$ and $\int_0^1 k(u) du < +\infty$. 

We assume that $V$ has unbounded variation. Let $\gamma_0$ be as in the terminology of \cite{satoyamazato}. We first prove that $\gamma_0 = 0$ which is equivalent to proving that the support of $I(V^{\uparrow})$ contains $0$. Let us apply Proposition \ref{decomposition} for some $y>0$. In view of \eqref{kesten2.0} and of the definition of $A^y$, we only need to prove that both the support of $\int_0^{\tau(V^{\uparrow}, y)} e^{- V^{\uparrow}(t)} dt$ and the support of $S_T$ contain $0$. 

First, note that, since $V$ (and therefore $V^{\sharp}$) has unbounded variation, Corollary VII.5 in \cite{Bertoin} yields that $\Psi_{V^{\sharp}}(\lambda)/\lambda \longrightarrow_{\lambda \rightarrow +\infty} +\infty$ so $\Phi_{V^{\sharp}}(\lambda)/\lambda \longrightarrow_{\lambda \rightarrow +\infty} 0$. $\Phi_{V^{\sharp}}$ is the Laplace exponent of the subordinator $\tau(V^{\sharp}, .)$ which has, therefore, no drift component. As a consequence $0$ belongs to the support of $\tau(V^{\sharp}, y)$ and, because of \eqref{nouvellepreuve1}, to the support of $\int_0^{\tau(V^{\uparrow}, y)} e^{- V^{\uparrow}(t)} dt$. 

Then, Lemma \ref{subexp} proves that the distribution of $S_T$ charges every neighborhood of $0$ so, by definition, the support of $S_T$ contains $0$. We thus have $\gamma_0 = 0$. 

If $I(V^{\uparrow})$ was of type $I_1, I_2, I_3, I_4$ or $I_5$ (still in the terminology of \cite{satoyamazato}), we could apply Theorem 1.6 of \cite{satoyamazato} with $n = 0$ and replace $\gamma_0$ by $0$ in there. The estimate obtained would be in conflict with Remarks \ref{estimdensite} and \ref{majouniverselle} which guarantee the existence of a constant $K$ such that the density of $I(V^{\uparrow})$ is bounded by $e^{-K/x}$ for $x$ small enough. We thus deduce that $I(V^{\uparrow})$ is not of type $I_1, I_2, I_3, I_4$ nor $I_5$. 
Since it is neither of type $I_7$ (because $\int_0^1 k(u) du < +\infty$), we get that $I(V^{\uparrow})$ is of type $I_6$. 

The fact that the density of $I(V^{\uparrow})$ belongs to the Schwartz space is then a consequence of the existence of moments of any positive order (by Theorem \ref{finiteandexpomoments}), the type $I_6$, and of the following Lemma \ref{schwartzselfdec}. 
The derivatives of all orders of this density converge to $0$ when $x$ goes to $0$ since this density is of class $\mathcal{C}^{\infty}$ and null on $]-\infty, 0[$. 

The $\log$-concavity of $x \mapsto \mathbb{P} (I(V^{\uparrow}) \leq x)$ on $]0, +\infty[$ is a consequence of the type $I_6$ and of Lemma 5.2 in \cite{satoyamazato}. 

It only remains to state and justify Lemma \ref{schwartzselfdec}. Our argument to justify it extends a little the proof of Theorem 28.4 of \cite{Sato}. 

\begin{lemme} \label{schwartzselfdec}

Let $X$ be a self-decomposable distribution of type $I_6$, in the terminology of \cite{satoyamazato}. If $X$ admits finite moments of any positive order, then its density belongs to the Schwartz space. 

\end{lemme}

\begin{proof}

Since $X$ is self-decomposable of type {$I_6$}, its characteristic {function} can be represented as 
\[ \mathbb{E} \left [ e^{i \xi X} \right ] = \exp \left ( \int_0^{+\infty} (e^{i \xi z} - 1) \frac{k(z)}{z} dz \right ) =: \varphi_X(\xi). \]
Without loss of generality we have assumed that the drift component of $X$ is null. In the above expression, $k$ is a right-continuous non-increasing function such that $k(0+) = +\infty$. Let us fix an arbitrary $m > 0$ and $\epsilon_m$ such that $k(z) \geq m$ for all $z \in [0, \epsilon_m]$. It is plain from the L\'evy-Kintchine formula that we can define an infinitely divisible random variable $Y_m$ with characteristic {function} 
\[ \mathbb{E} \left [ e^{i \xi Y_m} \right ] = \exp \left ( \int_0^{\epsilon_m} (e^{i \xi z} - 1) \frac{m}{z} dz \right ) =: A_m(\xi), \]
and an infinitely divisible random variable $Z_m$ with characteristic {function} 
\[ \mathbb{E} \left [ e^{i \xi Z_m} \right ] = \exp \left ( \int_0^{+\infty} (e^{i \xi z} - 1) \frac{k(z) - m \mathds{1}_{[0, \epsilon_m]}(z)}{z} dz \right ) =: B_m(\xi). \]
If we assume $Y_m$ and $Z_m$ to be independent we have clearly $\varphi_X(\xi) = A_m(\xi) \times B_m(\xi)$, so $X \overset{\mathcal{L}}{=} Y_m + Z_m$. We first study the asymptotic behavior of the successive derivatives of $A_m$ and $B_m$ to provide a bound for $\varphi_X(\xi)$. 

$X$ admits finite moments of any positive order, $Y_m$ and $Z_m$ are positive and satisfy $Y_m + Z_m \overset{\mathcal{L}}{=} X$, we thus deduce that $Y_m$ and $Z_m$ admit finite moments of any positive order. As a consequence $A_m$ and $B_m$ are of class $C^{\infty}$ and all their derivatives are bounded. In particular 
\begin{eqnarray}
\forall n \geq 0, \ \left | B_m^{(n)}(\xi) \right | = \underset{|\xi| \rightarrow +\infty}{\mathcal{O}} (1). \label{rapder2}
\end{eqnarray}

Then, note that 
\[ A_m'(\xi) = im \int_0^{\epsilon_m} e^{i \xi z} dz \times A_m(\xi) = \frac{m}{\xi} (e^{i \xi \epsilon_m} - 1) \times A_m(\xi). \]
Let us denote $C_m(\xi) := m(e^{i \xi \epsilon_m} - 1)/\xi$. It is easy to see that the derivatives of any order of $C_m$ converge to $0$ when $|\xi|$ goes to $+\infty$. Let us justify by strong induction that for any $n \geq 1$ we have 
\begin{eqnarray}
A_m^{(n)}(\xi) / A_m(\xi) \underset{|\xi| \rightarrow +\infty}{\longrightarrow} 0. \label{rapder}
\end{eqnarray}
When $n =1$, this is equivalent to the convergence to $0$ of $C_m(\xi)$, which is already known. Let us now fix $l \geq 1$ and assume that \eqref{rapder} is true for all $n \leq l$. Applying the Leibniz formula to derive $l$ times the product $A_m'(\xi) = C_m(\xi) \times A_m(\xi)$, we get 
\[ A_m^{(l+1)}(\xi) / A_m(\xi) = \sum_{j=0}^{l} \ C_l^j \ C_m^{(j)}(\xi) \ A_m^{(l-j)}(\xi) / A_m(\xi). \]
The latter converges to $0$ when $|\xi|$ goes to $+\infty$ by the induction hypothesis and the fact that the derivatives of any order of $C_m$ converge to $0$ at infinity. Thus the induction is proved. We can see that $|A_m(\xi)| = \mathcal{O}_{|\xi| \rightarrow +\infty} (|\xi|^{-m})$ (see for example the proof of Theorem 28.4 in \cite{Sato}). Combining with \eqref{rapder} this yields 
\begin{eqnarray}
\forall n \geq 0, \ \left | A_m^{(n)}(\xi) \right | = \underset{|\xi| \rightarrow +\infty}{\mathcal{O}} (|\xi|^{-m}). \label{rapder1}
\end{eqnarray}

From \eqref{rapder2} and \eqref{rapder1} applied with $n=0$ we know that $|\varphi_X(\xi)| = \mathcal{O}_{|\xi| \rightarrow +\infty} (|\xi|^{-m})$. For any $n \geq 1$ we can apply the Leibniz formula to derive $n$ times the product $\varphi_X(\xi) = A_m(\xi) \times B_m(\xi)$ and we get 
\[ \varphi_X^{(n)}(\xi) = \sum_{j=0}^{n} \ C_n^j \ A_m^{(j)}(\xi) \ B_m^{(n-j)}(\xi). \]
According to \eqref{rapder2} and \eqref{rapder1} we deduce that $\forall n \geq 0, \ |\varphi_X^{(n)}(\xi)| = \mathcal{O}_{|\xi| \rightarrow +\infty} (|\xi|^{-m})$. Then, since $m$ is arbitrary, this implies that $\varphi_X$ belongs to the Schwartz space. Since this space is stable by Fourier transform, $X$ admits a density in the Schwartz space. 

\end{proof}

\section{The spectrally positive case} \label{brief}

We now study the exponential functional of $Z^{\uparrow}$, where $Z$ is a spectrally positive L\'evy process drifting to $+\infty$ with unbounded variation. 

%
 

Recall that, for any $x \geq 0$, $Z^{\uparrow}_x$ must be seen as $Z$ conditioned to stay positive and starting from $x$. Note that, since $Z$ converges almost surely to infinity, for $x > 0$, $Z^{\uparrow}_x$ is only $Z_x$ conditioned in the usual sense to remain positive. 


\subsection{Finiteness, exponential moments: Proof of Theorem \ref{vneglapl}}

The idea is that adding a small term of negative drift to $Z^{\uparrow}$ does not change its convergence to $+ \infty$. Thus, $Z^{\uparrow}$ is ultimately greater than a deterministic linear function for which the exponential functional is defined and deterministically bounded. The key point is then to control the time taken by $Z^{\uparrow}$ to become greater than the linear function once and for good. We start with the following lemma. 

\begin{lemme} \label{vneglapltpslin}

For any $y>0$, there exists $\epsilon > 0$ and positive constants $c_1$ and $c_2$ such that 
\[ \forall s > 0, \ \mathbb{P} \left ( \mathcal{R}(Z^{\uparrow}_y(.) - (y + \epsilon \ .), {]-\infty, 0]}) > s \right ) \leq c_1 e^{-c_2 s}. \]

\end{lemme}

\begin{proof}

We fix $y > 0$. A spectrally negative L\'evy process $X$ drifts to $-\infty$ if and only if $\mathbb{E}[X(1)] < 0$ (see for example Corollary VII.2 in \cite{Bertoin}). $Z$ is a spectrally positive L\'evy process drifting to $+\infty$ so taking the dual we get $\mathbb{E}[Z(1)] > 0$. Now $\mathbb{E}[(Z-\epsilon.)(1)] = \mathbb{E}[Z(1)] -\epsilon$ which is positive for $\epsilon$ chosen small enough. 
This implies that $Z - \epsilon .$ is also a spectrally positive L\'evy process that drifts to $+\infty$ and which is not a subordinator. We have
\begin{align*}
\mathbb{P} \left ( \mathcal{R}(Z^{\uparrow}_y(.) - (y + \epsilon \ .), {]-\infty, 0]}) > s \right ) & = \mathbb{P} \left ( \inf_{t \in [s, + \infty[} Z^{\uparrow}_y(t) - (y + \epsilon t) \leq 0 \right ) \\
& = \mathcal{C} \ \mathbb{P} \left ( \inf_{t \in [s, + \infty[} Z_y(t) - (y + \epsilon t) \leq 0, \inf_{[0, +\infty[} Z_y > 0 \right ). 
\end{align*}
We have set $\mathcal{C} := 1/\mathbb{P} ( \inf_{[0, +\infty[} Z_y > 0 ) = 1/( 1-e^{-\kappa_Z y} )$. The last equality comes from the fact that $Z^{\uparrow}_y$ is only $Z_y$ conditioned to stay positive in the usual sense.  Now, noting that $Z_y \overset{\mathcal{L}}{=} y + Z$, we bound the above quantity by
\begin{align}
& \mathcal{C} \ \mathbb{P} \left ( \inf_{t \in [s, + \infty[} Z(t) - \epsilon t \leq 0 \right ) = \mathcal{C} \ \mathbb{P} \left ( \sup_{t \in [s, + \infty[} -Z(t) + \epsilon t \geq 0 \right ) \nonumber \\ 
\leq \ & \mathcal{C} \ \mathbb{P} \left ( \sup_{t \in [s, + \infty[} -Z(t) + \epsilon t \geq 0, \ -Z(s) + \epsilon s \leq 0 \right ) + \mathcal{C} \ \mathbb{P} \left ( -Z(s) + \epsilon s \geq 0 \right ) \nonumber \\
= \ & \mathcal{C} \ \mathbb{P} \left ( \sup_{t \in [0, + \infty[} -Z^s(t) + \epsilon t \geq -\left ( -Z(s) + \epsilon s \right ), -Z(s) + \epsilon s \leq 0 \right ) + \mathcal{C} \ \mathbb{P} \left ( -Z(s) + \epsilon s > 0 \right ). \label{vneglapl1}
\end{align}
From the stationarity and independence of increments, the process $-Z^s + \epsilon.$ is equal in law to $-Z + \epsilon.$ and independent from $-Z(s) + \epsilon s$. From Corollary VII.2 in \cite{Bertoin}, the supremum over $[0, + \infty[$ of the process $-Z^s + \epsilon.$ follows an exponential distribution with parameter $\alpha$, where $\alpha$ is the non-trivial zero of $\Psi_{-Z + \epsilon.}$. From this, combined with the independence from $-Z(s) + \epsilon s$, $(\ref{vneglapl1})$ becomes
\begin{align*}
& \mathcal{C} \ \mathbb{E} \left ( e^{ \alpha \left ( -Z(s) + \epsilon s \right )} \mathds{1}_{ \left \{ -Z(s) + \epsilon s \leq 0 \right \} } \right ) + \mathcal{C} \ \mathbb{P} \left ( -Z(s) + \epsilon s > 0 \right ) \\
\leq \ & \mathcal{C} \ \mathbb{E} \left ( e^{ \alpha \left ( -Z(s) + \epsilon s \right ) /2} \mathds{1}_{ \left \{ -Z(s) + \epsilon s \leq 0 \right \} } \right ) + \mathcal{C} \ \mathbb{P} \left ( e^{ \alpha \left ( -Z(s) + \epsilon s \right )/2} > 1 \right ). \end{align*}
We have used the decreases of the negative exponential for the first term and composed by the function $x \mapsto \exp(\frac{\alpha}{2} x)$ in the probability of the second term. The above is less than 
\[ \mathcal{C} \ \mathbb{E} \left ( e^{ \frac{\alpha}{2} \left ( -Z(s) + \epsilon s \right )} \right ) + \mathcal{C} \ \mathbb{P} \left ( e^{\frac{\alpha}{2} \left ( -Z(s) + \epsilon s \right ) } > 1 \right ) \leq \ 2 \mathcal{C} \ \mathbb{E} \left ( e^{ \alpha/2 \left ( -Z(s) + \epsilon s \right )} \right ), \]
where we have used Markov's inequality in the second term. 
\[ = 2 \mathcal{C} \ e^{s \Psi_{-Z + \epsilon.}(\alpha/2)}. \]
Since $\Psi_{-Z + \epsilon.}$ is negative on $]0, \alpha[$, we get the result with $c_1 = 2\mathcal{C}$ and $c_2 = -\Psi_{-Z + \epsilon.}(\alpha/2)$. 

%
%
%
%
%
%

\end{proof}

We can now prove Theorem \ref{vneglapl}. 

\begin{proof} of Theorem \ref{vneglapl}

We fix $y > 0$. In order to get $Z^{\uparrow}$ from $Z^{\uparrow}_y$, we use the decomposition given by Theorem 25 in \cite{Doney}. Let $m_y$ be the point where the process $Z^{\uparrow}_y$ reaches its infimum, $m_y := \sup \{ s \geq 0, \ Z^{\uparrow}_y(s-) \wedge Z^{\uparrow}_y(s) = \inf_{[0, +\infty[} Z^{\uparrow}_y \}$. Note that from the absence of negative jumps the infimum is always reached at least at $m_y-$ so $Z^{\uparrow}_y(m_y-) = \inf_{[0, +\infty[} Z^{\uparrow}_y$. Moreover, since $Z$ is regular for $]0, +\infty[$, $Z^{\uparrow}_y$ is actually continuous at $m_y$ (see the theorem in \cite{Doney}). The decomposition given by Theorem 25 in \cite{Doney} states that 

\begin{itemize}
\item The two processes $\left ( Z^{\uparrow}_y(m_y + s) - Z^{\uparrow}_y(m_y), \ s \geq 0 \right )$ and $\left (Z^{\uparrow}_y(s), \ 0 \leq s < m_y \right )$ are independent, 
\medbreak
\item $\left ( Z^{\uparrow}_y(m_y + s) - Z^{\uparrow}_y(m_y), \ s \geq 0 \right )$ is equal in law to $Z^{\uparrow}$. 
\end{itemize}
Now, 
\begin{align*}
I(Z^{\uparrow}_y) & = \int_0^{m_y} e^{-Z^{\uparrow}_y(u)}du + \int_{m_y}^{+\infty} e^{-Z^{\uparrow}_y(u)} du \\
& = \int_0^{m_y} e^{-Z^{\uparrow}_y(u)}du + e^{-Z^{\uparrow}_y(m_y)} \int_{0}^{+\infty} e^{-(Z^{\uparrow}_y(m_y + u) - Z^{\uparrow}_y(m_y))} du \\
& \geq e^{-y} \int_{0}^{+\infty} e^{-(Z^{\uparrow}_y(m_y + u) - Z^{\uparrow}_y(m_y))} du, 
\end{align*}
because almost surely $Z^{\uparrow}_y(m_y) \leq y$. Then, because of the above decomposition of the trajectory of $Z^{\uparrow}_y$, the latter is equal in law to $e^{-y} I(Z^{\uparrow})$. 
We thus get
\begin{eqnarray}
I(Z^{\uparrow}) \ \overset{sto}{\leq} \ e^{y} I(Z^{\uparrow}_y), \label{vneglapl2}
\end{eqnarray}
where $\overset{sto}{\leq}$ denotes a stochastic inequality. As a consequence we only need to prove the result for $I(Z^{\uparrow}_y)$. 
We now choose $\epsilon > 0$ as in Lemma \ref{vneglapltpslin}. We have 
\begin{align*}
0 \leq I(Z^{\uparrow}_y) & = \int_0^{\mathcal{R} \left (Z^{\uparrow}_y(.) - (y + \epsilon \ .), {]-\infty, 0]} \right )} e^{- Z^{\uparrow}_y(t)} dt + \int_{\mathcal{R} \left (Z^{\uparrow}_y(.) - (y + \epsilon \ .), {]-\infty, 0]} \right )}^{+ \infty} e^{- Z^{\uparrow}_y(t)} dt \\
& \leq \mathcal{R} \left (Z^{\uparrow}_y(.) - (y + \epsilon \ .), {]-\infty, 0]} \right ) + \int_{\mathcal{R} \left (Z^{\uparrow}_y(.) - (y + \epsilon \ .), {]-\infty, 0]} \right )}^{+ \infty} e^{- Z^{\uparrow}_y(t)} dt. 
\end{align*}
Then, for $t \geq \mathcal{R} (Z^{\uparrow}_y(.) - (y + \epsilon \ .), {]-\infty, 0]} )$, we have $Z^{\uparrow}_y(t) \geq y + \epsilon t$, so
\begin{align*}
0 \leq I(Z^{\uparrow}_y) & \leq \mathcal{R} \left (Z^{\uparrow}_y(.) - (y + \epsilon \ .), {]-\infty, 0]} \right ) + \int_{\mathcal{R} \left (Z^{\uparrow}_y(.) - (y + \epsilon \ .), {]-\infty, 0]} \right )}^{+ \infty} e^{-y - \epsilon t} dt \\
& \leq \mathcal{R} \left (Z^{\uparrow}_y(.) - (y + \epsilon \ .), {]-\infty, 0]} \right ) + \int_0^{+ \infty} e^{- y - \epsilon t} dt \\
& = \mathcal{R} \left (Z^{\uparrow}_y(.) - (y + \epsilon \ .), {]-\infty, 0]} \right ) + e^{-y}/\epsilon. 
\end{align*}
According to Lemma \ref{vneglapltpslin}, this is almost surely finite and admits some finite exponential moments. Thanks to $(\ref{vneglapl2})$ we have the same for $I(Z^{\uparrow})$ which is the expected result. 

\end{proof}

\begin{remarque}
The arguments of the proof of Theorem \ref{vneglapl} can be extended for a more general L\'evy process $Z$ if we assume that: $Z$ drifts to $+\infty$, $Z$ is regular for $]0, +\infty[$ and $]-\infty, 0[$, $\exists \lambda < 0$ s.t. $\mathbb{E}[\exp(\lambda Z(1))] < + \infty$ (which, according to Theorem 25.3 in \cite{Sato}, is equivalent to the fact that the restriction to $]-\infty, -1]$ of the L\'evy measure of $Z$ admits some finite exponential moments of negative order ). 
\end{remarque}

\subsection{Comparison between $I(Z^{\uparrow})$ and $I(Z)$: Proof of Theorem \ref{vnegqueue0}}


We need an analogous of Lemma \ref{dernierpassage2} in order to compare, as we did in subsection \ref{connection}, the exponential functionals $I(Z^{\uparrow})$ and $I(Z)$. We define $m$ to be the point where the process $Z$ reaches its infimum: $m := \sup \{ s \geq 0, \ Z(s-) \wedge Z(s) = \inf_{[0, +\infty[} Z \}$. Here again, from the absence of negative jumps, the infimum is always reached at least at $m-$ so $Z(m-) = \inf_{[0, +\infty[} Z$. Also, since $Z$ is regular for $]0, +\infty[$, $Z$ is actually continuous at $m$ (see the proof of Theorem 25 in \cite{Doney}). The arguments in the proof of Theorem 25 of \cite{Doney} yield 

\begin{lemme} \label{postminspecpos}

$Z (m + .) - Z(m)$ has the same law as $Z^{\uparrow}$ and is independent from the process $(Z(s), \ 0 \leq s \leq m)$. 

\end{lemme}

{We can now prove Theorem \ref{vnegqueue0}, using the decomposition from the above Lemma and arguments of fluctuation theory for L\'evy processes. 

\begin{proof} of Theorem \ref{vnegqueue0}


We have
\begin{align}
I(Z) & = \int_0^m e^{- Z(t)} dt + \int_m^{+ \infty} e^{- Z(t)} dt \nonumber \\
& = \int_0^m e^{- Z(t)} dt + e^{-Z(m)} \int_0^{+ \infty} e^{- (Z (m + t) - Z(m))} dt \label{amelqueue0} \\
& \geq \int_0^m e^{- Z(t)} dt + \int_0^{+ \infty} e^{- (Z (m + t) - Z(m))} dt \label{amelqueue1} \\
& \geq \int_0^{+ \infty} e^{- (Z (m + t) - Z(m))} dt. \nonumber 
\end{align}
For \eqref{amelqueue1} we have used the fact that almost surely $Z(m) \leq 0$. 
Lemma \ref{postminspecpos} tells us that the last term is equal in law to $I(Z^{\uparrow})$. 
We thus get $I(Z^{\uparrow}) \overset{sto}{\leq} I(Z)$ 
where $\overset{sto}{\leq}$ denotes a stochastic inequality. The first point of Theorem \ref{vnegqueue0} follows. 

We now prove the second point of the theorem. The combination of \eqref{amelqueue1} and Lemma \ref{postminspecpos} yields 
\begin{eqnarray}
\mathbb{P} \big ( I(Z) \leq x \big ) \leq \mathbb{P} \left( \int_0^m e^{- Z(t)} dt \leq x \right ) \times \mathbb{P} \big ( I(Z^{\uparrow}) \leq x \big ). \label{amelqueue2}
\end{eqnarray}
Let us fix $\epsilon \in ]0,1[$. Similarly as above, the combination of \eqref{amelqueue0} and Lemma \ref{postminspecpos} yields 
\begin{eqnarray}
\mathbb{P} \left( \int_0^m e^{- Z(t)} dt \leq \frac{\epsilon x}{2}, e^{-Z(m)} \leq 1+\frac{\epsilon}{2} \right ) \times \mathbb{P} \big ( I(Z^{\uparrow}) \leq x \big ) \leq \mathbb{P} \big ( I(Z) \leq (1+\epsilon)x \big ). \label{amelqueue3}
\end{eqnarray}
Now, note that $\int_0^m e^{- Z(t)} dt$ can be decomposed using fluctuation theory for L\'evy processes. Let $(\hat{L}^{-1}, \hat{H})$ denote the descending ladder process of $Z$, $\hat{L}(t) := \underline{Z}(t)$ is a local time at $0$ of $Z- \underline{Z}$ (this comes from Theorem VII.1 of \cite{Bertoin} and duality, since $Z$ is spectrally positive) and $\hat{L}^{-1}(.) = \tau(-Z,.)$ is the inverse of this local time, it is a subordinator killed at the exponential time $\hat{L}(\infty) = -Z(m)$ (that has parameter $\kappa_Z$, where $\kappa_Z$ is the non-trivial zero of $\Psi_{-Z}$, the Laplace exponent of $-Z$). For any $t \in [0, \hat{L}(\infty)[$, $\hat{H}(t) := -Z(\hat{L}^{-1}(t)) = t$, so that the Laplace exponent of $\hat{H}$ is $\lambda + \kappa_Z$. We denote by $\hat{\mathcal{N}}$ the excursion measure of $Z- \underline{Z}$ associated with $\hat{L}$, and for $\xi$ an excursion, recall that $\zeta(\xi) := \inf \{ s > 0, \ \xi(s) = 0 \}$ denotes its life-time and $H_0(\xi) := \sup_{[0, \zeta(\xi)]} \xi$ its height. $\zeta(\xi)$ and $H_0(\xi)$ can possibly be infinite. 

Let us also define $(L^{-1}, H)$, the ascending ladder process of $Z$: let $L$ be a local time at $0$ of $\overline{Z} - Z$, $L^{-1}$ the inverse of this local time, and $H(t) := Z(L^{-1}(t))$. Let $\Phi_H$ denote the Laplace exponent of $H$. 
According to the Wiener-Hopf factorization for Laplace exponents we have 
\begin{eqnarray}
\forall \lambda \geq 0, \ \Psi_{-Z}(\lambda) = -(-\lambda + \kappa_Z) \Phi_H(\lambda) = (\lambda - \kappa_Z) \Phi_H(\lambda). \label{wienerhopf}
\end{eqnarray}
Note that the above equality is true up to a multiplicative constant but we assume that the normalization of the local time $L$ has been chosen in such a way that this constant equals $1$. 

Recall from the Introduction (just before the statement of the theorem) that $\Phi_{-Z}$ denotes the Laplace exponent of the subordinator $\tau(-Z,.) = \hat{L}^{-1}(.)$ 
and that $\Phi_{-Z}$ is also the inverse of the restriction of $\Psi_{-Z}$ to $[\kappa_Z, +\infty[$: $\Phi_{-Z} = \Psi_{-Z |[\kappa_Z, +\infty[}^{-1}$. Since $Z$ has unbounded variation, Corollary VII.5 in \cite{Bertoin} yields that $\Psi_{-Z}(\lambda)/\lambda \longrightarrow_{\lambda \rightarrow +\infty} +\infty$ so $\Phi_{-Z}(\lambda)/\lambda \longrightarrow_{\lambda \rightarrow +\infty} 0$. Therefore, $\tau(-Z, .) = \hat{L}^{-1}(.)$ has no drift component. According to Corollary IV.6 in \cite{Bertoin}, there is a constant $\textbf{d} \geq 0$ such that $\int_0^t \mathds{1}_{Z(u)=\underline{Z}(u)} du = \textbf{d} \hat{L}(t)$, and according to Theorem IV.8 in \cite{Bertoin}, that $\textbf{d}$ is the drift component of $\hat{L}^{-1}$. In particular, we deduce that $\textbf{d} = 0$ in our case. Note also that $\Phi_{-Z}(\lambda) = \Psi_{-Z |[\kappa_Z, +\infty[}^{-1}(\lambda) \underset{\lambda \rightarrow +\infty}{\longrightarrow} +\infty$, so that $\tau(-Z, .) = \hat{L}^{-1}(.)$ is not a compound Poisson process. 

We can now decompose $\int_0^{\hat{L}^{-1}(t)} e^{- Z(u)} du$ similarly as in Lemma 2 of \cite{rivero2012} (note however that in our case, there is a minus in the exponential): let $Y$ be the subordinator defined by 
\[ \forall t \in [0, \hat{L}(\infty)[, \ Y(t) := \sum_{0 \leq u \leq t} \int_0^{\hat{L}^{-1}(u)-\hat{L}^{-1}(u-)} \exp \left ( -[Z(s+\hat{L}^{-1}(u-)) - Z(\hat{L}^{-1}(u-))] \right ) ds. \]
$Y$ must normally have a drift component equal to $\textbf{d} t$, but recall that $\textbf{d}$ has been justified to equal $0$ in our case. Note that $Y$ has a positive jump each time $\hat{L}^{-1}$ has so, since $\hat{L}^{-1}$ is not a compound Poisson process, neither is $Y$. Then, similarly as in Lemma 2 of \cite{rivero2012} we have 
\[ \left ( \int_0^{\hat{L}^{-1}(t)} e^{- Z(u)} du, \ 0 \leq t \leq \hat{L}(\infty) \right ) = \left ( \int_0^{t} e^{\hat{H}(u-)} dY(u), \ 0 \leq t \leq \hat{L}(\infty) \right ). \]
Note that in our case, these processes have finite life-time equal to $\hat{L}(\infty)$, that $\hat{L}^{-1}(\hat{L}(\infty)) = +\infty$, and that the jump of $\hat{L}^{-1}$ at $\hat{L}(\infty)$ corresponds to the infinite excursion of $Z- \underline{Z}$. Recall that $\hat{H}(u) = u$ for $u \in [0, \hat{L}(\infty)[$ and note that $\hat{L}^{-1}(\hat{L}(\infty)-) = m$, so taking the left limit of the previous processes at $\hat{L}(\infty) (= -Z(m))$ we get 
\[ \int_0^{m} e^{- Z(u)} du = \int_0^{\hat{L}(\infty)-} e^{u} dY(u). \]
We deduce that 
\begin{eqnarray}
Y(\hat{L}(\infty)-) \leq \int_0^{m} e^{- Z(u)} du \leq e^{\hat{L}(\infty)} Y(\hat{L}(\infty)-). \label{amelqueue4}
\end{eqnarray}
Considering the process of excursions of $Z- \underline{Z}$, $Y(\hat{L}(\infty)-)$ is the sum of the contributions of finite excursions, until the infinite excursion that occurs at time $\hat{L}(\infty)$. Therefore, $(\hat{L}(\infty), Y(\hat{L}(\infty)-))$ is equal in law to $(T, \tilde Y(T))$, where $\tilde Y$ is a non-killed subordinator built similarly as $Y$, but from the Poisson point process on $[0, +\infty[ \times \{ \xi, \ \zeta(\xi) < +\infty \}$ with intensity measure $dt \times \hat{\mathcal{N}}(. \cap \{ \xi, \ \zeta(\xi) < +\infty \})$, and $T$ is an independent exponential time with parameter $\kappa_Z$. Let $\Phi_{\tilde Y}$ denote the Laplace exponent of $\tilde Y$. Using the lower bound in \eqref{amelqueue4} and $(\hat{L}(\infty), Y(\hat{L}(\infty)-)) \overset{\mathcal{L}}{=} (T, \tilde Y(T))$, we have 
\begin{align}
\mathbb{P} \left( \int_0^m e^{- Z(t)} dt \leq x \right ) & \leq \mathbb{P} \left( Y(\hat{L}(\infty)-) \leq x \right ) = \mathbb{P} \left( \tilde Y(T) \leq x \right ) = \mathbb{P} \left( T \leq \tau(\tilde Y,x+) \right ) \nonumber \\
& = \mathbb{E} \left [ \int_0^{\tau(\tilde Y,x+)} \kappa_Z e^{-\kappa_Z u} du \right ] \leq \kappa_Z \mathbb{E} \left [ \tau(\tilde Y,x+) \right ] \nonumber \\
& = \kappa_Z V_{\tilde Y} ([0,x]) \leq \frac{\kappa_Z e}{\Phi_{\tilde Y}(1/x)}. \label{amelqueue5}
\end{align}
In the above, $V_{\tilde Y} (.)$ stands for the occupation measure of $\tilde Y$ 
and the last inequality comes from Proposition III.1 of \cite{Bertoin} (see in particular expression III.(5) there). 
Using this time the upper bound in \eqref{amelqueue4}, the fact that $\hat{L}(\infty) = -Z(m)$, and $(\hat{L}(\infty), Y(\hat{L}(\infty)-)) \overset{\mathcal{L}}{=} (T, \tilde Y(T))$, we have 
\begin{align}
& \mathbb{P} \left( \int_0^m e^{- Z(t)} dt \leq \frac{\epsilon x}{2}, e^{-Z(m)} \leq 1+\frac{\epsilon}{2} \right ) \geq \mathbb{P} \left( e^{\hat{L}(\infty)} Y(\hat{L}(\infty)-) \leq \frac{\epsilon x}{2}, e^{\hat{L}(\infty)} \leq 1+\frac{\epsilon}{2} \right ) \nonumber \\
\geq & \mathbb{P} \left( Y(\hat{L}(\infty)-) \leq \frac{\epsilon x}{2+\epsilon}, e^{\hat{L}(\infty)} \leq 1+\frac{\epsilon}{2} \right ) = \mathbb{P} \left( \tilde Y(T) \leq \frac{\epsilon x}{2+\epsilon}, e^{T} \leq 1+\frac{\epsilon}{2} \right ) \nonumber \\ 
= & \mathbb{P} \left( T \leq \tau \left ( \tilde Y,\frac{\epsilon x}{2+\epsilon}+ \right ) \wedge \log \left (1+\frac{\epsilon}{2} \right ) \right ) = \mathbb{E} \left [ \int_0^{\tau \left ( \tilde Y,\frac{\epsilon x}{2+\epsilon}+ \right ) \wedge \log \left (1+\frac{\epsilon}{2} \right )} \kappa_Z e^{-\kappa_Z u} du \right ] \nonumber \\
& \geq \kappa_Z \left (1+\frac{\epsilon}{2} \right )^{-\kappa_Z} \mathbb{E} \left [ \tau \left ( \tilde Y,\frac{\epsilon x}{2+\epsilon}+ \right ) \wedge \log \left (1+\frac{\epsilon}{2} \right ) \right ] \nonumber \\
& = \kappa_Z \left (1+\frac{\epsilon}{2} \right )^{-\kappa_Z} \left ( \mathbb{E} \left [ \tau \left ( \tilde Y,\frac{\epsilon x}{2+\epsilon}+ \right ) \right ] - \int_{\log \left (1+\frac{\epsilon}{2} \right )}^{+\infty} \mathbb{P} \left ( \tau \left ( \tilde Y,\frac{\epsilon x}{2+\epsilon}+ \right ) > s\right ) ds \right ) \nonumber \\
& = \kappa_Z \left (1+\frac{\epsilon}{2} \right )^{-\kappa_Z} \left ( V_{\tilde Y} \left ( \left [0, \frac{\epsilon x}{2+\epsilon} \right ] \right ) - \int_{\log \left (1+\frac{\epsilon}{2} \right )}^{+\infty} \mathbb{P} \left ( \tau \left ( \tilde Y,\frac{\epsilon x}{2+\epsilon}+ \right ) > s \right ) ds \right ). \label{amelqueue6}
\end{align}
Then, note that 
\begin{align}
& \int_{\log \left (1+\frac{\epsilon}{2} \right )}^{+\infty} \mathbb{P} \left ( \tau \left ( \tilde Y,\frac{\epsilon x}{2+\epsilon}+ \right ) > s \right ) ds = \int_{\log \left (1+\frac{\epsilon}{2} \right )}^{+\infty} \mathbb{P} \left (\tilde Y(s) \leq \frac{\epsilon x}{2+\epsilon} \right ) ds \nonumber \\
= & \int_{\log \left (1+\frac{\epsilon}{2} \right )}^{+\infty} \mathbb{P} \left ( \exp \left ( -\frac{2+\epsilon}{\epsilon x} \tilde Y(s) \right ) \geq e^{-1} \right ) ds \leq e \int_{\log \left (1+\frac{\epsilon}{2} \right )}^{+\infty} \mathbb{E} \left [ \exp \left ( -\frac{2+\epsilon}{\epsilon x} \tilde Y(s) \right ) \right ] ds \nonumber \\
= & e \int_{\log \left (1+\frac{\epsilon}{2} \right )}^{+\infty} \exp \left ( - s \Phi_{\tilde Y} \left (\frac{2+\epsilon}{\epsilon x} \right ) \right ) ds = \frac{e}{\Phi_{\tilde Y} \left (\frac{2+\epsilon}{\epsilon x} \right ) \left (1+\frac{\epsilon}{2} \right )^{\Phi_{\tilde Y} \left (\frac{2+\epsilon}{\epsilon x} \right )}}. \label{amelqueue7}
\end{align}
According to Proposition III.1 of \cite{Bertoin}, the term $V_{\tilde Y} ( [0, \epsilon x/(2+\epsilon) ] )$ in \eqref{amelqueue6} is greater than $1/4 \log(8)\Phi_{\tilde Y} ((2+\epsilon)/\epsilon x)$. Moreover $\Phi_{\tilde Y} ((2+\epsilon)/\epsilon x)$ converges to $+\infty$ when $x$ goes to $0$ since $\tilde Y$ is not a compound Poisson process (because $Y$ is not, so neither is $\tilde Y$), so combining with \eqref{amelqueue7} and putting into \eqref{amelqueue6} we get for $x$ small enough 
\begin{eqnarray}
\mathbb{P} \left( \int_0^m e^{- Z(t)} dt \leq \frac{\epsilon x}{2}, e^{-Z(m)} \leq 1+\frac{\epsilon}{2} \right ) \geq \frac{c_1}{\Phi_{\tilde Y} \left (\frac{2+\epsilon}{\epsilon x} \right )} \geq \frac{c_1 \epsilon}{(2+\epsilon)\Phi_{\tilde Y} \left (\frac{1}{x} \right )}. \label{amelqueue8}
\end{eqnarray}
In the above, $c_1 = c_1(\kappa_Z, \epsilon)$ is a positive constant and the last inequality comes from concavity of $\Phi_{\tilde Y}$. 
Now plugging \eqref{amelqueue5} and \eqref{amelqueue8} into respectively \eqref{amelqueue2} and \eqref{amelqueue3} we obtain the existence of a positive constant $c_2 = c_2(\kappa_Z, \epsilon)$ such that: 
\begin{eqnarray}
\frac1{\kappa_Z e} \mathbb{P} \big ( I(Z) \leq x \big ) \leq \frac{\mathbb{P} \big ( I(Z^{\uparrow}) \leq x \big )}{\Phi_{\tilde Y} \left (\frac{1}{x} \right )} \leq c_2 \mathbb{P} \big ( I(Z) \leq (1+\epsilon)x \big ). \label{amelqueue9}
\end{eqnarray}
Note that the upper bound is true only for $x$ small enough. 

To establish the result we now need to understand the asymptotic behavior of $\Phi_{\tilde Y}$ at infinity. By construction, $\tilde Y$ has no drift component and it only represents the contribution of finite excursions of $Z- \underline{Z}$. Therefore, similarly as in the proof of Lemma 2(iii) of \cite{rivero2012} we have 
\[ \forall \lambda \geq 0, \ \Phi_{\tilde Y}(\lambda) = \hat{\mathcal{N}} \left ( 1 - \exp \left ( -\lambda \int_0^{\zeta(\xi)} e^{- \xi(s)} ds \right ) ; \zeta(\xi) < +\infty \right ). \]
We can repeat, in our case, the computations from the proof of Lemma 2(iii) in \cite{rivero2012}. The requirement $\zeta(\xi) < +\infty$ enters into the expectation where it eventually becomes $\tau(-Z,x) < +\infty$. We obtain 
\[ \forall \lambda \geq 0, \ \Phi_{\tilde Y}(\lambda) = \lambda \int_0^{+\infty} e^{-x} \mathbb{E} \left [ \exp \left ( - \lambda e^{-x} \int_0^{\tau(-Z,x)} e^{-Z(s)} ds \right ) ; \tau(-Z,x) < +\infty \right ] V_H(dx). \]
In this expression, $V_H(dx)$ stands for the occupation measure of the ascending ladder height process $H$. The appearance of $V_H$ in the above expression relies on the fact that the occupation measure of the ascending ladder height process $H$ equals the occupation measure for excursions of $Z- \underline{Z}$, as proved in exercise VI.5 of \cite{Bertoin}. Let us mention that, for this to hold, it is important that the local time $L$ is normalized as we assumed a little before, that is, so that \eqref{wienerhopf} holds (otherwise $V_H$ has to be multiplied by some constant). 

Clearly $\int_0^{\tau(-Z,x)} e^{-Z(s)} ds \leq e^x \tau(-Z,x)$ so we get for all $\lambda \geq 0$ 
\begin{align}
\Phi_{\tilde Y}(\lambda) & \geq \lambda \int_0^{+\infty} e^{-x} \mathbb{E} \left [ e^{- \lambda \tau(-Z,x)} ; \tau(-Z,x) < +\infty \right ] V_H(dx) \nonumber \\
& = \lambda \int_0^{+\infty} e^{-x(1+\Phi_{-Z}(\lambda))} V_H(dx) = \frac{\lambda}{\Phi_H(1+\Phi_{-Z}(\lambda))} \nonumber \\
& \geq \frac{\lambda}{\Phi_H(\Phi_{-Z}(\lambda))} \times \frac{\Phi_{-Z}(\lambda)}{1+\Phi_{-Z}(\lambda)} = \frac{\lambda (\Phi_{-Z}(\lambda) - \kappa_Z)}{\Psi_{-Z}(\Phi_{-Z}(\lambda))} \times \frac{\Phi_{-Z}(\lambda)}{1+\Phi_{-Z}(\lambda)} \nonumber \\
& = (\Phi_{-Z}(\lambda) - \kappa_Z) \times \frac{\Phi_{-Z}(\lambda)}{1+\Phi_{-Z}(\lambda)} \underset{\lambda \rightarrow +\infty}{\sim} \Phi_{-Z}(\lambda). \label{equivlplb}
\end{align}
In the above we have used the fact that the Laplace transform of $V_H(dx)$ is $1/\Phi_H(.)$, the concavity of $\Phi_H$, \eqref{wienerhopf}, and the fact that $\Phi_{-Z} = \Psi_{-Z |[\kappa_Z, +\infty[}^{-1}$. 

Let us now prove an upper bound for the asymptotic behavior of $\Phi_{\tilde Y}$. Let us fix $\eta > 0$ and define $\tilde Y^{\eta}$ by cutting some jumps from $\tilde Y$: those resulting from excursions higher than $\eta$. As a result we have 
\[ \forall \lambda \geq 0, \ \Phi_{\tilde Y^{\eta}}(\lambda) = \hat{\mathcal{N}} \left ( 1 - \exp \left ( -\lambda \int_0^{\zeta(\xi)} e^{- \xi(s)} ds \right ) ; H_0(\xi) < \eta \right ). \]
Since the set of jumps removed has finite measure for $\hat{\mathcal{N}} (. \cap \{ \xi, \ \zeta(\xi) < +\infty \})$, $\tilde Y$ is equal to $\tilde Y^{\eta}$ plus an independent compound Poisson process which implies that $\Phi_{\tilde Y^{\eta}}(\lambda) \sim_{\lambda \rightarrow +\infty} \Phi_{\tilde Y}(\lambda)$ (because the Laplace exponent of a compound Poisson process is bounded). We have 
\begin{align*}
\forall \lambda \geq 0, \ \Phi_{\tilde Y^{\eta}}(\lambda) & = \lambda \ \hat{\mathcal{N}} \left ( \int_0^{\zeta(\xi)} e^{- \xi(s)} \exp \left ( -\lambda \int_s^{\zeta(\xi)} e^{- \xi(u)} du \right ) ds ; H_0(\xi) < \eta \right ) \\
& \leq \lambda \ \hat{\mathcal{N}} \left ( \int_0^{\zeta(\xi)} e^{- \xi(s)} \mathds{1}_{\xi(s) \leq \eta} \mathds{1}_{\tau(\xi(s+.), 0) < \tau(\xi(s+.), \xi(s)+\eta)} \exp \left ( -\lambda \int_s^{\zeta(\xi)} e^{- \xi(u)} du \right ) ds \right ). 
\end{align*}
Then we can repeat, from the above term, the computations from the proof of Lemma 2(iii) of \cite{rivero2012}. We get for all $\lambda \geq 0$, 
\begin{align}
\Phi_{\tilde Y^{\eta}}(\lambda) & \leq \lambda \int_0^{\eta} e^{-x} \mathbb{E} \left [ \exp \left ( - \lambda e^{-x} \int_0^{\tau(-Z,x)} e^{-Z(s)} ds \right ) ; \tau(-Z,x) < \tau(Z,\eta+) \right ] V_H(dx) \nonumber \\
& \leq \lambda \int_0^{\eta} \mathbb{E} \left [ e^{- \lambda e^{-2\eta} \tau(-Z,x)} ; \tau(-Z,x) < +\infty \right ] V_H(dx) = \lambda \int_0^{\eta} e^{-x\Phi_{-Z}(e^{-2\eta} \lambda)} V_H(dx) \nonumber \\
& \leq \lambda \int_0^{+\infty} e^{-x\Phi_{-Z}(e^{-2\eta} \lambda)} V_H(dx) = \frac{\lambda}{\Phi_H(\Phi_{-Z}(e^{-2\eta} \lambda))} \leq \frac{\lambda}{\Phi_H(e^{-2\eta} \Phi_{-Z}(\lambda))}  \nonumber \\ 
& \leq \frac{e^{2\eta} \lambda}{\Phi_H(\Phi_{-Z}(\lambda))} = \frac{e^{2\eta} \lambda (\Phi_{-Z}(\lambda) - \kappa_Z)}{\Psi_{-Z}(\Phi_{-Z}(\lambda))} = e^{2\eta} (\Phi_{-Z}(\lambda) - \kappa_Z) \underset{\lambda \rightarrow +\infty}{\sim} e^{2\eta} \Phi_{-Z}(\lambda). \label{equivlpub}
\end{align}
In the above we have used the fact that the Laplace transform of $V_H(dx)$ is $1/\Phi_H(.)$, the concavity of $\Phi_{-Z}$ and $\Phi_H$, \eqref{wienerhopf}, and the fact that $\Phi_{-Z} = \Psi_{-Z |[\kappa_Z, +\infty[}^{-1}$. 

Combining \eqref{equivlplb}, $\Phi_{\tilde Y}(\lambda) \sim_{\lambda \rightarrow +\infty} \Phi_{\tilde Y^{\eta}}(\lambda)$, \eqref{equivlpub}, and using that $\eta$ can be chosen as small as we want, we deduce that $\Phi_{\tilde Y}(\lambda) \sim_{\lambda \rightarrow +\infty} \Phi_{-Z}(\lambda)$. The combination of this with \eqref{amelqueue9} yields the second statement of the theorem.

\end{proof}
}

\begin{remarque} \label{tivialfinitnessvnb}
The finiteness of $I(Z^{\uparrow})$ can be proved directly using the first point of Theorem \ref{vnegqueue0} combined with the well known finiteness of $I(Z)$ (see Theorem 1 in \cite{Bertoinyor}). As in the spectrally negative case, this {comparison} with the functional of the non-conditioned process could not be used to derive the existence of exponential moments for $I(Z^{\uparrow})$ (Theorem \ref{vneglapl}), since the right tail of $I(Z)$ is heavier (see Remark \ref{compnoncondpecpos}). 
\end{remarque}

\begin{remarque} 
The arguments of the proof of the first point of Theorem \ref{vnegqueue0} can be extended for a more general L\'evy process $Z$ if we assume that $Z$ drifts to $+\infty$ and $Z$ is regular for $]0, +\infty[$ and $]-\infty, 0[$. For the second point, the spectral positivity of $Z$ is crucially used all along the proof. 
\end{remarque}

\section{Asymptotic tail at $0$ in the presence of positive jumps} \label{queue0sautspos}

\subsection{Proof of Theorem \ref{queue0general}}

As in the Introduction, $Y$ is a L\'evy process drifting to $+\infty$. Let us denote by $(A, d, \pi)$ its L\'evy triplet. $\pi$ therefore denotes its L\'evy measure and for $x > 0$, $\overline{\pi}(x) := \pi([x, +\infty[)$. 

We first prove \eqref{suppnb}. It is trivially satisfied if the support of $\pi$ is bounded from above. We thus assume that it is not. Let us fix $x \in ]0, 1[$ and define
\[ \pi^{a} := \pi ( . \cap [\log(1/x), +\infty[ ) \ \text{  and  } \ \pi^{b} := \pi ( . \cap ]-\infty, \log(1/x)[ ). \]
Let $Y^{a}$ and $Y^{b}$ be two independent L\'evy processes which generating triplets are respectively $(0, 0, \pi^{a})$ and $(A, d, \pi^{b})$. This yields $Y \overset{\mathcal{L}}{=} Y^{a} + Y^{b}$ according to the L\'evy-Khintchine formula. Let $J$ denote the random instant of the first jump of $Y^{a}$ (note that $J$ follows an exponential distribution with parameter $\overline{\pi}(\log(1/x))$) and $\tilde Y := Y(J+.)-Y(J)$. Because of the Markov property $\tilde Y$ is equal in law to $Y$ and independent from $\sigma (Y^{a}(s), {Y^{b}}(s), \ 0 \leq s \leq J)$. We have 
\begin{align*}
\mathbb{P} \left ( I(Y) \leq x \right ) & \geq \mathbb{P} \left ( J \leq x/3, \ \int_0^{J} e^{-Y^{b}(t)} dt \leq 2x/3, \ e^{-Y(J)} I(\tilde Y) \leq x/3 \right ) \\
& \geq \mathbb{P} \left ( J \leq x/3, \ \inf_{[0, J]} Y^{b} \geq -\log(2), \ e^{-Y(J)} I(\tilde Y ) \leq x/3 \right ). 
\end{align*}
Note that, on $\{ \inf_{[0, J]} Y ^{b} \geq -\log(2) \}$, $Y (J) = Y ^{a}(J) + Y ^{b}(J) \geq \log(1/x) -\log(2) = \log(1/2x)$ so $e^{-Y (J)} \leq 2x$. We thus have 
\begin{align*}
\mathbb{P} \left ( I(Y ) \leq x \right ) & \geq \mathbb{P} \left ( J \leq x/3, \ \inf_{[0, J]} Y ^{b} \geq -\log(2), \ 2x I(\tilde Y ) \leq x/3 \right ) \\
& = \mathbb{P} \left ( J \leq x/3, \ \inf_{[0, J]} Y ^{b} \geq -\log(2) \right ) \times \mathbb{P} \left ( I(\tilde Y ) \leq 1/6 \right ). 
\end{align*}
In the last equality we have used the independence between $\tilde Y $ and $\sigma (Y ^{a}(s), Y ^b(s), \ 0 \leq s \leq J)$. Then, 
\[ \mathbb{P} \left ( I(Y ) \leq x \right ) \geq \mathbb{P} \left ( J \leq x/3, \ \inf_{[0, 1]} Y ^{b} \geq -\log(2) \right ) \times \mathbb{P} \left ( I(\tilde Y ) \leq 1/6 \right ), \]
where we have used  the fact that, on $\{ J < x/3 \}$, $[0, J] \subset [0, 1]$. Then, since $Y ^{a}$ and $Y ^{b}$ are independent we obtain 
\begin{align*}
\mathbb{P} \left ( I(Y ) \leq x \right ) & \geq \mathbb{P} \left ( J \leq x/3 \right ) \times \mathbb{P} \left ( \inf_{[0, 1]} Y ^{b} \geq -\log(2) \right ) \times \mathbb{P} \left ( I(\tilde Y ) \leq 1/6 \right ) \\
& = \left ( 1 - e^{- x \overline{\pi}(\log(1/x))/3} \right ) \times \mathbb{P} \left ( \inf_{[0, 1]} Y ^{b} \geq -\log(2) \right ) \times \mathbb{P} \left ( I(\tilde Y ) \leq 1/6 \right ). 
\end{align*}
When $x$ goes to $0$, the first factor is equivalent to $x \overline{\pi}(\log(1/x))/3$ while the second factor converges to $\mathbb{P} ( \inf_{[0, 1]} Y  \geq -\log(2) )$ and the third factor remains constant. We thus get \eqref{suppnb}. 

The idea to prove \eqref{suppborne} is to first prove \eqref{caspoissonth} and then to make appear a Poisson process in the trajectory of a general $Y$. We thus now determine the left tail of $I(\alpha N)$ where $N$ is a standard Poisson process and $\alpha > 0$. Let us denote by $p$ the parameter of $N$. 

%

We proceed by studying the asymptotic behavior of the $\log$-Laplace transform of $I(\alpha N)$ at infinity. By a property of standard Poisson processes, it is easy to see that $I(\alpha N)$ can be written as
\[ \frac1{p} \sum_{k=0}^{+\infty} e^{-\alpha k} e_k, \]
where $(e_k)_{k \in \mathbb{N}}$ is a sequence of \textit{iid} exponential random variable with parameter $1$. This allows to compute the $\log$-Laplace transform of $I(\alpha N)$: 
\begin{eqnarray}
\forall \lambda \geq 0, \ - \log \left ( \mathbb{E} \left [ e^{-\lambda I(\alpha N)} \right ] \right ) = \sum_{k=0}^{+\infty} \log \left ( 1+ \frac{\lambda}{p} e^{-\alpha k} \right ). \label{serieloglapl}
\end{eqnarray}
Then, note that for any integer $k \geq 0$ we have 
\[ \log \left ( 1+ \frac{\lambda}{p} e^{-\alpha (k+1)} \right ) \leq \int_{k}^{k+1} \log \left ( 1+ \frac{\lambda}{p} e^{-\alpha y} \right ) dy \leq \log \left ( 1+ \frac{\lambda}{p} e^{-\alpha k} \right ). \]
Combining with \eqref{serieloglapl} we get for any $\lambda \geq 0$, 
\begin{eqnarray}
J(\lambda) \leq - \log \left ( \mathbb{E} \left [ e^{-\lambda I(\alpha N)} \right ] \right ) \leq J(\lambda) + \log \left ( 1+ \frac{\lambda}{p} \right ), \label{loglaplequivint}
\end{eqnarray}
where $J(\lambda) := \int_{0}^{+\infty} \log ( 1+ \frac{\lambda}{p} e^{-\alpha y} ) dy$. Making the change of variable $u = \lambda e^{-\alpha y} /p$ we get 
\[ J(\lambda) = \frac1{\alpha} \int_{0}^{\lambda/p} \frac{\log(1+u)}{u} du. \]
Now, $\lambda$ only determines the intervalle of integration. Since the above integral converges to $+\infty$ when $\lambda$ goes to infinity we have, for any $M > 0$, 
\[ J(\lambda) \underset{\lambda \rightarrow +\infty}{\sim} \frac1{\alpha} \int_{M}^{\lambda/p} \frac{\log(1+u)}{u} du. \]
In particular, we can fix $\epsilon > 0$ and choose $M$ so large such that $\forall u \geq M, \ (1-\epsilon) \log(u) \leq \log(1+u) \leq (1+\epsilon) \log(u)$. This implies in particular that 
\[ J(\lambda) \underset{\lambda \rightarrow +\infty}{\sim} \frac1{\alpha} \int_{1}^{\lambda/p} \frac{\log(u)}{u} du. \]
Now making the change of variable $v = \log(u)$ we get that the right hand side equals 
\[ \frac1{\alpha} \int_{0}^{\log (\lambda/p)} v dv = \frac1{2 \alpha} \left ( \log (\lambda/p) \right )^2. \]
As a consequence we have $J(\lambda) \sim ( \log (\lambda) )^2 / 2 \alpha$, so putting into \eqref{loglaplequivint} we get 
\begin{eqnarray}
- \log \left ( \mathbb{E} \left [ e^{-\lambda I(\alpha N)} \right ] \right ) \underset{\lambda \rightarrow +\infty}{\sim} \frac1{2 \alpha} \left ( \log (\lambda) \right )^2. \label{loglaplequivlog}
\end{eqnarray}

We now deduce the left tail of $I(\alpha N)$ by standard methods. From Markov's inequality we have 
\begin{eqnarray}
\forall x > 0, \ \mathbb{P} \left ( I(\alpha N) \leq x \right ) \leq e \mathbb{E} \left [ e^{- I(\alpha N)/x} \right ]. \label{markovpoisson}
\end{eqnarray}
Let us fix $\epsilon > 0$. We have 
\begin{align}
\forall x > 0, \ \mathbb{E} \left [ e^{- I(\alpha N)/x^{1+\epsilon}} \right ] & \leq \mathbb{E} \left [ e^{- I(\alpha N)/x^{1+\epsilon}} \mathds{1}_{ \left \{ I(\alpha N) \leq x \right \} } \right ] + \mathbb{E} \left [ e^{- I(\alpha N)/x^{1+\epsilon}} \mathds{1}_{ \left \{ I(\alpha N) > x \right \} } \right ] \nonumber \\
& \leq \mathbb{P} \left ( I(\alpha N) \leq x \right ) + e^{- 1/x^{\epsilon}}. \label{markovinvpoisson}
\end{align}
Combining \eqref{markovpoisson} and \eqref{markovinvpoisson} with \eqref{loglaplequivlog} we get 
\[ 1 \leq \underset{x \rightarrow 0}{\liminf} \frac{-\log \left ( \mathbb{P} \left ( I(\alpha N) \leq x \right ) \right )}{\frac1{2\alpha} \left ( \log (x) \right )^2} \leq \underset{x \rightarrow 0}{\limsup} \frac{-\log \left ( \mathbb{P} \left ( I(\alpha N) \leq x \right ) \right )}{\frac1{2\alpha} \left ( \log (x) \right )^2} \leq (1+\epsilon)^2. \]
Letting $\epsilon$ go to $0$ we get \eqref{caspoissonth}. 

We now prove \eqref{suppborne}. 
Since, by assumption, $\pi (. \cap ]0, +\infty[)$ is non zero, there exist $0 < \gamma_1< \gamma_2 < + \infty$ such that $\pi ( [\gamma_1, \gamma_2[) > 0$. Then, for $\eta \in ]0, 1[$ we define
\[ \pi^{\eta, 1} := \eta \pi ( . \cap [\gamma_1, \gamma_2[ ) \ \text{  and  } \ \pi^{\eta, 2} := \pi - \eta \pi ( . \cap [\gamma_1, \gamma_2[ ). \]
We set $Y^{\eta, 1}$ and $Y^{\eta, 2}$ to be two independent L\'evy processes which generating triplets are respectively $(0, 0, \pi^{\eta, 1})$ and $(A, d, \pi^{\eta, 2})$. We have $\pi = \pi^{\eta, 1} + \pi^{\eta, 2}$ so according to the L\'evy-Khintchine formula, 
\[ Y \overset{\mathcal{L}}{=} Y^{\eta, 1} + Y^{\eta, 2}. \]

Let us justify that $\eta$ can be chosen in such a way so that $Y^{\eta, 2}$ drifts to $+\infty$. First let us assume that $\mathbb{E} [(Y(1))^-] < +\infty$. Then $\mathbb{E} [Y(1)]$ is defined and belongs to $]0, +\infty]$ (otherwise $Y$ would not drift to $+\infty$ according to the law of large number and Theorem 36.7 in \cite{Sato}). Since $\mathbb{E}[Y(1)] = \mathbb{E}[Y^{\eta, 1}(1)] + \mathbb{E}[Y^{\eta, 2}(1)]$ and $\mathbb{E}[Y^{\eta, 1}(1)] < \gamma_2 \eta \pi ( [\gamma_1, \gamma_2[)$, we have that $\mathbb{E}[Y^{\eta, 2}(1)]$ is positive for $\eta$ small enough. By the law of large numbers, this implies that $Y^{\eta, 2}$ drifts to $+\infty$ for such small $\eta$. We now assume that $\mathbb{E} [(Y(1))^-] = +\infty$, then we also have $\mathbb{E} [(Y(1))^+] = +\infty$ (otherwise $Y$ would not drift to $+\infty$ according to the law of large numbers). Since $Y$ drifts to $+\infty$, $Y$ is in the case $(1)$ of Remark 37.13 of \cite{Sato}. Note that $\pi^{\eta, 2}$ and $\pi$ have the same asymptotic tails at $+\infty$ and $-\infty$ so replacing $\pi$ by $\pi^{\eta, 2}$ does not  change the finiteness of $K^-$ (resp. the infiniteness of $K^+$) from Remark 37.13 of \cite{Sato}. As a consequence for any $\eta \in ]0, 1[$, $Y^{\eta, 2}$ is also in the case $(1)$ of Remark 37.13 of \cite{Sato}, and therefore drifts to $+\infty$. In any case we can chose $\eta$ such that $Y^{\eta, 2}$ drifts to $+\infty$. We thus choose such an $\eta_0 \in ]0, 1[$ and denote by $m_2$ the random instant at which $Y^{\eta_0, 2}$ reaches its minimum. 


$Y^{\eta_0, 1}$ is a compound Poisson process, 
we define $N$ the counting process of its jumps: $N(t) := \sharp \left \{ s \in [0, t], \ Y^{\eta_0, 1}(s)- Y^{\eta_0, 1}(s-) > 0 \right \}$. $N$ is thus a standard Poisson process with parameter $c_0 := \eta_0 \pi ( [\gamma_1, \gamma_2[)$ and it is independent of $Y^{\eta_0, 2}$. We have 
\begin{align*}
I(Y) & = \int_0^{+ \infty} e^{- \left ( Y^{\eta_0, 1}(t) + Y^{\eta_0, 2}(t) \right )} dt \\
& \leq e^{-Y^{\eta_0, 2}(m_2)} \int_0^{+ \infty} e^{-Y^{\eta_0, 1}(t)} dt \\
& \leq e^{-Y^{\eta_0, 2}(m_2)} \int_0^{+ \infty} e^{-\gamma_1 N(t)} dt \\
& = e^{-Y^{\eta_0, 2}(m_2)} I(\gamma_1 N). \end{align*}
Then, from the independence between the two factors 
\begin{eqnarray} \mathbb{P} \left ( I(\gamma_1 N) \leq x/2 \right ) \times \mathbb{P} \left ( e^{-Y^{\eta_0, 2}(m_2)} \leq 2 \right ) \leq \mathbb{P} \left ( I(Y) \leq x \right ) \label{vnegqueue02}. \end{eqnarray}
Applying \eqref{caspoissonth} to $I(\gamma_1 N)$ we get, for $\epsilon > 0$ and all $x$ small enough, 

\[ e^{-(1+\epsilon) (\log(x))^2 / 2 \gamma_1} \leq \mathbb{P} \left ( I(\gamma_1 N) \leq x/2 \right ). \]
Combining with 
\eqref{vnegqueue02}, we get the result with $c = (1+\epsilon)/ 2 \gamma_1$. Since $\epsilon$ can be chosen as close as we want to $0$ and $\gamma_1$ can be chosen as close as we want to $S$, \eqref{suppborne} follows. 

%

\subsection{Proof of Proposition \ref{vnegprecisequeue0}}

We now prove Proposition \ref{vnegprecisequeue0}. The ingredients of the proof are: the factorization identity of $I(Y)$ from \cite{ref6patie}, Theorem 2.5 of \cite{pardo2013} which gives the asymptotic tails of exponential functionals of subordinators, and easy calculations on regularly varying functions. We give all the details in order to make the proof rigorous. {It can be seen that the idea of the proof is basically the same as for the proof of Theorem 7 in \cite{aristarivero}.} 

As in the statement of the proposition, we assume that $Y$ is spectrally positive and its L\'evy measure $\pi$ satisfies $(A)$ for some $\alpha > 0$. 
Let $H$ be the ascending ladder height process associated with $Y$. 
$H$ is a proper subordinator since $Y$ drifts to $+\infty$. Let $\pi_H$ denote the L\'evy measure of $H$ and for any $x > 0, \overline{\pi_H}(x) := \pi_H([x, +\infty[)$. 
{Let $\kappa_Y$ be the non-trivial zero of the Laplace exponent of $-Y$. Since $-Y$ is spectrally negative, we can apply the equality just before $(9.4.7)$ in \cite{Doney}, which yields 
\[ \forall x \geq 0, \ \overline{\pi_H}(x) = \int_0^{+\infty} e^{-\kappa_Y z} \overline{\pi}(z + x) dz \ \ \ \text{so} \ \ \ \frac{\overline{\pi_H}(x)}{\overline{\pi}(x)} = \int_0^{+\infty} e^{-\kappa_Y z} \frac{\overline{\pi}(z + x)}{\overline{\pi}(x)} dz. \]
}$\overline{\pi}$ is non-increasing so, for any $z \geq 0$, we have $\overline{\pi}(z + x)/\overline{\pi}(x) \leq 1$. By the assumption $(A)$, $\overline{\pi}(z + x)/\overline{\pi}(x)$ converges, for any fixed $z \geq 0$, to $e^{-\alpha z}$ when $x$ goes to $+\infty$. By dominated convergence, we conclude that $\overline{\pi_H}(x)/\overline{\pi}(x)$ converges to $1/(\alpha + \kappa_Y)$ when $x$ goes to $+\infty$, so 
\[ \overline{\pi_H}(x) \underset{x \rightarrow + \infty}{\sim} \frac1{\alpha + \kappa_Y} \overline{\pi}(x). \]
As a consequence, $\pi_H$ satisfies $(A)$ with the same $\alpha$. We can thus apply Theorem 2.5 of \cite{pardo2013} and we get $\mathbb{E}[I(H)^{-\alpha}] < +\infty$ together with the following equivalent for $m_H$, the density of the exponential functional $I(H)$. 
\[ m_H(x) \underset{x \rightarrow 0}{\sim} \mathbb{E}[I(H)^{-\alpha}] \ \overline{\pi_H}(\log (1/x)) \underset{x \rightarrow 0}{\sim}  \frac{\mathbb{E}[I(H)^{-\alpha}]}{\alpha + \kappa_Y} \ \overline{\pi}(\log (1/x)). \]

Now, let us fix $\epsilon > 0$. Thanks to the previous equivalence, there exists $y_0 \in ]0, 1[$ such that 
\begin{eqnarray}
\forall y \in ]0, y_0[, \ (1-\epsilon) \frac{\mathbb{E}[I(H)^{-\alpha}]}{\alpha + \kappa_Y} \ \overline{\pi}(\log (1/y)) \leq m_H(y) \leq (1+\epsilon) \frac{\mathbb{E}[I(H)^{-\alpha}]}{\alpha + \kappa_Y} \ \overline{\pi}(\log (1/y)). \label{tailexpoladder1}
\end{eqnarray}
According to Corollary 2.1(2) in \cite{ref6patie}, the density $m_Y$ of $I(Y)$ is given by the following expression for $x > 0$. 
\begin{align}
m_Y(x) & = \frac{x^{-\kappa_Y - 1}}{\Gamma(\kappa_Y)} \int_{0}^{+\infty} e^{-y/x} y^{\kappa_Y} m_H(y) dy \geq \frac{x^{-\kappa_Y - 1}}{\Gamma(\kappa_Y)} \int_{0}^{y_0} e^{-y/x} y^{\kappa_Y} m_H(y) dy \nonumber \\
& \geq (1-\epsilon) \frac{\mathbb{E}[I(H)^{-\alpha}] x^{-\kappa_Y - 1}}{(\alpha + \kappa_Y) \Gamma(\kappa_Y)} \int_{0}^{y_0} e^{-y/x} y^{\kappa_Y} \overline{\pi}(\log (1/y)) dy \nonumber \\
& = (1-\epsilon) \frac{\mathbb{E}[I(H)^{-\alpha}] x^{-\kappa_Y - 1}}{(\alpha + \kappa_Y) \Gamma(\kappa_Y)} \overline{\pi}(\log (1/x)) \int_{0}^{y_0} e^{-y/x} y^{\kappa_Y} \frac{\overline{\pi}(\log (1/y))}{\overline{\pi}(\log (1/x))} dy \nonumber \\
& = (1-\epsilon) \frac{\mathbb{E}[I(H)^{-\alpha}]}{(\alpha + \kappa_Y) \Gamma(\kappa_Y)} \overline{\pi}(\log (1/x)) \int_{0}^{+\infty} e^{-u} u^{\kappa_Y} \frac{\overline{\pi}(\log (1/ux))}{\overline{\pi}(\log (1/x))} \mathds{1}_{u < y_0 / x} du, \label{simpldens1}
\end{align}
where we have used the lower bound of \eqref{tailexpoladder1} and made the change of variable $u = y/x$. Note that $(A)$ implies that $\overline{\pi}(\log (.))$ has $-\alpha$-regular variation so, according to Karamata representation theorem, 
\begin{eqnarray}
\overline{\pi}(\log (z)) = z^{-\alpha} \ \exp \left ( c (z) + \int_{1}^{z} \frac{r(t)}{t} dt \right ). \label{karamata2}
\end{eqnarray}
In this representation $c$ is a bounded measurable function that converges to a constant $c$ at $+\infty$, and $r$ is a bounded measurable function that converges to $0$ at $+\infty$. By decreasing $y_0$ if necessary, we can assume that for all $z \geq 1/y_0$: $c - \epsilon \leq c (z) \leq c + \epsilon$ and $|r(z)| \leq \epsilon$. As a consequence, for all $x \leq y_0$ and $u \in ]0, y_0/x[$ we have 
\begin{align*}
\frac{\overline{\pi}(\log (1/ux))}{\overline{\pi}(\log (1/x))} & \geq u^{\alpha} \exp \left ( c (1/ux) - c (1/x) + \int_{1/x}^{1/ux} \frac{r(t)}{t} dt \right ) \\
& \geq u^{\alpha} \exp \left ( -2 \epsilon - \epsilon |\log(u)| \right ) = u^{\alpha} (u \wedge 1/u)^{\epsilon} e^{-2\epsilon}.
\end{align*}
In the last inequality we have used the fact that, since $1/x$ and $1/xu$ are greater than $1/y_0$ we have $- 2 \epsilon \leq c (1/ux) - c (1/x) \leq 2 \epsilon$ and $\epsilon |\log(u)| = \epsilon |\log(1/ux) - \log(1/x)| \geq \int_{1/x}^{1/ux} \frac{r(t)}{t} dt \geq - \epsilon |\log(1/ux) - \log(1/x)| = - \epsilon |\log(u)|$. Putting into \eqref{simpldens1} we get for all $x \leq y_0$, 
\[ m_Y(x) \geq (1-\epsilon) e^{-2\epsilon} \frac{\mathbb{E}[I(H)^{-\alpha}]}{(\alpha + \kappa_Y) \Gamma(\kappa_Y)} \overline{\pi}(\log (1/x)) \int_{0}^{+\infty} e^{-u} u^{\alpha + \kappa_Y} (u \wedge 1/u)^{\epsilon} \mathds{1}_{u < y_0 / x} du. \]
By monotone convergence in the integral of the right hand side we obtain 
\[ \underset{x \rightarrow 0}{\liminf} \frac{m_Y(x)}{\overline{\pi}(\log (1/x))} \geq (1-\epsilon) e^{-2\epsilon} \frac{\mathbb{E}[I(H)^{-\alpha}]}{(\alpha + \kappa_Y) \Gamma(\kappa_Y)} \int_{0}^{+\infty} e^{-u} u^{\alpha + \kappa_Y} (u \wedge 1/u)^{\epsilon} du. \]
Letting $\epsilon$ go to $0$ we obtain, still by monotone convergence, 
\begin{align}
\underset{x \rightarrow 0}{\liminf} \frac{m_Y(x)}{\overline{\pi}(\log (1/x))} & \geq \frac{\mathbb{E}[I(H)^{-\alpha}]}{(\alpha + \kappa_Y) \Gamma(\kappa_Y)} \int_{0}^{+\infty} e^{-u} u^{\alpha + \kappa_Y} du \nonumber \\
& = \frac{\mathbb{E}[I(H)^{-\alpha}] \Gamma(\alpha + \kappa_Y + 1)}{(\alpha + \kappa_Y) \Gamma(\kappa_Y)} = \frac{\mathbb{E}[I(H)^{-\alpha}] \Gamma(\alpha + \kappa_Y)}{\Gamma(\kappa_Y)}. \label{simpldens2}
\end{align}

For the upper bound, using again the expression of $m_Y$ from Corollary 2.1(2) in \cite{ref6patie} we have, for $x > 0$, 
\begin{eqnarray}
m_Y(x) = \frac{x^{-\kappa_Y - 1}}{\Gamma(\kappa_Y)} \int_{0}^{y_0} e^{-y/x} y^{\kappa_Y} m_H(y) dy + \frac{x^{-\kappa_Y - 1}}{\Gamma(\kappa_Y)} \int_{y_0}^{+\infty} e^{-y/x} y^{\kappa_Y} m_H(y) dy. \label{simpldens3}
\end{eqnarray}
If $x < y_0/\kappa_Y$, then $(y \mapsto e^{-y/x} y^{\kappa_Y})$ is decreasing on $[y_0, +\infty[$. As a consequence, if $x$ is small enough the second term of \eqref{simpldens3} is less than 
\[ \frac{x^{-\kappa_Y - 1}}{\Gamma(\kappa_Y)} e^{-y_0/x} y_0^{\kappa_Y} \int_{y_0}^{+\infty} m_H(y) dy. \]
For the first term of \eqref{simpldens3} we can proceed as we did for the lower bound of $m_Y(x)$ and get that it is less than 
\[ (1+\epsilon) e^{2\epsilon} \frac{\mathbb{E}[I(H)^{-\alpha}]}{(\alpha + \kappa_Y) \Gamma(\kappa_Y)} \overline{\pi}(\log (1/x)) \int_{0}^{+\infty} e^{-u} u^{\alpha + \kappa_Y} (u \vee 1/u)^{\epsilon} du. \]
Then, since $\overline{\pi}(\log(.))$ has $-\alpha$-regular variation we have $x^{\eta} = o(\overline{\pi}(\log (1/x)))$ for any $\eta > \alpha$. As a consequence the upper bound of the second term of \eqref{simpldens3} is negligible compared to the upper bound of the first term. We thus get 
\[ \underset{x \rightarrow 0}{\limsup} \frac{m_Y(x)}{\overline{\pi}(\log (1/x))} \leq (1+\epsilon) e^{2\epsilon} \frac{\mathbb{E}[I(H)^{-\alpha}]}{(\alpha + \kappa_Y) \Gamma(\kappa_Y)} \int_{0}^{+\infty} e^{-u} u^{\alpha + \kappa_Y} (u \vee 1/u)^{\epsilon} du. \]
Here again, we can let $\epsilon$ go to $0$ and obtain the same value as in \eqref{simpldens2}. Thus, combining with \eqref{simpldens2} we get 
{
\[ \underset{x \rightarrow 0}{\lim} \frac{m_Y(x)}{\overline{\pi}(\log (1/x))} = \frac{\mathbb{E}[I(H)^{-\alpha}] \Gamma(\alpha + \kappa_Y)}{\Gamma(\kappa_Y)}. \]
Now, note that the term $\Gamma(\alpha + \kappa_Y)/\Gamma(\kappa_Y)$ equals $\mathbb{E}[G_{\kappa_Y}^{\alpha}]$, where $G_{\kappa_Y}$ denotes a Gamma random variable of parameter $\kappa_Y$. Recall also that $\mathbb{E}[I(H)^{-\alpha}] < +\infty$. Putting in relation with the decomposition $I(Y) = I(H) \times G_{\kappa_Y}^{-1}$, from Corollary 2.1(1) in \cite{ref6patie}, we deduce that $\mathbb{E}[I(Y)^{-\alpha}] < +\infty$ and that the limit $\mathbb{E}[I(H)^{-\alpha}] \Gamma(\alpha + \kappa_Y)/\Gamma(\kappa_Y)$ equals $\mathbb{E}[I(Y)^{-\alpha}]$, which yields the result. Note that $\mathbb{E}[I(Y)^{-\alpha}] < +\infty$ is also a consequence of the fact that $Y$ is spectrally positive together with Theorem 3 of \cite{Bertoinyor}.}

{
\begin{remarque} 
We have used Corollary 2.1(2) from \cite{ref6patie}, which is true in the spectrally positive case, in order to relate the density $m_Y$ of $I(Y)$ to the density $m_H$ of $I(H)$. Note that even for a more general $Y$ the asymptotics of the distribution functions of $I(Y)$ and $I(H)$ can be directly related thanks to Theorem 6 of \cite{aristarivero}. 
\end{remarque} 
}

\textbf{Acknowledgements: }

We are grateful to two anonymous referees for their helpful comments that allowed us to considerably improve the paper. In particular, we thank one of the referees for the simpler and nicer proof he or she suggested to show the existence of exponential moments for $I(V^{\uparrow})$. 

\bibliographystyle{plain}
\bibliography{thbiblio}

\end{document}